\documentclass[11pt]{amsart}
\usepackage{amsmath,amsfonts,amsthm,amssymb,amscd,hyperref,url}
\def\classification#1{\def\@class{#1}}
\classification{\null}

\DeclareFontFamily{OT1}{rsfs}{}
\DeclareFontShape{OT1}{rsfs}{n}{it}{<-> rsfs10}{}
\DeclareMathAlphabet{\mathscr}{OT1}{rsfs}{n}{it}

\DeclareMathOperator{\Stab}{Stab}

\DeclareMathOperator{\ad}{ad}

\DeclareMathOperator{\supp}{supp}

\DeclareMathOperator{\mo}{\,mod}

\DeclareMathOperator{\GL}{GL}
\DeclareMathOperator{\SO}{SO}
\DeclareMathOperator{\diam}{diam}

\DeclareMathOperator{\SL}{SL}

\DeclareMathOperator{\PSL}{PSL}

\DeclareMathOperator{\Sym}{Sym}
\DeclareMathOperator{\Prob}{Prob}
\DeclareMathOperator{\Alt}{Alt}
\DeclareMathOperator{\Sp}{Sp}

\DeclareMathOperator{\tr}{tr}

\DeclareMathOperator{\Cl}{Cl}
\DeclareMathOperator{\SU}{SU}

\newtheorem{prop}{Proposition}[section]
\newtheorem{thm}[prop]{Theorem}

\newtheorem{conj}{Conjecture}

\newtheorem*{Cha}{Challenge}
\newtheorem{cor}[prop]{Corollary}
\newtheorem{lem}[prop]{Lemma}
\newtheorem{defn}{Definition}

\newenvironment{Rem}{{\bf Remark.}}{}
\theoremstyle{remark}

\numberwithin{equation}{section}

\begin{document}
\title{Growth in groups: ideas and perspectives}
\author{H. A. Helfgott}
\address{Harald A. Helfgott, 
\'Ecole Normale Sup\'erieure, D\'epartement de Math\'ematiques, 45 rue d'Ulm, F-75230 Paris, France}
\email{harald.helfgott@ens.fr}
\subjclass[2010]{Primary: 20F69; Secondary: 20D60, 11B30, 20B05}

\begin{abstract}
This is a survey of methods developed in the last few years
to prove results on growth
in non-commutative groups. These techniques have their roots in both
additive combinatorics and group theory, as well as other fields. We discuss linear algebraic
groups, with $\SL_2(\mathbb{Z}/p\mathbb{Z})$ as the basic
example, as well as permutation groups. The emphasis will lie on the
ideas behind the methods.
\end{abstract}

\maketitle

\begin{center}
{\em In memory of \'{A}kos Seress (1958--2013)}
\end{center}

\section{Introduction}

The study of growth and expansion in infinite families of groups has undergone
remarkable developments in the last decade. We will see some basic themes at 
work in different contexts -- linear groups, permutation groups -- which have
seen the conjunction of a variety of ideas -- 
combinatorial, geometric, probabilistic.
The effect 
has been a still ongoing
 series of results that are far stronger and more general than those known before.

\subsection{What do we mean by ``growth''?}\label{subs:whatisg}

Let $A$ be a finite subset of a group $G$. Consider the sets
\[\begin{aligned}
A&,\\
A\cdot A &= \{x\cdot y : x,y\in A\},\\
A\cdot A\cdot A &= \{x\cdot y\cdot z: x,y,z\in A\},\\
&\dotsc\\
A^k &= \{x_1 x_2 \dotsc x_k : x_i\in A\}.
\end{aligned}\]
Write $|S|$ for the {\em size} of a finite set $S$, meaning simply the number
of elements of $S$. A question arises naturally: how does $|A^k|$ grow as $k$
grows?

This is just one of a web of interrelated questions that, until recently, were
studied within separate areas of mathematics:
\begin{itemize}
\item {\em Additive combinatorics} treats the case of $G$ abelian,
\item {\em Geometric group theory} studies $|A^k|$ as $k\to \infty$ for $G$
infinite,
\item {\em Subgroup classification}, within group theory, can be said to
study the special case $|A| = |A\cdot A|$, since, for $A$ finite with $e\in A$,
we have $|A| = |A\cdot A|$ precisely when $A$ is a subgroup of $G$.
(As we shall later see, this is not a joke, or rather it is a very fruitful
one.)
\end{itemize}

When $G$ is finite, rather than asking ourselves how $|A^k|$ behaves for 
$k\to \infty$ -- it obviously becomes constant as soon as $A^k = \langle A\rangle$, where $\langle A\rangle$ is the subgroup of $G$ generated by $A$ -- we ask
what is the least $k$ such that $A^k = G$. (Here and henceforth we 
assume that $A$ does generate $G$, and that $e\in A$.) 
This value of $k$ is called the
{\em diameter} of $G$ with respect to $A$.
The word ``diameter'' comes from graph theory; as we shall see,
a pair $(G,A)$ has a graph associated to it, and the diameter of the graph
is the same as the diameter we have just defined.

Other key related concepts are:
\begin{itemize}
\item {\em expander graphs}: every expander graph has very small diameter (where
{\em very small} $= O(\log |G|)$), but not vice versa -- that is, expansion is stronger than having very small diameter;
\item {\em mixing time}, meaning the least $k$ such that, when $x_1, \dotsc, x_k$
are taken uniformly and at random within $A$, the distribution of the product
$x_1\dotsb x_k$ is close to uniform. 
\end{itemize}
We will make both notions precise later.

This survey covers a series of rapid developments on these questions since 2005.
The focus will be on $G$ finite. Some of the results here extend to infinite 
groups with ease; indeed, working with, say, matrix groups over $\mathbb{Z}/p
\mathbb{Z}$ is often strictly harder than working over matrix groups over 
$\mathbb{R}$ or $\mathbb{C}$. The emphasis will be on methods -- often
with
a distinct combinatorial flavor -- that made strong, general results on finite
groups possible.

Several main themes run through proofs that seem very different on
the technical surface. One of them is the idea of {\em stable configurations
under group actions} as the main object of study. It is this, with 
slowly growing sets ({\em approximate subgroups}) as a special case, that
encapsulates not just the results but a great deal of the approach to proving 
them.

Contributions have been made by mathematicians from many fields. 
While one cannot claim or attempt completeness, I have tried to give as 
full a historical picture
 as I have been able to; 
this is so in part to give credit to all those
to whom credit is due, but also so as to familiarize readers who are experts
in a field with work coming from other domains,
and most of all so as to make both the ideas themselves and their actual
development clear.

\subsection{Main results covered.}

Our focus will lie on the ideas behind the main growth results, 
going from
$\SL_2(\mathbb{Z}/p\mathbb{Z})$ (\cite{Hel08}, reexamined in the light of
 \cite{HeSL3}, \cite{BGT}, \cite{PS}) up to new work on the symmetric group
\cite{MR3152942}. We will be able to give essentially complete proofs of Theorems
\ref{thm:main08} to \ref{thm:hs} below.

In particular, this includes Thm.~\ref{thm:bg} (Bourgain-Gamburd), the
result that set a paradigm on how to prove expansion starting from a result
on growth in groups.  We will spend less time on the applications of
expansion; while that is undoubtedly a rich topic, 
it is one whose developments to date
have been well covered in \cite{KowBourbaki} and \cite{MR2869010}.
The notes \cite{Konline} and \cite{Tonline} treat some of the material in the
present survey as well as some applications.

\subsubsection{Growth in linear groups.}
One of our goals will be to show the main ideas behind proofs of growth
in linear groups of bounded rank. In particular, we will give most of
the details of what amounts to an ``up-to-date'' proof of the following 
result, in such a way that the proof generalizes naturally. (In other words,
it will incorporate ideas from the series of developments to which the 
first proof gave rise.)

\begin{thm}[Helfgott \cite{Hel08}]\label{thm:main08}
Let $G = \SL_2(\mathbb{F}_p)$. Let $A\subset G$ generate $G$. Then either
\[|A^3|\geq |A|^{1+\delta}\]
or $(A \cup A^{-1}\cup \{e\})^k = G$, where $\delta>0$ and 
$k\geq 1$ are absolute constants.
\end{thm}
Here, as usual, ``absolute constant'' means ``really a constant'';
in particular, the constants $\delta$ and $k$ do not depend on $p$.
Nikolov-Pyber \cite{MR2800484} (following Gowers \cite{MR2410393};
see also \cite{BNP}) showed one 
can replace $(A \cup A^{-1}\cup \{e\})^k$ by $A^3$.
 Kowalski \cite{MR3144176} has shown one can take $\delta=1/3024$
(assuming $A=A^{-1}$, $e\in A$; we will see in \S \ref{sec:ruzs} that
these are very ``light'' assumptions). On the other end, Button and
Roney-Dougal \cite{BRD} have given an example that shows that
$\delta \leq (\log_2 7 - 1)/6 = 0.30122\dotsc$.

There are two main
kinds of generalizations: changing the field and changing the
Lie type. We will discuss them at the end of \S \ref{subs:pivoro}. In brief
-- Thm. \ref{thm:main08} is now known to hold (with $\delta$ depending on
the rank) for all finite simple groups of Lie type (\cite{PS}, \cite{BGT});
there are also generalizations to $\mathbb{C}$ \cite{BGSU2} and to
solvable groups \cite{GH2}.

\begin{cor}\label{cor:mesmar}
The diameter of
$G_p = \SL_2(\mathbb{Z}/p \mathbb{Z})$ with respect to any set of generators
$A$ is $(\log |G_p|)^{O(1)}$. 
\end{cor}
The proof of this Corollary to Thm.~\ref{thm:main08} is really a simple 
exercise; this is a good point for the reader to grab a nearby pencil.
\begin{proof}
Applying Thm.~\ref{thm:main08}
$\ell$ times, we obtain that either 
\begin{equation}\label{eq:josro}
|A^{3^\ell}|\geq |A|^{(1+\delta)^\ell}\end{equation}
or 
\begin{equation}\label{eq:korso}
(S \cup S^{-1} \cup \{e\})^k = G_p\end{equation}
for $S = A^{3^{\ell-1}}$. We choose 
$\ell = \lfloor (\log (\log |G_p|/\log |A|))/\log (1+\delta) + 1\rfloor$,
which, if (\ref{eq:josro}) were true, would give us $|A^{3^\ell}|>|G|$; since
this is absurd,
(\ref{eq:korso}) holds, and so the diameter of
$G_p$ with respect to $A\cup A^{-1}$ is 
at most 
\begin{equation}\label{eq:koloko}
 \ll k (\log |G_p|/\log |A|)^{\frac{1}{\log 1+\delta}} \ll
(\log |G_p|)^{O(1/\delta)}.\end{equation}
Now, a general result
 \cite[Thm.~1.4]{Babeul} states that the diameter
$d^+$ of a group $G$ with respect to a set of generators $A$ 
is at most $d^2 (\log |G|)^3$, where $d$ is the diameter of $G$
with respect to the set of generators $A \cup A^{-1}$. Thus,
(\ref{eq:koloko}) implies that the diameter of $G_p$ with respect to $A$ is 
\begin{equation}\label{eq:monro}
\ll (\log |G_p|)^{O(2/\delta) + 3}
= (\log |G_p|)^{O(1/\delta)}.\end{equation}
(We could have avoided the use of \cite[Thm.~1.4]{Babeul} by
applying the improved version of Thm.~\ref{thm:main08}, with
$A^3 = G$ instead of $(A\cup A^{-1} \cup \{e\})^k$, but why avoid
quoting an elegant general result that the reader could find
useful?) Since $\delta$ is an absolute constant and $|G_p|\geq 6$,
(\ref{eq:monro}) bounds the diameter by $(\log |G_p|)^{O(1)}$, as was desired.
\end{proof}

A conjecture due to Babai (see \cite[p.~176]{BS88}) states that, 
for any finite, non-abelian
simple group $G$ and any set of generators $A$, the diameter of 
$G$ with respect to $A$ is $O\left((\log |G|)^C\right)$, where $C$ is
an absolute constant. Thus, Cor. \ref{cor:mesmar}
states that Babai's conjecture holds for 
$G = \SL_2(\mathbb{Z}/p\mathbb{Z})$ 
and $A$ arbitrary. (We should actually say: ``for $G = 
\PSL_2(\mathbb{Z}/p\mathbb{Z})$'', since $\SL_2$ is just almost simple,
whereas $\PSL_2$ really is simple.) 

Consider now an important special case: let $A$ be a subset of
$\SL_2(\mathbb{Z})$, and suppose that
its projections $A_p = A \mo p$ generate $G_p$ for $p$ larger than a constant.
This is the case, in particular, if $\langle A\rangle$ is
Zariski-dense, i.e., if it is not contained in a proper subvariety of
$\SL_2$. (The implication is proven, in varying degrees of
generality, in \cite{MR763908}, \cite{MR735226}, \cite{MR880952} and
\cite{MR1329903}.) Both (\ref{eq:hosen}) and (\ref{eq:pantalon}) are
examples of valid $A$.

It is not hard to show (as we will sketch in \S \ref{subs:majaro}) that,
in this situation, Thm.~\ref{thm:main08} implies a logarithmic bound
($O_A(\log |G|)$) for the diameter of $G_p$ with respect to $A_p$.
Bourgain and Gamburd proved a rather stronger statement.

\begin{thm}[Bourgain-Gamburd \cite{MR2415383}]\label{thm:bg}
Let $A\subset \SL_2(\mathbb{Z})$ generate a Zariski-dense subgroup
of $\SL_2(\mathbb{Z})$. Then 
$(\SL_2(\mathbb{Z}/p\mathbb{Z}),A \mo p)$ are a family of expanders.
\end{thm}
In general, expansion is stronger than logarithmic mixing time, which 
is stronger than logarithmic diameter. Bourgain and Gamburd first show that 
a kind of mixing time (in a very weak
sense: an $\ell$ such that $|\mu^{(\ell)}|_2^2\ll |G|^{-1+\epsilon}$,
where $\mu^{(\ell)}$ is the distribution after $\ell$ steps of
a random walk) is indeed logarithmic: 
for the first $\log_c(p/2)$ steps, a random walk mixes well 
(for the same reason as what we said we will sketch in \S \ref{subs:majaro}); 
then, for a constant number of steps, they apply a
result on the $\ell_2$ ``flattening'' of measures under convolutions
(Prop.~\ref{prop:flatlem}) that they prove using Thm.~\ref{thm:main08}
(via a non-commutative version \cite{MR2501249}
of a result from additive combinatorics, the Balog-Szemer\'edi-Gowers 
theorem). The fact that expansion does follow from a logarithmic bound on 
this kind of weak mixing time for groups such as $\SL_2(\mathbb{Z}/p
\mathbb{Z})$ is due to Sarnak-Xue \cite{SarnakXue}. We will
go over this in detail in \S \ref{subs:majaro}.

Theorem \ref{thm:bg} has found manifold applications
(see, e.g., \cite{MR2587341} (the {\em affine sieve})); 
we refer again
to \cite{KowBourbaki} and \cite{MR2869010}. Traditionally,
 such results as we had on
expansion in $\SL_2(\mathbb{Z}/p\mathbb{Z})$ were deduced from results
on the spectral gap of the (continuous) Laplacian on the surface 
$\Gamma(p)\backslash \mathbb{H}$. Thm.~\ref{thm:bg} is a much more general
result, based on a combinatorial result, namely, Thm.~\ref{thm:main08}.

Later, Bourgain, Gamburd and Sarnak \cite{MR2892611} reversed the
traditional direction of the 
implication, showing that Thm.~\ref{thm:bg} can be used to obtain
spectral gaps for the Laplacian on general quotients 
$\Lambda \backslash \mathbb{H}$ ($\Lambda<\SL_2(\mathbb{Z})$ 
Zariski-dense in $\SL_2$).

By now, as we shall see in \S \ref{subs:pivoro}, Thms. \ref{thm:main08} and
\ref{thm:bg} have been generalized several times. There remain two open 
questions, both related to uniformity. One is whether 
$(\SL_2(\mathbb{Z}/p\mathbb{Z}),A_p)$ might be a family of expanders as $A_p$ 
ranges over all generating
sets $A_p\subset \SL_2(\mathbb{Z}/p\mathbb{Z})$. The other one concerns
uniformity on the {\em rank} of the group. Thm.~\ref{thm:main08}  is now known 
to hold over $\SL_n$ -- which is of rank $n-1$ -- but only if $\delta$ is 
allowed to depend on $n$. If $\delta>0$ is held constant and $n\to \infty$,
counterexamples are known. The question is then whether consequences such as 
Cor.~\ref{cor:mesmar} might still be true uniformly, i.e., with implied 
constants independent of $n$.


\subsubsection{The symmetric group and beyond}
The Classification of Finite Simple Groups tells us that every finite, simple,
non-abelian group is either a matrix group, or the alternating group 
$\Alt(n)$, or one of a finite list of exceptions. (The list is
irrelevant for asymptotic statements, precisely because it is finite.)
All of the above work on matrix groups leaves unanswered the corresponding
questions on diameter and growth in $\Alt(n)$ and other permutation groups.

The question of the diameter of permutation groups can be
stated precisely in a playful way.
Let a set $A$ of ways to scramble a finite set $\Omega$ be given.
This is the familiar setting of {\em permutation puzzles}: Rubik's cube, 
Alexander's star, Hungarian rings\dots .  People say that a position has
a solution if it can be unscrambled back to a fixed
`starting position' by means of some succession of moves in $A$. Given
that we are told that a position has a solution, does it follow that
it has a short solution?

The answer is {\em yes} \cite{MR3152942}. The only condition is that
$\langle A\rangle$ be {\em transitive}, i.e., that, given two elements
$x$, $y$ of $\Omega$, there be a succession of moves in $A$ that, when
combined, take $x$ to $y$.
Transitivity is necessary: it is easy
to construct a non-transitive group of very large diameter
\cite[Example 1.2]{BS92}. (However, if the number of orbits\footnote{An {\em
    orbit}, in a permutation group $G<\Sym(n)$, 
means an orbit of $\{1,2,\dotsc,n\}$ under the action of $G$. Thus,
Rubik's cube has three orbits: corners, sides and centers (if you
are allowed to rotate the cube in space).} is bounded,
then the problem reduces to the transitive case.) 


Thinking for a moment, we see that a request for short 
solutions is the same as a request for a small diameter: if $A = A^{-1}$,
the diameter of $\langle A\rangle$ with respect to $A$ equals the maximum, over
all positions, of the length of the shortest solution to that position.

As \cite{BS92} showed, questions on the diameter of permutation groups
reduce (with some loss) to the case $G=\Alt(n)$. Babai's
conjecture 
gives that the diameter of $\Alt(n)$ with respect to any set of generators
is $\ll (\log |G|)^{O(1)} \ll n^{O(1)}$.
This special case of the conjecture actually predates Babai; \cite[p. 232]{BS92} calls it ``folkloric''. There are earlier references in print \cite{KMS84},
\cite{McK} where the question is posed as to the exact conditions
under which a bound of $n^{O(1)}$ might be valid; in particular 
(\cite[\S 4]{KMS84}), might
transitivity be sufficient? For an arbitrary transitive subgroup of
$\Sym(n)$, Babai and Seress
\cite[Conj. 1.6]{BS92} conjectured the weaker bound
\begin{equation}\label{eq:kok}
\leq \exp((\log n)^{O(1)})\end{equation} on the diameter with respect
to any set of generators.

\begin{thm}[Helfgott-Seress \cite{MR3152942}]\label{thm:hs}
Let $G = \Sym(n)$ or $\Alt(n)$. Let $A$ be any set of generators of
$G$. Then
the diameter of $G$ with respect to $A$ is 
\[\leq \exp(O((\log n)^4 \log \log n)),\] where the implied constant is absolute.
\end{thm}
By a general result \cite{BS92} bounding the diameter of a transitive group in 
terms of the diameter of its largest factor of the form $\Alt(n)$,
Thm.~\ref{thm:hs} implies that (\ref{eq:kok}) holds for all transitive
permutation groups on $n$ elements and any set of generators.

What is the importance of $\Sym(n)$ and $\Alt(n)$, 
from the perspective of linear 
groups? It is not just a matter of historical importance (in that conjectures
for permutation groups preceded Babai's more general conjecture) or of 
generality. The groups $\Sym(n)$ and $\Alt(n)$ are, in a sense, creatures
of pure rank; $\Sym(n)$ corresponds particularly closely to what $\SL_n$
over a field with one element would be like.\footnote{While the field with
one element does not exist, objects over the field with one element can be 
defined and studied. This is an idea going back to Tits \cite{MR0108765}; see, e.g.,
\cite{MR2917136}.} Uniformity on the rank is precisely what
is still missing in the linear algebraic case; the new result on $\Sym(n)$ and
$\Alt(n)$ can be seen -- optimistically --  
as breaking the barrier of rank dependence, just as 
\cite{Hel08} showed that independence on the field was a feasible goal.

\begin{center}
* * *
\end{center}

In parallel to the work on permutation groups in the line of \cite{BLS87},
\cite{BBS04} et al. -- works having their roots in the study of algorithms --
there is also an entire related area of study coming from probability theory.
This area is well represented by the text \cite{MR2466937}; the emphasis there
is in part on mixing times for random processes that may be more general than a random walk. See, for example, results expressed in terms of card-shuffling, 
such as the Bayer-Diaconis ``seven-shuffle'' theorem \cite{MR1161056}.

The interest in studying the diameter and the expander properties of 
$\langle A\rangle$ with respect to $A$ for $A=\{g,h\}\subset \Sym(n)$,
$g$, $h$ {\em random}, comes
in part from this area. (This is also of interest for linear algebraic
groups; see \cite{MR2415383}.) Here a result of Babai-Hayes \cite{BH} based
on \cite{BBS04} shows that, almost certainly (i.e.,
with probability tending to $1$ as $n\to \infty$), the diameter of
$\langle A\rangle$ with respect to $A$ is polynomial in $n$. (A classical
result of Dixon \cite{dixon} states that $\langle A\rangle$ is almost certainly
$\Sym(n)$ or $\Alt(n)$.) Schlage-Puchta \cite{MR2965280} improved the bound to
$O(n^3 \log n)$.

Recent work \cite{HSZ} proves that the diameter
of $\langle A\rangle$ with respect to $A$ 
is in fact $n^2 (\log n)^{O(1)}$ with probability
tending to $1$ as $n\to \infty$; the mixing time is $n^3 (\log n)^{O(1)}$.
At play is a combination of probabilistic and combinatorial ideas 
(\cite{BroSha}, \cite{BBS04}); there is some common ground with the ideas in 
\cite{MR3152942} (discussed here in \S \ref{subs:genset}) on 
generation and random walks.

\subsection{Notation.}
By $f(n)\ll g(n)$, $g(n)\gg f(n)$ and $f(n) = O(g(n))$ we mean the 
same thing, namely, that there are $N>0$, $C>0$ such that $|f(n)|\leq C\cdot 
g(n)$ for all $n\geq N$. We write $\ll_a$, $\gg_a$, $O_a$ if 
$N$ and $C$ depend on $a$ (say). 

As usual, $f(n) = o(g(n))$ means that
$|f(n)|/g(n)$ tends to $0$ as $n\to \infty$.
We write $O^*(x)$ to mean any quantity at most $x$ in absolute value.
Thus, if $f(n)=O^*(g(n))$, then $f(n)=O(g(n))$ (with $N=1$ and $C=1$).

Given a subset $A\subset X$, we let
 $1_A:G\to \mathbb{C}$ be the characteristic function of
$A$:
\[1_A(x) = \begin{cases} 1 &\text{if $x\in A$,}\\ 0 &\text{otherwise.}
\end{cases}
\]
\subsection{Acknowledgements.}
The present paper was written in fulfillment of the conditions for the author's
Adams Prize (Cambridge). The author is very grateful to the Universidad
Aut\'onoma de Madrid 
for the opportunity to lecture on the subject during
his stay there as a visiting professor; special thanks are due to J.
Cilleruelo. Notes taken by students during his previous
series of lectures at the AQUA summer school (TIFR, Mumbai, 2010) were also
helpful, as was the opportunity to teach at the GANT school in EPFL (Lausanne),
and, in the last stages of preparation of the text, at the
 Analytic Number Theory workshop at IHES.
The author would like to thank the organizers in general and
E. Kowalski and Ph. Michel in particular.

A special place in the acknowledgements is merited by A. Granville, whose course and whose
notes \cite{Gra} introduced the author to additive combinatorics.
E. Kowalski, M. Rudnev, P. Varj\'u and E. Vlad
 must also be thanked profusely
for their close reading of large sections of the paper.
Many thanks are also due to L. Babai, F. Brumley, N. Gill, B. Green, 
L. Pyber, P. Spiga and an anonymous referee for their helpful comments.

\section{Cultural background}\label{subs:hinto}


The problems addressed in this survey first arose and were studied
 within different areas. The purpose of this section is to give a 
glimpse at a broad picture, or rather a series of glimpses from different
perspectives. For the sake of clarity, and to avoid repetitions 
with earlier and later discussions of recent results, we will look
at matters as they stood in different fields right before 
Theorem \ref{thm:main08} was proven, followed quickly by
Theorem \ref{thm:bg} and a series of results by many authors. 

This section is, so to speak, a ``cultural'' one, in that it presents matters quickly,
without any claim to completeness, with the aim of setting a scene; it could
in principle be skipped at first by a reader in a hurry. While it 
does define a key concept -- that of {\em expansion} -- and makes more
precise another one -- the notion of {\em mixing time} --  we will review
the definitions later, before their first technical use. What would be missed
by skipping this section would mainly be motivation, and a sense of the newness
of some of the interactions between the fields presented here as separate units.

\subsection{Additive combinatorics.} This is a relatively recent name
for a field one of whose starting points is the work of Freiman 
(\cite{MR0360496}; see also \cite{MR1139055})
classifying subsets $A$ of $\mathbb{Z}$ such that $A+A$ is not much larger
than $A$. (The strongest bound to date is that of Sanders 
\cite{MR2994508}; we still do not have the strongest conjectured form of
this kind of result.)
 Work by Ruzsa and others (\cite{MR0266892}, \cite{MR2314377}, 
\cite{MR810596}) showed how that, if $|A+A|$ is not much larger than $|A|$,
then the sequence $|A|, |A+A|, |A+A+A|,\dotsc$ grows slowly; we will go
over this in \S \ref{sec:ruzs}.

The use of the sign $+$ and the word {\em additive} both recall
us that additive combinatorics treated abelian groups, even if some of its 
techniques turned out 
to generalize to non-abelian groups rather naturally (see, e.g., 
\cite{MR2501249}). We will take a detailed look at some
of the bases of the subject in \S \ref{sec:arithcomb}.

\subsection{Expanders, spectral gaps and property (T)} Let us start by 
defining expanders. For $A$ a set of generators of a
finite group $G$, we say that the pair $(G,A)$ gives us an 
$\epsilon$-{\em expander} if 
every subset $S$ of $G$ with $|S|\leq |G|/2$ satisfies 
$|S \cup A S|\geq (1+\epsilon) |S|$. An $\epsilon$-expander always
has very small diameter (meaning $\text{diameter}=O((\log |G|)/\epsilon)$) 
and, if the identity $e$ is in $A$, also has
 very small $\ell_2$ and $\ell_\infty$ mixing time.
 
Alternatively, we can define the expander property in terms of the size
of the second smallest eigenvalue $\lambda_1$ of the discrete Laplacian of
the Cayley graph.
The (normalized) adjacency matrix $\mathscr{A}$ of a graph is the operator 
that maps a function $f$ on the set of vertices $V$ of the graph to the
function 
$\mathscr{A} f$
on $V$ whose value at $v$ is the average of $f(w)$ on the neighbors $w$ of
$v$. The {\em discrete Laplacian} is simply $\triangle = I-\mathscr{A}$.
A {\em Cayley graph}  $\Gamma(G,A)$ is the graph having $G$ as
its set of vertices and $\{(g,ag): g\in G, a\in A\}$ as its set
of edges. Assume $A = A^{-1}$; then $\triangle$ is a symmetric operator,
and all its eigenvalues $\lambda_0, \lambda_1,\dotsc$  are thus
real. Clearly, the smallest eigenvalue is $\lambda_0=0$, corresponding
to the constant eigenfunctions. We
say that $\Gamma(G,A)$ is an $\epsilon$-expander if $\lambda\geq \epsilon$
for every eigenvalue $\lambda$ corresponding to a non-constant eigenfunction.
For $|A|$ bounded, this is equivalent to the first definition of
$\epsilon$-expander we gave
 (though the constant 
$\epsilon$ is different in the two sides of the equivalence).

We can order the eigenvalues $0 = \lambda_0\leq \lambda_1\leq \lambda_2\leq \dotsc$.
We call $|\lambda_1-\lambda_0| = \lambda_1$ the {\em spectral gap}. 
We say that a family of graphs is an {\em expander family}, or has a 
spectral gap, if there is an $\epsilon>0$ such that, for every graph $\Gamma$ in the
family, $\Gamma$ is an $\epsilon$-expander.

Of course, the discrete Laplacian is an analogue of the classical Laplacian
$\Delta$, defined on surfaces. The Laplacian $\Delta$ on a surface also has
real, non-negative eigenvalues; we speak of a spectral gap if there is
an $\epsilon>0$ such that
$\lambda>\epsilon$ for every eigenvalue $\lambda$ corresponding to 
a non-constant eigenfunction of $\Delta$. An especially well-researched class of surfaces
is that given by quotients $G\backslash \mathbb{H}$ of the upper half plane 
$\mathbb{H} = \{x+ i y: y>0\}$ by subgroups $G$ of $\SL_2(\mathbb{R})$: 
the group $\SL_2(\mathbb{R})$ acts on
$\mathbb{H}$ by fractional linear transformations, and thus, given a discrete 
subgroup
$G$ of $\SL_2(\mathbb{R})$, we can study the surface $G\backslash \mathbb{H}$,
which is the result of identifying points $x_1$, $x_2$ of $\mathbb{H}$ for
which there is a $g\in G$ such that $g(x_1)=x_2$.  In particular, consider
the discrete subgroup $\Gamma(N)= \{g\in \SL_2(\mathbb{Z}): g\equiv I \mo N\}$.
A central result in the theory of modular
forms (Selberg, \cite{MR0182610}) states that the Laplacian on the surface
$\Gamma(N)\backslash \mathbb{H}$ has a spectral gap.
 It has long been known (see \S \ref{subs:majaro})
that the existence of this spectral gap on the surface
$\Gamma(N)\backslash \mathbb{H}$
implies that the pairs $(\SL_2(\mathbb{Z}/p\mathbb{Z}), A \mo p)$ with
\begin{equation}\label{eq:hosen}
A = \left\{ \left(\begin{array}{cc} 1 &1\\0 &1\end{array}\right),
\left(\begin{array}{cc} 1 &0\\1 &1\end{array}\right)
\right\}\end{equation}
are an expander family, i.e., are $\epsilon$-expanders for some
fixed $\epsilon>0$ as $p$ varies over all primes. Before \cite{Hel08} and \cite{MR2415383}, 
little was known for more general $A$; e.g., for
\begin{equation}\label{eq:pantalon}
A = \left\{ \left(\begin{array}{cc} 1 &3\\0 &1\end{array}\right),
\left(\begin{array}{cc} 1 &0\\3 &1\end{array}\right)
\right\},\end{equation}
there were no good diameter bounds, let alone a proof that
$(\SL_2(\mathbb{Z}/p\mathbb{Z}), A \mo p)$ is a family of expanders.
(This is a favorite example of Lubotzky's.)

For $G=\SL_n(\mathbb{Z}/p\mathbb{Z})$, $n\geq 3$, the proof of expansion
for some $A$ was arguably more direct (due to property (T), for which
relatively elementary proofs were known \cite{MR0209390}), but the case of
general $A$ was open, just as for $n=2$. 

Kassabov applied what was known for $\SL_n$ (and linear 
algebraic groups 
in general) to prove the existence of expanders for the symmetric group
\cite{MR2342639}. Other relevant works are \cite{SarnakXue}
(giving an elementary treatment of expansion for $A$ as in
(\ref{eq:hosen}), and, in general, 
``arithmetic lattices''\footnote{For more on the case of
arithmetic lattices (a case that covers
  (\ref{eq:hosen}) but not (\ref{eq:pantalon})), see the references in 
\cite[\S 1.1]{SGV}}
in $\SL_2$),
\cite{MR1645694}, \cite{MR1691549}, \cite{MR1900698} (strengthening
and generalizing \cite{SarnakXue} to some infinite-index groups), and 
\cite{MR2104475} and \cite{MR2231895} (both of them influenced by
the Solovay-Kitaev algorithm, as in \cite[App.~3]{MR1796805}; see
\cite{Varlat} for recent work along this line). Still, (\ref{eq:pantalon})
remained as a frustrating example of wide classes of generators for which
little was known.

\subsection{Mixing times and diameters.} Let $A$ be a set of generators of
a finite group $G$. Recall that the {\em diameter} of $G$ with respect to $A$
is defined to be 
the least $k$ such that every element $g$ of $G$ can be written as 
$g = x_1 x_2 \dotsc x_r$ for some $x_i\in A$, $r\leq k$; this is the same
as the graph-theoretical diameter of the Cayley graph $\Gamma(G,A)$.
As we said before, Babai's
conjecture \cite[p.~176]{BS88} posits that, if $G$ is simple and non-abelian,
the diameter of $G$ with respect to any set $A$ of generators is small, that is to say, $(\log |G|)^{O(1)}$. (Note that the bound does not depend
on $A$.) For the alternating group $G=\Alt(n)$,
this was a folklore conjecture; work towards it includes \cite{BS88},
\cite{MR1022771} and \cite{BBS04}.

The {\em mixing time} is the least $k$ such that,
when $x_1,\dotsc, x_k$ are taken uniformly and at random within $A$,
the distribution of the product $x_1 \dotsb x_k$ (or, if you prefer,
the outcome of a random walk of length $k$ on the Cayley graph
$\Gamma(G,A)$) is close to uniform in $G$. 
We speak
of {\em $\ell_2$ mixing time}, {\em $\ell_\infty$ mixing time}, etc., depending
on the norm used to define ``close to''. Here most work has focused
on permutation groups, with a strong probabilistic flavor: see 
\cite{BBS04}, \cite{BH}, \cite{MR1208801}, 
\cite{MR626813}, \cite{MR1245303} as well 
as the recent, rather thorough source \cite{MR2466937}, which contains
plenty of references.

\subsection{Group theory: subgroup classification.} 
If $e\in A$, the extreme case $|A\cdot A|=|A|$ happens exactly when
$A$ is a subgroup of $G$. There are results on subgroup
classification from the 80s and 90s intended to bypass parts of 
the Classification of Finite Simple Groups by elementary arguments. The
motivation was clear: the Classification Theorem had -- especially at first
-- an extremely long proof, dispersed in many articles. Thus, alternative
proofs for some of its consequences were in demand.

As we will later see, several of these alternative, relatively elementary
arguments -- in particular, those in  \cite{LP} (circulated in preprint
form since ca.\,1998.), \cite{Bab82} and \cite{Pyb93} -- later
played an important role in the study of growth: their techniques
for studying sets with $|A \cdot A|=|A|$ (that is, subgroups) turned out
to be robust enough to extend to the study of sets for which $|A \cdot A|$
is not much larger than $|A|$.

\subsection{Asymptotic group theory. Model theory.}


As we said in \S \ref{subs:whatisg}, geometric group theory tends to study
the growth of $|A^k|$ for $k\to \infty$, where $A$ is a set of generators of
an infinite group $G$. One of the main results of geometric group theory
is Gromov's theorem \cite{MR623534}, which states that, if $|A^k| \ll k^{O(1)}$,
then $G$ is virtually nilpotent, i.e., $G$ has a nilpotent subgroup of
finite index. In other words, Gromov's theorem classifies groups where $A^k$
grows very slowly (meaning polynomially).

{\em Model theory} is essentially a branch of logic with applications
to algebraic structures.
The work of
Hrushovski and his collaborators, culminating in \cite{MR2833482} 
(see also \cite{MR1329903} and \cite{MR2436141}), used model theory
 to study subgroups of algebraic groups. (As we will see later, this effort
was influenced by the work of Larsen-Pink \cite{LP}, and also served to
elucidate it.) In \cite{MR2833482}, Hrushovski gave a new proof of
Gromov's theorem.

We will later have the chance to discuss (\S \ref{subs:dime}) the 
recent influence of
\cite{MR1329903}, \cite{MR2436141} and \cite{MR2833482} on the area of this
survey, and on the study of growth in finite linear algebraic groups in
particular.

%


\section{Methodological background: arithmetic combinatorics}\label{sec:arithcomb}
The terms ``additive combinatorics'' and ``arithmetic combinatorics'' 
are relatively new. To judge from \cite{MR2289012}, they cover 
at least some of additive number theory and the geometry of numbers.
What may be called the core of additive combinatorics is the study of
the behavior of arbitrary sets under addition (as opposed to, say, 
the primes or $k$th powers). In this sense, the subject originated from
at least two streams, one coursing through work on arithmetic progressions
by Schur, van der Waerden, Roth \cite{MR0051853}, Szemer\'edi
\cite{MR0245555}, Furstenberg \cite{MR0498471} (leading to the ergodic work
of Host-Kra \cite{MR2150389} and Ziegler \cite{MR2257397}; see also
Szegedy \cite{Szege}),  
Gowers \cite{MR1844079}, and Green-Tao \cite{MR2415379}, among many others,
 and 
another one based on the study of growth in abelian groups, starting with Freiman 
\cite{MR0360496}, Erd\H{o}s-Szemer\'edi \cite{MR820223} and Ruzsa
\cite{MR1281447}, \cite{MR1701207} and continuing in \cite{MR2166359}, 
\cite{MR1909605}, \cite{MR2994508}. There
has also been a vein of a more geometrical flavor (e.g., \cite{MR729791}). 

The use of the term {\em arithmetic combinatorics} instead of {\em additive
combinatorics} emphasizes results on growth that do not require commutativity,
as well as results on fields and rings 
(the {\em sum-product theorem}, \S \ref{sec:sump}).

The reader who is playing hopscotch may jump ahead to section \S 
\ref{sec:sojor}, and thus miss the following material: the reason why 
Theorem \ref{thm:main08} is and can be stated as a lower bound on $|A^3|$ 
(\S \ref{sec:ruzs}), the connection of such a bound
to the definition of {\em approximate
subgroups} (also \S \ref{sec:ruzs}), the Balog-Szemer\'edi-Gowers theorem
 and its use by Bourgain and Gamburd 
to establish an equivalence between the growth of sets and the flattening
of measures (\S \ref{prop:flatlem}), and the relation between the sum-product
theorem and growth in the affine group (\S \ref{sec:okto}). In 
\S \ref{sec:solnil}, we will take a brief look at solvable and nilpotent groups
in general. 

Arithmetic combinatorics is important for the history
of the subject, and also provides part of its {\em raison d'\^{e}tre}; the study
of growth in non-commutative groups, after all, could be said to be just the
non-commutative case of arithmetic combinatorics. There is more to arithmetic
combinatorics and to growth in groups than that, but there is still some truth
to the statement, particularly in view of the genesis of the material we
aim to describe.

\subsection{Triple products and approximate subgroups}\label{sec:ruzs}
Some of additive combinatorics can be described as the study of {\em sets
that grow slowly}. In abelian groups, results are often stated so as
to classify sets $A$ such that $|A^2|$ is not much larger than $|A|$;
in non-abelian groups, works starting with \cite{Hel08} classify sets
$A$ such that $|A^3|$ is not much larger than $|A|$. 

There is a reason for this difference in conventions. In an abelian
group, if $|A^2| < K |A|$, then $|A^k| < K^{O(k)} |A|$ -- i.e., if a set
does not grow after one multiplication with itself, it will not grow
under many. This is a result of Pl\"unnecke \cite{MR0266892} and Ruzsa \cite{MR2314377}.
(Petridis \cite{MR3063158} recently gave a purely additive-combinatorial
proof.)
In a non-abelian group $G$, there can be sets $A$ breaking this rule:
for example, if $H\triangleleft G$, $g\in G\setminus H$ and $A = H
\cup \{g\}$, then $|A^2| < 3 |A|$, but $A^3\supset H g H$, and
$H g H$ can be much larger than $A$. (For instance, if $H$ is the subgroup
of $G = \SL_2(\mathbb{Z}/p\mathbb{Z})$ leaving a basis vector $e_1$ fixed, and
$w$ is the element of $G$ taking $e_1$ to $e_2$ and $e_2$ to $-e_1$, then
$H w H$ is of size $|H|^2$. We will later see (proof of Prop.~\ref{prop:port})
that this is not an isolated example --  
it can be quite useful to stick a subgroup $H$ in different directions (so to speak) in order 
to get a large product.)

However, Ruzsa's ideas do carry over to the non-abelian case, as was
pointed out in \cite{Hel08} and \cite{MR2501249}; in fact, \cite{MR810596}
carries over without change, since the assumption that $G$ is abelian
is never really used. We must 
assume that $|A^3|$ is small, not just $|A^2|$, and then it does follow
that $|A^k|$ is small.

\begin{lem}[Ruzsa triangle inequality]\label{lem:schatte}
Let $A$, $B$ and $C$ be finite subsets of a group $G$. Then
\begin{equation}\label{eq:eolt}
|A C^{-1}| |B| \leq |A B^{-1}| |B C^{-1}| .\end{equation}
\end{lem}
Commutativity is not needed. In fact, what is being used is in some
sense more basic than a group structure; as shown in \cite{GHR}, the same
argument works naturally in any abstract projective plane endowed with the
little Desargues axiom.
\begin{proof}
We will construct an injection
$\iota:A C^{-1} \times B \hookrightarrow A B^{-1} \times
 B C^{-1}$. For every 
$d\in A C^{-1}$, choose $(f_1(d),f_2(d)) = (a,c)\in A\times C$ such that
$d = a c^{-1}$. Define $\iota(d,b) = (f_1(d) b^{-1}, b (f_2(d))^{-1})$.
We can recover $d = f_1(d) (f_2(d))^{-1}$ from $\iota(d,b)$; hence
we can recover $(f_1,f_2)(d)=(a,c)$, and thus $b$ as well. Therefore,
 $\iota$ is an injection.
\end{proof}
It follows easily that 
\begin{equation}\label{eq:mony}
\frac {|(A \cup A^{-1} \cup \{e\})^3|}{|A|} \leq \left(3\frac{|A^3|}{|A|}\right)^3\end{equation}
for any finite subset $A$ of any group $G$, and, moreover,
\begin{equation}\label{eq:jotor}
\frac{|A^k|}{|A|} \leq \left(\frac{|A^3|}{|A|}\right)^{k-2}\end{equation}
for any $A\subset G$ such that $A = A^{-1}$ (i.e., $A$ contains
the inverse of every element in $A$). (Both of these statements go back
to Ruzsa (or Ruzsa-Turj\'anyi \cite{MR810596}), 
at least for $G$ abelian.) 
Again, this is the sort of statement for which it is nicer to give a
proof oneself than to read another person's proof. Let us do one case:
 $|A A^{-1} A| |A| = 
|A A^{-1} A| |A^{-1}|
\leq |A^2| |A^{-1} A^{-1} A|$ (by Lemma \ref{lem:schatte} with
$B=A^{-1}$ and $C=A^{-1} A$) and $|A^{-1} A^{-1} A| |A|
\leq |A^{-1} A^{-1} A^{-1}| |A^2| = |A^3| |A^2|$ (again by Lemma 
\ref{lem:schatte}), implying 
$|A A^{-1} A|/|A| \leq |A^2|^2 |A^3|/|A|^3 \leq (|A^3|/|A|)^3$. The 
rest of (\ref{eq:mony}) and (\ref{eq:jotor}) is left as an exercise.

Because of (\ref{eq:mony}) and (\ref{eq:jotor}), from now on,
we can generally focus on studying when $|A^3|$ is or isn't much larger
than $|A|$, assuming, without any essential loss of generality, that $A=A^{-1}$
and $e\in A$. Obviously, we can apply
(\ref{eq:jotor}) to $A \cup A^{-1} \cup \{e\}$ after applying
(\ref{eq:mony}). 

The paper \cite{MR2501249} focused on translating several results
from additive combinatorics to the non-abelian context.
In the course of this task, Tao defined
what he called an {\em approximate group}. ({\em Approximate subgroup}
might be more suggestive, as will become clear in \S \ref{sec:sojor}.)
A $K$-approximate subgroup of a group $G$ is a set $A\subset G$ such that
\begin{enumerate}
\item\label{eq:pennya} $A=A^{-1}$ and $e\in A$,
\item\label{eq:pennyb} there is a subset $X\subset G$ such that
$|X|\leq K$ and $A^2\subset X\cdot A$.
\end{enumerate}
This is essentially equivalent to the notion of a {\em slowly growing set}
$A$ (or {\em set of small tripling}), that is a set for which
 $|A^3|\leq K' |A|$: a $K$-approximate group
is a slowly growing set (trivially, with $K'=K^2$) and, for a slowly
growing set $A$ with $A = A^{-1}$ and $e\in A$, the set $A^3$ is 
a $K^{O(1)}$-approximate subgroup;
this was shown by Tao \cite[Cor. 3.11]{MR2501249}, with the
essential ingredient being the Ruzsa covering lemma (\cite{MR1701207}).

\begin{lem}[Ruzsa covering lemma]\label{lem:ruzsco}
Let $A$ and $B$ be finite subsets of a group $G$. Assume
$|A\cdot B|\leq K |B|$. Then there is a subset
$X\subset A$ with $|X|\leq K$ such that $A\subset X\cdot B\cdot B^{-1}$.
\end{lem}
\begin{proof}
Let $\{a_1, a_2, \dotsc, a_k\}$ be a maximal subset of $A$ with the
property that the cosets $a_j B$, $1\leq j\leq k$, are all disjoint.
It is clear that $k\leq |A\cdot B|/|B|\leq K$. Let $x\in A$. Since $\{a_1,
a_2, \dotsc, a_k\}$ is maximal, there is a $j$ such that $a_j B \cap
x B$ is non-empty. Then $x\in a_j B B^{-1}$. Thus, 
the sets $a_j B B^{-1}$ cover $A$.
\end{proof}

Tao also showed 
that one can classify sets $A$ of small {\em doubling} in terms of
approximate subgroups, using the covering lemma as one of the main tools:

\begin{lem}{\rm \cite[Cor. 4.7]{MR2501249}}\label{lem:corot}
Let $A$ be a finite subset of a group $G$. If $|A \cdot A|\leq K |A|$ or
$|A \cdot A^{-1}|\leq K |A|$, then $A$ lies in 
the union of at most $O(K^{O(1)})$ cosets of an $O(K^{O(1)})$ approximate
subgroup $H$ of size $|H|\ll K^{O(1)} |A|$.
\end{lem}

We will now see that the notion of sets of small growth, besides being 
essentially equivalent
to that of approximate subgroups, are also closely connected to another 
growth-related concept, this one based on measures.

\subsection{Balog-Szemer\'edi-Gowers. Flattening lemma (Bourgain-Gamburd).}
The Balog-Szemer\'edi-Gowers theorem is a key tool in modern additive 
combinatorics. As we shall see, it is also important in the study of growth
on groups, since it allows one to pass from results on sets to results on
measures.
The first version of the theorem was due to Balog and
Szemer\'edi \cite{MR1305895}. 
Gowers \cite[Prop. 12]{MR1631259} improved the bounds dramatically,
making all dependencies polynomial; this is needed for our applications.
Then Tao showed that the proof (which is essentially graph-theoretical,
as shown in \cite{MR2155059})
also works in a non-commutative setting \cite[\S 5]{MR2501249}.

First, we need a definition. Its commutative counterpart, the {\em additive energy}, is very common in additive combinatorics.
\begin{defn}
Let $G$ be a group. Let $A,B\subset G$ be finite sets. The 
{\em multiplicative energy} $E(A,B)$ is 
\[E(A,B) = \sum_{g\in G} |(1_A\ast 1_B)(g)|^2 = 
|\{(a_1,a_2,b_1,b_2)\in A\times A\times B\times B: a_1 b_1 = a_2 b_2\}|.\]
\end{defn}
Here, as usual, the convolution 
$f\ast g$ is defined by \[(f\ast g) = \sum_y f(y) g(y^{-1} x).\]
Clearly, $E(A,B)\leq \min(|A|^2 |B|, |A| |B|^2)$.
By Cauchy-Schwarz, we also have
\begin{equation}\label{eq:gocs}
E(A,B) \cdot |A B|\geq \left|1_A \ast 1_B\right|_1^2 = |A|^2 |B|^2.\end{equation}
In other words, if $|A B|$ is not much larger than $|A|$ or $|B|$, then $E(A,B)$
is large. It turns out that, while the na\"ive converse does not hold (in that, even if $E(A,B)$ is large, $|A B|$ can be much larger than $|A|$ and $|B|$), a converse statement of sorts is true.
\begin{prop}[Non-commutative  Balog-Szemer\'edi-Gowers
\cite{MR2501249}]\label{prop:bsg}
Let $G$ be a group. Let $A,B\subset G$ be finite. Suppose that $E(A,B)
\geq |A|^{3/2} |B|^{3/2}/K$. Then there are $A'\subset A$, $B'\subset B$ such
that $|A'|\gg |A|/K$, $|B'|\gg |B|/K$ and $|A'\cdot B'|\ll K^8
\sqrt{|A| |B|}$, where the implied constants are absolute.
\end{prop} 

The Balog-Szemer\'edi-Gowers theorem (for $G$ commutative) already
played a minor role in \cite{Hel08}; \cite{BGSU2} would later show 
this was not necessary. What concerns us most here is its use
for $G$ non-commutative in \cite{MR2415383}: Bourgain and Gamburd
showed how to use Prop.~\ref{prop:bsg} to reduce a statement on the
``flattening'' of measures to a statement about the growth of sets
(namely, Thm.~\ref{thm:main08}).

Again, this will be some sort of converse to an easier statement. The easier
statement here would be that, if $A$ is a set satisfying $|A^2|\leq K|A|$,
then  $|\mu\ast \mu|_2 \geq K^{-1} |\mu|_2$ for $\mu = (1/|A|) \mu_A$.
That is an exercise. (Hint: write $|\mu_A\ast \mu_A|_2^2 = E(A,A)/|A|^4$ and use (\ref{eq:gocs})
to bound $E(A,A)$.)

\begin{prop}[\cite{MR2415383}, ``flattening lemma'']\label{prop:flatlem}
Let $G$ be a finite group. Let $\mu$ be a probability measure on $G$
with $\mu(g) = \mu(g^{-1})$ for all $g\in G$. Suppose
that
\begin{equation}\label{eq:gorto}
|\mu\ast \mu|_2 \geq K^{-1} |\mu|_2\end{equation}
for some $K>0$.
Then there is a $K^{O(1)}$-approximate subgroup $H\subset G$ of size
$\ll K^{O(1)}/|\mu|_2^2$ and an element $g\in G$ such that
$\mu(H g) \gg K^{-O(1)}$. (The implied constants are absolute.)
\end{prop}
Note that $\mu(H g)\gg K^{-O(1)}$ implies $(\mu\ast \mu)(H^2) \geq
\mu(H g) \mu(g^{-1} H) \gg K^{-O(1)}$ (since $\mu(g)=\mu(g^{-1})$ and
$H = H^{-1}$).

We will give Bourgain and Gamburd's proof
with a technical simplification due to Tao \cite{Tonline}. Wigderson
seems to have suggested an analogous simplification (based on an idea
already in \cite{MR2272272}).
\begin{proof}
Just for expository purposes, 
we consider first the case of $\mu = (1/|A|) 1_A$, where $1_A$ is the 
characteristic function of a set $A\subset G$ (i.e., $1_A(g)=1$ if $g\in A$,
$1_A(g)=0$ if $g\notin A$). Then 
\begin{equation}\label{eq:gost}
|\mu\ast \mu|_2^2 = \frac{1}{|A|^4} E(A,A),\;\;\;\;\;
|\mu|_2^2 = \frac{1}{|A|}.\end{equation}
Thus, (\ref{eq:gorto}) means that $E(A,A) \geq K^{-2} |A|^3$. Hence,
by Prop.~\ref{prop:bsg}, there are $A_1', A_2'\subset A$ such that
$|A_1'|, |A_2'|\gg |A|/K^2$ and $|A_1' A_2'|\ll K^{18} \sqrt{|A_1'| |A_2'|}$.
By the Ruzsa triangle inequality (\ref{eq:eolt}), $|A_1' A_1'|\ll K^{36} |A_1'|$.
Thus, by Lem.~\ref{lem:corot},
 $A_1'$ lies in 
a union of $\ll K^{O(1)}$ cosets of an $O(K^{O(1)})$ approximate subgroup $H$
of size $\ll K^{O(1)} |A_1'| \leq K^{O(1)} |A|$. At least one of these cosets
$H g$ must contain $\gg K^{-O(1)} |A_1'|$ elements of $A_1'$, and thus of $A$. 
Hence $\mu(H g)\gg K^{-O(1)}$.

Now consider the case
 of general $\mu$.
The idea is that (thanks in part to (\ref{eq:gorto})) the bulk of $\mu$ is given by the values $\mu(g)$ neither
much larger nor much smaller than a certain value $a$; that ``bulk'' (call it
$\mu_\sim$) behaves essentially as a characteristic function, thus reducing
the situation to the one we have already considered.

Inspired by the second equation in (\ref{eq:gost}), we define 
$a=|\mu|_2^2$, and let $A$ be the set of all $g\in G$ with $\mu(g)
\geq a/(C K^c)$, where $c, C>0$ will be set later. We let $\mu_A
= (1/|A|) 1_A$; we must check that $|\mu_A\ast \mu_A|_2$ is large relative
to $|\mu_A|_2 = 1/\sqrt{|A|}$.

First, note that each $g\in A$ makes, by definition, a contribution of 
$\geq a^2/(C^2 K^{2c})$ to $|\mu|_2^2$; hence $|A|\leq C^2 K^{2 c}/a$, and so
$1/|A|\geq a/(C^2 K^{2 c})$.

We split $\mu = \mu_< + \mu_\sim + \mu_>$, where $\mu_<(g) = \mu(g)$ when
$\mu(g)<a/(C K^c)$ and $0$ otherwise, and $\mu_>(g) = \mu(g)$ when
$\mu(g)>C K^c a$ and $0$ otherwise. Now $|f\ast g|_2 \leq |f|_2 |g|_1$
for any $f$, $g$ (Young's inequality, special case; follows from Cauchy-Schwarz).
Hence 
\[\begin{aligned}|\mu \ast \mu_<|_2  &= |\mu_< \ast \mu|_2
\leq |\mu|_1 |\mu_<|_2 \leq 
1\cdot \sqrt{|\mu_<|_\infty |\mu_<|_1} \leq \frac{\sqrt{a}}{\sqrt{C K^c}},\\
|\mu_> \ast \mu|_2 &= |\mu \ast \mu_>|_2 
\leq |\mu_>|_1 |\mu|_2 \leq \frac{|\mu_>|_2^2}{\min_{g\in G} |\mu_>(g)|}
\sqrt{a} \leq \frac{a^{3/2}}{C K^c a} \leq \frac{\sqrt{a}}{\sqrt{C K^c}} 
.\end{aligned}\]
Thus, we can afford to cut off the tails: we obtain, by (\ref{eq:gorto}),
\begin{equation}\label{eq:sata}|\mu_\sim \ast \mu_\sim|_2 
\geq K^{-1} \sqrt{a} - \frac{4 \sqrt{a}}{\sqrt{C K^c}}
\geq \frac{1}{5} K^{-1} \sqrt{a},\end{equation} 
where we have set $C=5$, $c=2$. We are almost done; we now need to go from
$\mu_\sim$, which is roughly a characteristic function, to $\mu_A$, which
is actually a characteristic function.

The inequality (\ref{eq:sata}) enables us to bound
\[|\mu_A \ast \mu_A|_2 \geq 
\frac{1/|A|}{|\mu_\sim|_\infty} 
|\mu_\sim \ast \mu_\sim|_2 \geq 
\frac{1/|A|}{C K^c a} \cdot \frac{1}{5} K^{-1} \sqrt{a} = 
\frac{1}{25 K^3} \frac{1/|A|}{\sqrt{a}} .\]
By $1/|A|\geq a/(C^2 K^{2 c}) = a/(5^2 K^4)$ and
$|\mu_A \ast \mu_A|_2^2 = E(A,A)/|A|^4 \leq 1/|A|$, this implies both
\[|\mu_A \ast \mu_A|_2 \geq \frac{1/|\sqrt{A}|}{5^3 K^5} = 
\frac{|\mu_A|_2}{5 K^2}
\;\;\;\;\;\;
\text{and}\;\;\;\;\;\;
|A|\geq \frac{1/a}{5^4 K^6}.\] 

 We now have
the setup we had at the beginning, only with $\mu_A$ instead of $\mu$
and $5^3 K^5$ instead of $K$. Proceeding as before, we obtain a
$K^{O(1)}$-approximate subgroup $H\subset K$ 
such that $\mu_A(Hg)\gg K^{-O(1)}$ for some $g\in G$, and so
\[\mu(H g) \geq \frac{a}{5 K^2} 1_A(H g) =
\frac{a |A|}{5 K^2} \mu_A(H g) \gg K^{-O(1)}.\]
\end{proof}

\subsection{The sum-product theorem. Growth in solvable
  groups.}\label{sec:sump}
\subsubsection{The affine group and the sum-product theorem.}\label{sec:okto}
The analogue of the following lemma over $\mathbb{R}$ had been known for a
long while 
(Erd\H{o}s-Szemer\'edi \cite{MR820223}). The version over finite fields is harder, since
there is no natural topology or fully natural ordering to work with. 
(Over $\mathbb{R}$, 
there is a brief and very natural proof \cite{MR1472816} based on a result
that is essentially topological \cite{MR729791}; the best
known bound for the sum-product theorem over $\mathbb{R}$ has a direct proof,
also topological \cite{MR2538014}.)
\begin{thm}[Sum-product theorem \cite{MR2053599},
\cite{MR2225493}; see also \cite{MR1948103}]
\label{thm:orb}
For any $A\subset \mathbb{F}_p^*$
with $C < |A| <p^{1 - \epsilon}$, $\epsilon>0$, we have
\[\max(|A\cdot A|,|A+A|) > |A|^{1 + \delta},\]
where $C>0$ and $\delta>0$ depend only on $\epsilon$.
\end{thm}
The proof was strengthened and simplified in \cite{MR2289012} and
\cite{MR2359478}

The same result holds for $\mathbb{F}_q$, $q=p^a$, and indeed for arbitrary
fields; we must only be careful to specify that $A$ is not concentrated in a
proper subfield. The strength of this result must be underlined: $A$ is
growing by a factor of $|A|^\delta$, where $\delta>0$ is moreover
 independent of $p$. In contrast, even after impressive recent improvements
(\cite{MR2994508}; see also \cite{MR2738997}), the main additive-combinatorial 
result for abelian groups (Freiman's theorem) gives growth by smaller 
factors. 

Rather than prove Thm.~\ref{thm:orb}, let us prove the key intermediate
result towards it; it is enough for many applications, and it also illustrates
the connection between the sum-product theorem and growth in solvable groups.
The following idea was put forward in \cite[\S 3.1]{HeSL3} and developed 
there and in later works: the sum-product theorem is really a result
about the action of a group on another group; in its usual formulation
(Thm.~\ref{thm:orb}), the group that is acting is $\mathbb{F}_p^*$ 
(by multiplication),
and the group being acted upon is $\mathbb{F}_p^+$.

Let $G$ be the {\em affine group}
\begin{equation}\label{eq:jot}
G = \left\{\left(\begin{matrix} r & a\\0 &1\end{matrix}\right) :
r \in \mathbb{F}_p^*, a\in \mathbb{F}_p\right\}.\end{equation}
Consider the following subgroups of $G$:
\begin{equation}\label{eq:gutr}
U = \left\{\left(\begin{matrix} 1 & a\\0 &1\end{matrix}\right) :
a\in \mathbb{F}_p\right\},\;\;\;\;\;\;
T = \left\{\left(\begin{matrix} r & 0\\0 &1\end{matrix}\right) :
r\in \mathbb{F}_p^*\right\}.\end{equation}
These are simple examples of a {\em solvable} group $G$, of a maximal
{\em unipotent} subgroup $U$ and of a maximal torus $T$. Actually, the
centralizer $C(g)$ of any element $g$
of $G$ not in $\pm U$ is a maximal torus.

We look at two actions -- that of $U$ on itself (by the group operation)
and that of $T$ on $U$ (by conjugation; $U$ is a normal subgroup of $G$). They turn out to correspond to
addition and multiplication in $\mathbb{F}_p$, respectively:
\[\begin{aligned}
\left(\begin{matrix} 1 & a_1\\0 &1\end{matrix}\right) \cdot
\left(\begin{matrix} 1 & a_2\\0 &1\end{matrix}\right) &=
\left(\begin{matrix} 1 & a_1+a_2\\0 &1\end{matrix}\right)\\
\left(\begin{matrix} r & 0\\0 &1\end{matrix}\right) \cdot
\left(\begin{matrix} 1 & a\\0 &1\end{matrix}\right) \cdot
\left(\begin{matrix} r^{-1} & 0\\0 &1\end{matrix}\right) 
&= \left(\begin{matrix} 1 & r a\\0 &1\end{matrix}\right) .
 \end{aligned}\]

Thus, we see that growth in $U$ (under the actions of $U$ and $T$) 
is tightly linked to growth in $\mathbb{F}_p$ (under addition and 
multiplication). 
 
In fact, the result we will prove on these two actions (Prop.~\ref{prop:jutor}),
implies immediately the ``key intermediate result'' we want:
\begin{prop}[\cite{MR2359478}, Corollary 3.5]\label{prop:musoff}
Let $X\subset \mathbb{F}_p$, $Y \subset \mathbb{F}_p^*$ be given with
$X = - X$, $0\in X$, $1\in Y$. Then
\[|4 Y X + 2 Y^2 X|\geq \frac{1}{2} \min(|X| |Y|,p).\]
\end{prop}
We write $4 X$ (say) for $X^4$ when $X$ is a subset of an additive group;
thus, e.g., $2 Y^2 X = Y^2 X + Y^2 X$.

Thm.~\ref{thm:orb} follows from Prop.~\ref{prop:musoff}
after the application of the
Katz-Tao Lemma \cite[\S 2.8]{MR2289012}, which plays a role for sums
and products analogous to that played by Lemma \ref{lem:schatte} (Ruzsa)
for group operations.

\subsubsection{Solvable and nilpotent groups}\label{sec:solnil}
Before we actually start working on the affine group (\S \ref{subs:pivo}),
let us take a look at the general situation for solvable and nilpotent groups.
This will not be needed for our treatment of the affine group, but it will
be a chance to consider several intriguing questions, some of them going
outside our main framework.

N. Gill and the author \cite{GH2} proved 
growth in all solvable subgroups of $\GL_n(\mathbb{F}_p)$, in the sense
of Prop.\ \ref{prop:sabat}. The main two challenges were the existence
of elements outside $U$ that are not semisimple (and thus their action on $U$
has non-trivial fixed points) and the relatively complicated subgroup structure
that solvable subgroups of $\GL_n(\mathbb{F}_p)$ can have.
The case of solvable groups over $\mathbb{F}_q$ remains 
open; a proof along the lines of \cite{GH2} should be feasible but cumbersome.
(The case of $\GL_n(\mathbb{F}_q)$ does not reduce to $\GL_n(\mathbb{F}_p)$,
 since doing that in a naive way
would increase the rank $n$ depending on $q$, and we want
results independent on $p$ or $q$.) As usual in this context, infinite fields
can be easier if they have a ``sensible'' topology and/or if the subgroup
structure is simpler (\cite{MR2749571}, \cite{MR2825469}, \cite{MR2912036}).
The problem also becomes more accessible if, instead of aiming at bounds of the
quality $|A^3|\geq |A|^{1+\delta}$, we aim at weaker bounds 
\cite{MR2791295}, since then more tools are admissible.


Already in $G$ as in (\ref{eq:jot}), growth-related behavior can be
complex. We will show that subsets $A$ of $G$ do grow rapidly
under the group operation, outside some very specific circumstances.
However, the action of $G$ on $U$ does not, in general, give us 
expander graphs.
To be precise: identify $U$ with $\mathbb{F}_p$, fix a 
$\lambda\in \mathbb{Z}^+$, and say we have $\epsilon$-expansion
if, for every $S\subset \mathbb{F}_p$ with $|S|<p/2$,
\begin{equation}\label{eq:ojart}
|S\cup (S+1) \cup \lambda S| \geq (1+\epsilon) |S|.\end{equation}
(Here the addition of $1$ can be thought of as coming from the action
of $U$ on itself, and multiplication by $\lambda$ comes from the action of
$G/U$ on $U$ by conjugation.) 

In order to relate expansion in the sense of (\ref{eq:ojart}) to
eigenvalues of graphs, it is necessary to define Schreier graphs.
A {\em Schreier graph} is given by a set of actions
$A$ on a set $X$ (here $X=\mathbb{F}_p$); the set of vertices is $X$,
and the set of edges is $\{(x,a(x)): x\in X, a\in A\}$. (A Cayley
graph is thus a special case of a Schreier graph: $A\subset X$, $A$
acts on $X$ by multiplication.)
Now, the spectrum of the discrete Laplacian for the
 Schreier graph $\Gamma$ given by $x\mapsto x+1$, $x\to \lambda x$ is given
 in \cite{MR1780213}, which, in particular, shows that there is no constant 
spectral gap, i.e.,
$\lambda_1$ is {\em not} bounded away from $0$ as $p\to \infty$; this implies
that there is no fixed $\epsilon>0$ such that
(\ref{eq:ojart}) holds for all $S\subset \mathbb{F}_p$ with $|S|\geq p/2$
and all sufficiently large $p$.

J. Cilleruelo points out that one can prove this directly by modifying a 
construction
by G. Fiz Pontiveros \cite{Fizpont}, based in turn on an idea of Rokhlin's
\cite{MR0193206}: let $I$ be the reduction of $\{0\leq n\leq \epsilon p/3\}$
modulo $p$, and let $\phi:\mathbb{Z}/p\mathbb{Z} \to \mathbb{Z}/p\mathbb{Z}$
be the multiplication-by-$\lambda$ map; define 
\[S = \bigcup_{0\leq i \leq \frac{2}{\epsilon}} \phi^{-1}(I).\]
Then $|S|\sim p/3$ and $|S\cup (S+1) \cup \lambda S|
\leq (1+\epsilon) |S|$ for any $p$
larger than a constant depending only on $\epsilon$. 


Can one make the somewhat weaker assertion that the diameter of the 
Schreier graph $\Gamma$ is
small? For $\lambda\in \mathbb{Z}^+$, this is very easy: writing $x$ in base 
$\lambda$, we
obtain that the diameter if $O_\lambda(\log x)$. If we want a bound
independent of $\lambda$ for $\lambda\in \mathbb{F}_p^*$ arbitrary, 
the problem is subtler. (We must impose the condition that the order
of $\lambda$ not be too small, or else what we try to prove could be false.)
Using a result of Konyagin \cite{MR1289921}, it is possible to show that the 
diameter is $O((\log |G|)^{O(1)})$, where both implied constants are absolute.\footnote{This was the outcome
of a discussion among B. Bukh, A. Harper and the author. We thank E.
Lindenstrauss for referring us to Konyagin's paper.}

{\em Nilpotent groups.} The case of nilpotent groups is fairly close to
the special case of abelian groups. Here Fisher-Katz-Peng
\cite{MR2587441} and Tao \cite{MR2791295} laid out the groundwork;
the recent preprint \cite{Tointon} contains a general statement. 
In summary -- tools that yield Freiman's theorem over abelian groups
(also called Freiman-Ruzsa in that generality) can be adapted to work
over nilpotent groups. On the flip side, bounds are quantitatively weak
and necessarily conditional on the non-existence of significant structure
(progressions), just as in the abelian case.

Finally, let us take a brief look at the asymptotics of growth in infinite
solvable and nilpotent groups. Here some of the main results are
\cite{MR0248688}, \cite{MR0244899}, \cite{MR0379672}, \cite{MR0369608};
in summary, a set of generators $A$ of a
solvable group $G$ has polynomial growth (i.e. $|A^k|\ll |A|^{O(1)}$) if and
only if $G$ has a nilpotent subgroup of finite index. (When $G$ is not
assumed to be solvable, the ``only if'' direction becomes very hard; this is
due to Tits for linear groups (a consequence of the ``Tits alternative''
\cite{MR0286898}) and to Gromov for general groups (the celebrated
paper \cite{MR623534}).)

\section{Group actions: stabilizers, orbits and pivots}\label{sec:sojor}
\subsection{The orbit-stabilizer theorem for sets}\label{subs:joro}
A leitmotif recurs in recent work on growth in groups: results on subgroups
can often be generalized to subsets. This is especially the case if the
proofs are quantitative, constructive, or, as we shall later see,
probabilistic.

The {\em orbit-stabilizer theorem for sets} is an example both paradigmatic
and basic; it underlies a surprising number of other results on growth. 
It also helps to
put forward a case for seeing group actions, rather than groups themselves,
as the main object of study. We state it as in \cite[\S 3.1]{MR3152942}, though it is 
already implicit in \cite{Hel08} (and clear in \cite{HeSL3}).

We recall that an {\em action} $G\curvearrowright X$ is a homomorphism from a group $G$ to 
the group of automorphisms of a set $X$. (The automorphisms of a set $X$
are just the bijections from $X$ to $X$; we will see actions on objects with
richer structures later.)  For $A\subset G$ and $x\in X$, 
the {\em orbit} $A x$ is the
set $A x = \{g\cdot x : g\in A\}$. The {\em stabilizer} $\Stab(x)\subset G$ 
is given by $\Stab(x) = \{g\in G: g\cdot x = x\}$.

(Permutation group theorists prefer to use actions on the right; they
write $x^g$ for $g(x)$, $G_x$ for $\Stab(x)$, and use right cosets by default.
We will use that notation in \S \ref{sec:pergro}, where we will also
 write $x^A$ instead of $Ax$, in consequence.) 

\begin{lem}[Orbit-stabilizer theorem for sets]\label{lem:orbsta}
Let $G$ be a group acting on a set $X$. Let $x\in X$, and let $A\subseteq G$ be non-empty. 
Then
\begin{equation}\label{eq:applepie}
|(A^{-1} A) \cap \Stab(x)|\geq \frac{|A|}{|A x|}.\end{equation}
Moreover, for every $B\subseteq G$,
\begin{equation}\label{eq:easypie}
|B A| \geq |A \cap \Stab(x)| |B x| .
\end{equation}
\end{lem}
The usual orbit-stabilizer theorem is the special case $A = B = H$, $H$
a subgroup of $G$.
\begin{proof}[Sketch of proof]
Exercise: (\ref{eq:applepie}) is proven by pigeonhole, (\ref{eq:easypie})
by counting.
\end{proof}

Let $H$ be a subgroup of $G$.
The following lemmas are all direct consequences of the above for the
natural action $G\curvearrowright X=G/H$ defined by group multiplication. 
Lemma \ref{lem:musa} gives us elements in a subgroup of $G$;
Lemmas \ref{lem:dora}--\ref{lem:arana} tell us that, to obtain growth in a group, it is enough to obtain growth in 
a subgroup or in a quotient.
\begin{lem}\label{lem:musa} {\rm \cite[Lem. 7.2]{HeSL3}}
Let $G$ be a group and $H$ a subgroup thereof. Let $A\subset G$ be a 
non-empty set. Then
\begin{equation}\label{eq:vento}
|A A^{-1} \cap H| \geq \frac{|A|}{r},\end{equation}
where $r$ is the number of cosets of $H$ intersecting $A$. 
\end{lem}
\begin{lem} {\rm \cite[Lem. 3.5]{MR3152942}}\label{lem:dora}
Let $G$ be a group and $H$ a subgroup thereof. 
Let $A\subset G$ be 
a non-empty set with $A = A^{-1}$.
 Then, for any $k>0$,
\begin{equation}\label{eq:avoc1}
|A^{k+1}| \geq \frac{|A^{k}\cap H|}{|A^2\cap H|} |A| .
\end{equation}
\end{lem}
\begin{lem}{\rm \cite[Lem. 7.4]{HeSL3}}\label{lem:arana}
Let $G$ be a group and $H$ a subgroup thereof. Write $\pi_{G/H}:G\to G/H$
for the quotient map.
Let $A\subseteq G$ be a non-empty set with $A = A^{-1}$. Then, for any $k>0$,
\[|A^{k+2}| \geq \frac{|\pi_{G/H}(A^k)|}{|\pi_{G/H}(A)|} |A| .\]
\end{lem} 
\begin{proof}[Hints for Lemmas \ref{lem:musa}--\ref{lem:arana}] Let $G\curvearrowright X=G/H$ be the natural action by multiplication;
let $x\in X$ be the equivalence class of the identity (i.e., $H$).
For Lem.~\ref{lem:dora}, use first (\ref{eq:applepie}), then
(\ref{eq:easypie}) (with $A^k$ instead of $A$ and $A$ instead of $B$
for (\ref{eq:easypie})). For Lem.~\ref{lem:arana}, use first (\ref{eq:easypie})
(with $A^2$ instead of $A$ and $A^k$ instead of $B$), then (\ref{eq:applepie}).
\end{proof}

In the above, as is often the case, the assumption $A=A^{-1}$ is inessential
but convenient from the point of view of notation. (Obviously, if $A$  is
a set 
not fulfilling $A=A^{-1}$, we can apply the lemmas to $A\cup A^{-1}$ rather
than to $A$.)

As far as the orbit-stabilizer theorem (Lemma \ref{lem:orbsta}) is concerned,
 the action of $G$ on itself by multiplication is dull -- all stabilizers are
trivial. However, the action of $G$ on itself by conjugation is rather interesting. Write $C_H(g)$ for the centralizer
\[C_H(g) = \{h\in H: h g h^{-1} = g\}\]
and $\Cl(g)$ for the conjugacy class
\[\Cl(g) = \{h g h^{-1} : h\in G\}.\]
We can write $C(g)$ as short for $C_G(g)$.

\begin{lem}\label{lem:lawve}
 Let $A\subset G$ be a non-empty set with $A = A^{-1}$. Then,
for every $g\in A^l$, $l\geq 1$,
\[|A^2\cap C(g)|\geq \frac{|A|}{|A^{l+2}\cap \Cl(g)|}.\]
\end{lem}
\begin{proof}
Let $G\curvearrowright G$ be the action of $G$ on itself by conjugation. Apply 
(\ref{eq:applepie}) with $x=g$; the orbit of $g$ under conjugation by $A$ 
is contained in $A^{l+2}\cap \Cl(g)$
\end{proof}

\subsection{A pivoting argument: the affine group}\label{subs:pivo}

We will now see how to obtain growth in the affine group (\ref{eq:jot}).
The main ideas in the proof of Prop.~\ref{prop:jutor} below were extracted
in \cite{HeSL3} from the proof of the sum-product theorem in \cite{MR2359478}. 
In \cite[\S 3]{HeSL3} and then in \cite{GH2}, a similar strategy was shown
to work for more general solvable groups. The theme of {\em pivoting}
will recur in \S \ref{subs:pivoro}.

First, let us see how to construct many elements in $U$ and $T$ starting 
from $A$.
\begin{lem}\label{lem:cano}
Let $G$ be the affine group over $\mathbb{F}_p$ (\ref{eq:jot}). Let $U$
be the maximal unipotent subgroup of $G$, and $\pi:G\to G/U$ the quotient
map.

Let $A\subset G$, $A=A^{-1}$.
Assume $A\not\subset U$; let $x$ be an element of $A$ not in $U$.
 Then 
\begin{equation}\label{eq:hastar}
|A^2 \cap U| \geq \frac{|A|}{|\pi(A)|},\;\;\;\;\;\;
|A^2 \cap T| \geq \frac{|A|}{|A^5|} |\pi(A)|
\end{equation}
for $T=C(x)$.
\end{lem}
Recall $U$ is given by (\ref{eq:gutr}). Since $x\not\in U$, its centralizer
$T=C(x)$ is a maximal torus.
\begin{proof}
By Lemma \ref{lem:musa},
$A_u:= A^{-1} A \cap U$ has at least $|A|/|\pi(A)|$ elements.
Consider the action of $G$ on itself by conjugation. Then, by Lemma 
\ref{lem:orbsta}, $|A^{-1} A\cap \Stab(x)|\geq |A|/|A(x)|$. 
(Here $A(x)$ is the orbit of $x$ under the action (by conjugation) of
$A$.) We set
$A_t := (A^{-1} A) \cap \Stab(x) \subset T$. Clearly, 
$|A (x)| = |A(x) x^{-1}|$ and
$(A x) x^{-1}\subset A^4\cap U$, and so $|A(x)|\leq |A^4 \cap U|$.  At the same time, by (\ref{eq:easypie})
applied to the action $G\curvearrowright G/U$ by left multiplication, 
$|A^5| = |A^4 A|\geq |A^4 \cap U|\cdot |\pi(A)|$. Hence 
\[|A_t| \geq \frac{|A|}{|A^4\cap U|} \geq \frac{|A|}{|A^5|} |\pi(A)|.\] 
\end{proof}

As per previous notation, 
$A^2_t = A_t \cdot A_t$, $A_t(A_u) = 
\{t_1(u_1): t_1\in A_t, u_1\in A_u\}$ and $t(u) = t u t^{-1}$ (that is, $T$ acts on $U$ by conjugation) .

\begin{prop}\label{prop:jutor}
Let $G$ be the affine group over $\mathbb{F}_p$, $U$ the maximal unipotent
subgroup of $G$, and $T$ a maximal torus. Let $A_u\subset U$, $A_t\subset T$. 
Assume
$A_u = A_u^{-1}$, $e\in A_t, A_u$ and $A_u\nsubseteq \{e\}$. Then
\begin{equation}\label{eq:jces}|(A^2_t(A_u))^6|
\geq \frac{1}{2} \min(|A_u| |A_t|,p).\end{equation}
\end{prop}
To emphasize that what we prove really does imply Thm.~\ref{prop:musoff}
as stated, let us make clear that we will actually prove that
\[|A_t(A_u) A_t^2(A_u) A_t(A_u^2) A_t^2(A_u) A_t(A_u)|\geq
\frac{1}{2} \min(|A_u| |A_t|,p),\] which is slightly stronger than
(\ref{eq:jces}).
\begin{proof}
Call $a\in U$ a {\em pivot}
if the function $\phi_{a}:A_u \times A_t \to U$ given by
\[(u,t)\mapsto u t(a) =  u t a t^{-1}\]
is injective. 

{\em Case (a): There is a pivot $a$ in $A_u$.} 
Then
$|\phi_a(A_u,A_t)| = |A_u| |A_t|$, and so 
\[|A_u A_t(a)| \geq |\phi_a(A_u,A_t)| = |A_u| |A_t| .\]
This is the motivation for the name ``pivot'': the element $a$ is the pivot
on which we build an injection $\phi_a$, giving us the growth we want.

{\em Case (b): There are no pivots in $U$.} 
As we are about to see, 
this case can arise only if either $A_u$ or $A_t$ is large with respect to $p$.
Say that $(u_1,t_1)$, $(u_2,t_2)$ {\em collide} for $a\in U$ if
$\phi_a(u_1,t_1) = \phi_a(u_2,t_2)$. Saying that there are no pivots in $U$
is the same as saying that, for every $a\in U$, there are at least two distinct
$(u_1,t_1)$, $(u_2,t_2)$ that collide for $a$. Now, two distinct
$(u_1,t_1)$, $(u_2,t_2)$ can collide for at most one $a\in U\setminus \{e\}$.
(The reader can quickly check this; the reason is that
 a collision corresponds to 
a solution
to a non-trivial linear equation, which can have at most one solution.)
Hence, if there are no pivots, $|A_u|^2 |A_t|^2 \geq |U\setminus \{e\}| = 
p-1$, i.e., $|A_u|\cdot |A_t|$ is large with respect to $p$.
This already hints that this case will not be hard; it will yield to
Cauchy-Schwarz and the like.

 Choose the most ``pivot-like'' $a\in U$, meaning
an element $a\in U$ such that the number of collisions
\[\kappa_a = |\{u_1,u_2\in A_u, t_1,t_2\in A_t: \phi_a(u_1,t_1) = 
\phi_a(u_2,t_2)\}|\]
is minimal. As we were saying, 
two distinct $(u_1,t_1)$, $(u_2,t_2)$ collide for at most
one $a\in U\setminus e$. Hence
the total number of collisions $\sum_{a\in U\setminus \{e\}} \kappa_a$
is $\leq |A_u| |A_t| (p-1) + |A_u|^2 |A_t|^2$, and so
\[\kappa_a \leq \frac{|A_u| |A_t| (p-1) + 
|A_u|^2 |A_t|^2}{p-1} \leq
|A_u| |A_t| + \frac{|A_u|^2 |A_t|^2}{p} .\]
Cauchy-Schwarz implies that $|\phi_a(A_u,A_t)|\geq |A_u|^2 |A_t|^2/\kappa_a$,
and so
\[|\phi_a(A_u,A_t)|\geq \frac{|A_u|^2 |A_t|^2 }{
|A_u| |A_t| + \frac{|A_u|^2 |A_t|^2}{p}}
= \frac{1}{\frac{1}{|A_u| |A_t|} +
\frac{1}{p}}\geq \frac{1}{2} 
\min(|A_u| |A_t|,p).\]

We are not quite done, since $a$ may not be in $A$.
Since $a$ is {\em not} a pivot (as there are none), there exist distinct
$(u_1,t_1)$, $(u_2,t_2)$ such that $\phi_a(u_1,t_1)=\phi_a(u_2,t_2)$. Then
$t_1\ne t_2$ (why?), and so the map $\psi_{t_1,t_2}:U\to U$ given by
$u\mapsto t_1(u) (t_2(u))^{-1}$ is injective. What follows may be thought
of as the ``unfolding'' step, in that there is an element $a$ we want to remove
from an expression, and we do so by applying to the expression a map that will
send $a$ to something known. (We will be using the commutativity of $T$ here.)

For any $u\in U$, $t\in T$, since $T$ is abelian,
\begin{equation}\label{eq:unfold}\begin{aligned}
\psi_{t_1,t_2}(\phi_a(u,t)) &= t_1(u t(a)) (t_2(u t(a)))^{-1} =
t_1(u) t(t_1(a) (t_2(a))^{-1}) (t_2(u))^{-1}\\
&= t_1(u) t(\psi_{t_1,t_2}(a)) (t_2(u))^{-1} = t_1(u) t(u_1^{-1} u_2) (t_2(u))^{-1}
,\end{aligned}\end{equation}
(Note that $a$ has just
disappeared.) Hence, 
\[\psi_{t_1,t_2}(\phi_a(A_u,A_t))\subset A_t(A_u) A_t(A_u^2) A_t(A_u)
\subset (A_t(A_u))^4.\] Since
$\psi_{t_1,t_2}$ is injective, we conclude that
\[|(A_t(A_u))^4| \geq |\psi_{t_1,t_2}(\phi_a(A_u,A_t))| = 
|\phi_a(A_u,A_t)| \geq \frac{1}{2} \min(|A_u| |A_t|,p).\]

There is an idea here that we are about to see again: any element $a$ that 
is not a pivot can, by this very fact, 
be given in terms of some $u_1, u_2\in A_u$, 
$t_1, t_2\in A_t$, and so an expression involving $a$ can often be transformed
into one involving only elements of $A_u$ and $A_t$.

{\em Case (c): There are pivots and non-pivots in $U$.} 
This is what we can think of as an inductive step -- essentially the same
inductive step as in ``escape from subvarieties'' (to be discussed in \S \ref{subs:waziri}). To do induction, we do
not really need an ordering; generation will suffice.
Since $A_u \nsubseteq \{e\}$, $A_u$ generates
$U$. This implies that there is a non-pivot $a\in U$ and a $g\in A_u$ such that
$g a$ is a pivot. Then $\phi_{ag}:A_u\times A_t\to U$ is injective. Much as 
in (\ref{eq:unfold}), we unfold:
\begin{equation}\begin{aligned}
\psi_{t_1,t_2}(\phi_{g a}(u,t)) &= t_1(u t(g) t(a))
(t_2(u t(g) t(a)))^{-1} \\
&= 
t_1(u t(g)) t(u_1^{-1} u_2) (t_2(u t(g)))^{-1} ,
\end{aligned}\end{equation}
where $(u_1,t_1)$, $(u_2,t_2)$ are distinct pairs such that 
$\phi_a(u_1,t_1) = \phi_a(u_2,t_2)$. Just as before,
$\psi_{t_1,t_2}$ is injective.
Hence
\[|A_t(A_u) A_t^2(A_u) A_t(A_u^2) A_t^2(A_u) A_t(A_u)| \geq 
|\psi_{t_1,t_2}(\phi_{g a}(u,t))| = |A_u| |A_t|.\]

The idea to recall here is that, if $S$ is a subset of an orbit $\mathscr{O} = \langle A\rangle x$ such that $S\ne \emptyset$ and $S\ne \mathscr{O}$, then
there is an $s\in S$ and a $g\in A$ such that $g s\not\in S$. In other words,
we use the point at which we escape from $S$.
\end{proof} 
 We are using the fact that $G$ is the affine group over $\mathbb{F}_p$
(and not over some other field) only at the beginning of case (c), when we
say that, for $A_u\subset U$, $A_u \nsubseteq \{e\}$ implies 
$\langle A_u\rangle = U$.

\begin{prop}\label{prop:sabat}
Let $G$ be the affine group over $\mathbb{F}_p$.
Let $U$
be the maximal unipotent subgroup of $G$, and $\pi:G\to G/U$ the quotient
map.

Let $A\subset G$, $A=A^{-1}$, $e\in A$. Assume $A$ is not contained in any
maximal torus.
Then either
\begin{equation}\label{eq:jomar}
|A^{57}| \geq \frac{1}{2} \sqrt{|\pi(A)|} \cdot |A|
\end{equation}
or 
\begin{equation}
|A^{57}| \geq \frac{1}{2} |\pi(A)| p \text{\;\;\;\; and\;\;\;\;\;}
U \subset A^{112}
.\end{equation}
\end{prop}
\begin{proof}
We can assume $A\not\subset \pm U$, as otherwise what we are trying to prove
is trivial. Let $g$ be an element of $A$ not in $\pm U$; its centralizer $C(g)$
is a maximal torus $T$. By assumption, there is an element $h$ of $A$ 
not in $T$. Then $h g h^{-1} g^{-1}\ne e$. At the same time, it does lie in
$A^4\cap U$, and so $A^4\cap U$ is not $\{e\}$.

Let $A_u = A^4 \cap U$, $A_t=A^2\cap T$; their size is bounded 
from below by (\ref{eq:hastar}). 
Applying Prop.~\ref{prop:jutor}, we obtain
\[|A^{56}\cap U| \geq \frac{1}{2} \min(|A_u| |A_t|,p)  \geq
\frac{1}{2} \min\left(\frac{|A|}{|A^5|} \cdot |A|,p\right).\]
By (\ref{eq:easypie}), $|A^{57}|\geq |A^{56}\cap U|\cdot |\pi(A)|$. Clearly,
if $|A|/|A^5| < 1/\sqrt{|\pi(A)|}$, then $|A^{57}| \geq |A^5| > \sqrt{|\pi(A)|}
 \cdot |A|$.
\end{proof}
The exponent $57$ in (\ref{eq:jomar}) is not optimal, but, qualitatively 
speaking, Prop.\ \ref{prop:sabat} is as good a result as one can aim to for now:
the assumption $A\not\subset T$ is necessary, the bound $\gg |\pi(A)|\cdot p$
can be tight when $U\subset A$. For $A\subset U$, getting a better-than-trivial
bound amounts to Freiman's theorem in $\mathbb{F}_p$, and getting a growth
factor of a power $|A|^{\delta}$ (rather than $\sqrt{|\pi(A)|}$) would involve
getting a version of Freiman's theorem of polynomial strength (a difficult open 
problem).
 
Incidentally, (\ref{eq:jomar}) can be seen as a very simple result of the
``classification of approximate subgroups'' kind: if a set $A$ grows slowly
($|A^k| \leq |A|^{1+\delta}$, $k=57$, $\delta$ small) then either $A$
is contained in a subgroup (a maximal torus) or $A$ is
almost contained in a subgroup
($U$, with ``almost contained'' meaning that $|\pi(A)|\leq |A|^{\delta}$) or
$A^k$ contains a subgroup ($H=U$) such that $\langle A\rangle/H$ is nilpotent
(here, in fact, abelian).

A result of this kind was what \cite{GH2} proved for
solvable subgroups and what \cite{HeSL3} proved for $\SL_3(\mathbb{F}_p)$;
that is to say, one can try to classify growth in general linear algebraic
groups, leaving only the nilpotent case aside.
This was called the ``Helfgott-Lindenstrauss conjecture'' in 
\cite{BGTstru},
which proved it in an impressively general but quantitatively very weak 
sense. In particular, \cite{BGTstru} does recover a 
proof of Gromov's theorem (close to Hrushovski's \cite{MR2833482}), but it does not seem strong
enough to give useful bounds for finite groups.

\begin{center} 
* * *
\end{center}

The use of pivoting for general groups was first advocated in \cite{HeSL3},
but it came to full fruition only later, thanks to \cite{BGT}
and \cite{PS}: due in part 
to what in retrospect was a technical difficulty -- namely, an intermediate
result had an unnecessarily weak quantifier (see the remarks at the end
of \S \ref{subs:dime}) -- 
\cite{HeSL3} still uses a sum-product theorem at a certain point, though
it does develop a more abstract setting, which we have just
demonstrated here in the simplest case. 

{\em Excursus.} Let us comment briefly on how the sum-product theorem was used 
in the original proof of Thm.~\ref{thm:main08}. The point
 is not (as one might naively think) that multiplying matrices involves
summing products of matrix entries; that leads nowhere. Rather, the first
proof of Thm.~\ref{thm:main08} used at one point the identity
\begin{equation}\label{eq:feglar}
\tr(g) \tr(h) = \tr(gh) + \tr(g h^{-1}),
\end{equation}
valid for $\SL_2$, to show that, since the set $\tr(A)$ grows {\em either}
under multiplication or addition (by the sum-product theorem), it grows
{\em both} under multiplication and addition. Growth under addition then
 gets used to show that $\tr(A)$ grows under linear combinations, which
are then 
achieved by multiplying elements of $A$ by a fixed element of $A$; thus,
$\tr(A^k)$ grows as $k$ increases, and this is used to show that $A^k$ grows.

Later generalizations (\cite{HeSL3}, \cite{GH1}) showed that one can work with 
just about any linear identities involving traces instead of (\ref{eq:feglar}), provided that there are 
enough (independent) identities. This worked up to a certain point for $\SL_n$
and $\Sp_{2n}$, but not for $\SO_n$, thereby frustrating a full proof for
$\SL_n$. Both \cite{BGT} and \cite{PS} showed one can do without such 
identities altogether.


\section{Growth in linear algebraic groups}\label{sec:corro}

Here we will go over an essentially complete and self-contained proof of
Thm.~\ref{thm:main08}. The proof we will give
is somewhat more direct and easier to generalize than that in \cite{Hel08};
it is influenced by \cite{HeSL3}, \cite{BGT}, \cite{PS}, and also by
the exposition in \cite{MR3144176}. The basic elements are, however, the same:
a dimensional estimate gives us tori with many elements on them, and, aided
by an escape lemma, we will be able to use these tori to prove the theorem
by contradiction, using a pivoting argument (indirectly in \cite{Hel08},
directly here). The proof of the case of $\SL_2$ will be used to anchor
a more general discussion; we will introduce the concepts used in the general 
case, explaining them by means of $\SL_2$. We will actually prove Thm.~\ref{thm:main08} for a general finite field $\mathbb{F}_q$, since we have no longer
any use for the assumption that $q$ be prime.

We will then show how Thm.~\ref{thm:main08}  
and Prop.~\ref{prop:flatlem} imply Thm.~\ref{thm:bg} (Bourgain-Gamburd).

\subsection{Escape}\label{subs:waziri}
At some points in the argument, we will need to make sure that we can
find an element $g\in A^k$ that is {\em not} special: for example,
we want to be able to use a $g$ that is not unipotent, that does not have
a given $\vec{v}$ as an eigenvector, that {\em is} regular semisimple
(i.e., has a full set of distinct eigenvalues), etc. As \cite{BGT} states,
arguments allowing to do this appear in several places in the literature.
The first version of \cite{Hel08} did this ``by hand'' in each case, and so does
\cite{MR3144176}; that approach is useful if one aims at optimizing bounds, but
our aim here is to proceed conceptually. The following general statement, used
in \cite{HeSL3}, is modelled very closely after \cite[Prop. 3.2]{MR2129706}.

\begin{lem}[Escape]\label{lem:esc}
Let $G$ be a group acting linearly on a vector space $V/K$, $K$ a field.
Let $W$ be a subvariety of $V$ all of whose components have positive codimension
in $V$. Let $A\subset G$, $A =A^{-1}$, $e\in A$; 
let $x\in V$
be such that the orbit $\langle A\rangle\cdot x$ of $x$ is not contained in $W$.

Then there are constants $k$, $c$ depending only the number, dimension and
degree of the irreducible components of $W$ such that
there are at least $\max(1,c|A|)$ elements
$g\in A^k$ for which $g x\notin W$. 
\end{lem}
In other words, if $x$ can escape from $W$ at all, it can escape from $W$ in a 
bounded number of steps. 
\begin{proof}[Proof for a special case]
Let us first do the special case of $W$ an irreducible linear subvariety.
We will proceed by induction on the dimension of 
$W$. If $\dim(W)=0$, then $W$ consists of a single point, and the statement is
clear: $\langle A\rangle\cdot x \not\subset \{x\}$ implies that there is a 
$g_0\in A$ such that $g_0 x \ne x$; if there are fewer than $|A|/2$ such 
elements of $A$, any product $g^{-1} g_0$ with $g x = x$ satisfies
$g^{-1} g_0 x \ne x$, and there are $> |A|/2$ such products.

Assume, then, that $\dim(W)>0$, and that the statement has been proven for all
$W'$ with $\dim(W')<\dim(W)$. If $g W = W$ for all $g\in A$, then either
(a) $g x$ does not lie on $W$ for any $g\in A$, proving
the statement, or
(b) $g x$ lies on $W$ for every $g\in \langle A\rangle$, contradicting
the assumption. Assume that $g W \ne W$ for some $g\in A$; then $W'=gW\cap W$
is an irreducible linear variety with $\dim(W') < \dim(W)$. Thus,
by the inductive hypothesis, there are at least $\max(1, c' |A|)$ elements
$g'\in A^{k'}$ ($c'$, $k'$ depending only on $\dim(W')$) such that
$g' x$ does not lie on $W' = g W\cap W$. Hence, for each such $g'$, either 
$g^{-1} g' x$ or $g' x$ does not lie on $W$. We have thus proven the
statement with $c = c'/2$, $k = k'+1$.
\end{proof}
\begin{proof}[Adapting the proof to the general case]
Remove first the assumption of irreducibility; then $W$ is the union
of $r$ components, not necessarily all of the same dimension. The intersection 
$W' = g W \cap W$ may also have several components, but no more than $r^2$.
Let $W_1$ be a component of $W$ of maximal dimension $d$.  
By the argument in the
first sketch, we can find a $g\in A$ such that $g W_1\ne W_1$. (If $g x$
does not lie on $W_1$ for any $g\in A$, we simply remove $W_1$ from $W$
and repeat.) Hence $W' = g W\cap W$ has fewer components of dimension $d$
than $W$ does. We can thus carry out the induction on (a) the maximum of
the dimensions of the components of $W$, (b) the number of components of
maximal dimension: when (a) does not go down, it stays the same and (b) goes
down; moreover, the number of components of lower dimension stays under control,
as the total number of components $r$ gets no more than squared, as we said.

Removing the assumption that $V$ is linear is actually easy: the same
argument works, and we only need to make sure that the total number of
components (and their degree) stays under control; this is so by
Bezout's theorem (in a general form, such as that in
\cite[p.251]{MR1658464}, where the statement is credited 
to Fulton and MacPherson).
\end{proof}
As Pyber and Szab\'o showed in \cite{PS}, one can merge the ``escape'' argument
above with the ``dimensional estimates'' we are about to discuss, in that, 
in our context, an escape statement such as Lemma \ref{lem:esc} is really a weak
version of a dimensional estimate: Lemma \ref{lem:esc} tells us that many
images $g x$ escape from a proper subvariety $W$, whereas a dimensional estimate
tells us that, if $A$ grows slowly,
 very few images $g x$, $g\in A^k$, lie on a proper subvariety $W\subset G$.
We will, however, use Lemma \ref{lem:esc} as a tool to prove dimensional
estimates and other statements, much as in \cite{HeSL3} (or \cite{BGT}).

\subsection{Dimensional estimates}\label{subs:dime}
By a {\em dimensional estimate} we mean a lower or upper bound on an 
intersection of the form $A^k \cap W$, where $A\subset G(K)$,
$W$ is a subvariety of $G$ and $G/K$ is an algebraic group. As the reader
will notice, the bounds that we obtain will be meaningful when $A$ grows 
relatively slowly. However, no assumption on $A$ is made, other than that it
generate\footnote{Even this can be relaxed to require only that $\langle
  A\rangle$ not be contained in the union $V$ of a bounded number of varieties of 
positive codimension and bounded degree, as is clear from the arguments we will
see and as \cite{BGT} states explicitly. This boundedness condition 
is called ``bounded complexity'' in \cite{BGT}. The ``complexity''
in \cite{BGT} corresponds to the degree vector
$\overrightarrow{\text{deg}}(V)$ in \cite{HeSL3}.} $G(K)$.

Let us first look at a particularly simple example; we will not actually use
it as such here, but it was important in \cite{Hel08} and \cite{HeSL3}, and it
exemplifies what is meant by a ``dimensional estimate'' and one
way in which it can be proven. (Moreover, its higher-rank 
analogues do come into generalizations of what we will do to $\SL_n$ and
other higher-rank groups, and the ideas in its proof will be reused for
Prop.~\ref{prop:bantu}.)

\begin{prop}[{\rm {\cite[Lem~4.7]{Hel08}}; {\cite[Cor.~5.4]{HeSL3}}, case $n=2$}]\label{prop:port}
Let $G = \SL_2$, $K$ a field. Let $A\subset G(K)$ be a finite set with
 $A=A^{-1}$, $e\in A$, $\langle A \rangle = G(K)$.
Let $T$ be a maximal torus of $G$. Then
\begin{equation}\label{eq:hobo}|A \cap T(K)| \ll |A^k|^{1/3},\end{equation}
where $k$ and the implied constant are absolute.
\end{prop}
A {\em maximal torus}, in $\SL_2$ (or $\SL_n$), is just the group of
matrices that are diagonal with respect to some fixed basis of
$\overline{K}^2$. Here $G(K)$ simply means
the ``set of $K$-valued points'' of $G$, i.e., the group $\SL_2(K)$. (In
general, according to standard formalism,
an algebraic group is an abstract object (a variety plus morphisms);
its set of $K$-valued points is a group.)

The meaning of $1/3$ in (\ref{eq:hobo}) is that it equals $\dim(T)/\dim(G)$.
This will come through in the proof: we will manage to fit three copies of
$T$ inside $G$ in, so to speak, independent directions.
\begin{proof}[Proof, as in \cite{HeSL3}]
It would be enough to construct an injective map
\[\phi:T(K)\times T(K)\times T(K)\to G(K)\]
such that $\phi(T(K)\cap A,T(K)\cap A,T(K)\cap A)\subset A^k$, since then
\begin{equation}\label{eq:iskra}\begin{aligned}
|T(K)\cap A|^3 &= |(T(K)\cap A)\times (T(K)\cap A)\times (T(K)\cap A)|\\
 &= |\phi((T(K)\cap A) \times (T(K)\cap A) \times (T(K)\cap A))| \leq |A^k|.
\end{aligned}\end{equation}

It is easy to see that we can relax the condition that $\phi$ be injective;
for example, it is enough to assume that every preimage
$\phi^{-1}(g)$ has bounded size, and even then we can relax the condition
still further by requiring only that $\phi^{-1}(g) \cap (T(K)\setminus S)^3$
be of bounded size,
where $|S|$ is itself bounded, etc. Let us first construct $\phi$ and then see
how far we have to relax injectivity.

We define $\phi:T(K)\times T(K)\times T(K)\to G(K)$ by
\begin{equation}\label{eq:koweo}
\phi(t_0,t_1,t_2) = t_0 \cdot g_1 t_1 g_1^{-1} \cdot g_2 t_2 g_2^{-1},\end{equation}
where $g_1, g_2\in A^{k'}$, $k' = O(1)$ are about to be specified. It
is easy to show that there are $g_1', g_2'\in G(\overline{K})$ such that
$v$, $g_1' v g_1'^{-1}$ and $g_2' v g_2'^{-1}$ are linearly independent, where
$v$ is a non-zero tangent vector to $T$ at the origin. Now, the pairs 
$(g_1',g_2')\in G(\overline{K}) \times G(\overline{K})$ for which
$v$, $g_1' v g_1'^{-1}$, $g_2' v g_2'^{-1}$ are not linearly independent form
a subvariety $W$ of $G\times G$ of bounded degree
(given by the vanishing of a determinant). Since
$G\times G$ is irreducible, and since we have shown that there is at least
one point $(g_1',g_2')$ outside $W$, we see that all the components of $W$
have to be of positive codimension. Hence we can apply Lemma \ref{lem:esc}
(escape) with $A\times A$ (which generates $G\times G$) 
instead of $A$ and $G\times G$ instead of $G$ and $V$, and obtain that
there are $g_1,g_2\in A^{k'}$, $k'=O(1)$, such that 
$v$, $g_1 v g_1^{-1}$ and $g_2 v g_2^{-1}$ are linearly independent.
This means that the derivative of $\phi$ at the origin $(e,e,e)$
of $T\times T\times T$ is non-degenerate when any such $g_1,g_2\in A^k$
are given.

The points of $T\times T\times T$ at which $\phi$ has degenerate derivative
form, again, a subvariety $W_0$; since $T\times T\times T$ is irreducible,
and since, as we have just shown, the origin $(e,e,e)$ does not lie on $W_0$,
we see that $W_0$ is a union of components of positive codimension. This means that there
is a subvariety $W_1 \subset T\times T$ of bounded degree (bounded, mind you,
independently of $g_1$ and $g_2$), made out of components of positive codimension, such that, for all $(t_0,t_1)\in T(K)\times T(K)$
not on $W_1$, there are $O(1)$ elements $t_2\in T(K)$ such that 
$(t_0,t_1,t_2)$ lies on $W_0$; we also see that there is a 
subvariety $W_2\subset T$
of bounded degree and positive codimension such that, for all $t_0\in T(K)$
not on $W_2$, there are $O(1)$ elements $t_1\in T(K)$ such that $(t_0,t_1)$
lies on $W_1$.

Given any point $y$ on $G(K)$, its preimage under the restriction
$\phi|_{(T\times T\times T)\setminus W_0}$ lies on a variety of dimension zero:
if this were not the case, the preimage $\phi^{-1}(y)$ would be a variety
$V_y$ such that there is at least one point $x$ not on $W_0$ lying on a 
component of positive dimension of $V_y$. There would then have to be 
a non-zero tangent vector to $V_y$ at $x$, and we see that its image
under $D\phi$ would be $0$, i.e., $D\phi$ would be degenerate at $x$,
implying that $x$ lies on $W_0$; contradiction. 

The preimage of $y$ under $\phi|_{(T\times T\times T)\setminus W_0}$, besides
being zero-dimensional, is also
of bounded degree, because $\phi$ is of bounded degree. Hence the preimage consists of at
most $C$ points, $C$ a constant.

Similarly, considering the boundedness of the degrees of 
$W_0$, $W_1$ and $W_2$, we see that
 there are at most $O(1)$ points $t_0$ on $W_2$, there are at most
$|A\cap T(K)|\cdot O(1)$ points $(t_0,t_1)\in (A\cap T(K))\times (A\cap T(K))
$ on $W_1$ for $t_0$ not on $W_2$,
and there are at most $|A\cap T(K)|^2\cdot O(1)$ points $(t_0,t_1,t_2)\in
(A\cap T(K))\times (A\cap T(K))\times (A\cap T(K))$ on
$W_0$ for $(t_0,t_1)$ not on $W_1$. Hence
\[|\{x\in X: x\notin W_0(K)\}| \geq |A\cap T(K)|^3 - O(|A\cap T(K)|^2)\]
for $X =  (A\cap T(K))\times (A\cap T(K))\times (A\cap T(K))$,
and so
\[\phi(\{x\in X:
x\notin W_0(K)\}|) \geq \frac{|A\cap T(K)|^3 - O(|A\cap T(K)|^2)}{C}.\]
Since $\phi(X)$ lies in
$A^k$ for $k = 2 k' + 3$, we see that
\[|A\cap T(K)|^3 - O(|A\cap T(K)|^2) \leq C |A|\]
and so $|A\cap T(K)|^3 \ll |A^k|^{1/3}$.
\end{proof}
The proof above follows \cite[Cor.~5.4]{HeSL3}, which establishes the same 
result for $\SL_n$ and indeed for all classical Chevalley groups. It is
stated in more conceptual terms than the proof given in
\cite[Lem.~4.7]{Hel08} (as might be expected). In consequence, it generalizes
more easily; for instance, 
\cite{HeSL3} carries out the same argument for non-maximal tori
and unipotent subgroups. The main
step was stated in general terms in \cite[Prop. 4.12]{HeSL3}.

However, the situation was still not fully satisfactory. There is one passage
in the proof of Prop.~\ref{prop:port} that isn't quite abstract enough: the
one that starts with ``It is easy to show''. This is a very simple computation
for $\SL_2$, and fairly easy even for $\SL_n$. The problem is that it has
to be done from scratch for a specific algebraic
 group $G$ and a specific variety $V$ every time we want to generalize
Prop.~\ref{prop:port}. This means, first, that we have to go through Lie group
types if we want a statement that is general on $G$, and that is tedious.
Second, and more importantly, this kept the author from giving a statement
for fully general $V$ in \cite{HeSL3}, as opposed to a series of statements for
different varieties $V$.

A full generalization in these two senses was achieved independently by
\cite{PS} and \cite{BGT}. It turns out that all we need to know about the
algebraic group $G$ is that it is simple (or almost simple, like $\SL_n$).
Under that condition, and assuming that $\langle A \rangle = G(K)$,
Pyber and Szab\'o proved \cite[Thm~24]{PS} that, for every subvariety 
$V\subset G$ of positive dimension and every $\epsilon>0$,
\begin{equation}\label{eq:absinthe}|A\cap V|\ll_\epsilon |A^k|^{(1+\epsilon) \frac{\dim(V)}{\dim(G)}},
\end{equation}
where $k$ and the implied constant depend only on $\epsilon$, on the
degree and number of the components of $V$ and on the rank (and Lie type) of 
$G$. This they did by greatly generalizing and strengthening the
arguments in \cite{HeSL3} (such as, for example, the proof of 
Lemma \ref{prop:port} above). 

The route followed by Breuillard, Green and Tao \cite{BGT} was a little different.
By then \cite{HeSL3} was known -- the first version was made public in 2008 --
and the author had conversed at length with one of the authors of
\cite{BGT} about the ideas involved and the difficulties remaining.
The preprint of Larsen-Pink \cite{LP} was available, as were works
by Hrushovski-Wagner
\cite{MR2436141} and Hrushovski \cite{MR2833482}\footnote{Pyber and Szab\'o
also mention the earlier paper \cite{MR1329903} (Hrushovski-Pillay) 
as an influence.
Part of the role of \cite{MR2833482} was to make the work of Larsen-Pink
clearer. According to \cite[p.~1108]{LP}, there was actually
a gap in the original version of \cite{LP}, and \cite{MR2833482} filled
that gap, besides giving a more general statement.}.
The aim of \cite{LP} was to give a classification of subgroups of $\GL$ 
without using the Classification Theorem of finite simple groups. This 
involved stating and
proving (\ref{eq:absinthe}) (without $\epsilon$) for $A$ a subgroup of $G(K)$
\cite[Thm. 4.2]{LP}. This proof turned out to be robust (as Hrushovski and
Wagner's 
model-theoretical work may have already indicated): \cite{BGT} adapted it
to the case of $A$ a set, obtaining (\ref{eq:absinthe}) for every $A$ with
$\langle A\rangle = G(K)$ (without the $\epsilon$, i.e., strengthened). 

Rather than prove (\ref{eq:absinthe}) here in full generality, 
we will prove in a case in which
we need it. This is a case that, for $\SL_2$, would have been accessible
by an approach even more concrete than that in \cite{Hel08}, as
\cite{MR3144176} shows. However, the more conceptual proof below is arguably
simpler, and also displays the main ideas in the proof of the general
statement (\ref{eq:absinthe}).

\begin{prop}\label{prop:bantu}
Let $G=\SL_2$, $K$ a field. Let $A\subset G(K)$ be a finite set with $A=A^{-1}$,
$e\in A$, $\langle A\rangle = G$. Let $g$ be a regular semisimple element of $G(K)$. Then
\begin{equation}\label{eq:bernan}
|A\cap \Cl(g)|\ll |A^k|^{2/3} ,\end{equation}
where $k$ is an absolute constant.
\end{prop}
For $G=\SL_2$, $g$ is regular semisimple if and only if it has two distinct
eigenvalues. In that case, $\Cl(g)$ is just the
subvariety $W$ of $G$ defined by $W=\{x:\tr(x) = \tr(g)\}$.
(In general, in any linear algebraic group $G$,
the conjugacy class of a semisimple element is a (closed) variety
\cite[Cor.~2.4.4]{MR842444}.)
Thus, $\dim(\Cl(g))/\dim(G) = 2/3$; this is the meaning of the exponent in 
(\ref{eq:bernan}).
The centralizer of a regular semisimple element is a torus.

The proof below is a little closer to \cite{LP} ({\em apud} \cite{BGT})
than to \cite{PS}, and also brings in some more ideas from \cite{HeSL3}.
In both \cite{LP} and \cite{PS}, recursion is used to reduce the problem
to one for lower-dimensional varieties (not unlike what happens in the proof of
Lem.~\ref{lem:esc}).\footnote{The $\epsilon$ in (\ref{eq:absinthe}) appears in 
\cite{PS} because the recursion there does not always end in the 
zero-dimensional case; rather, an excess concentration on a variety
gets shuttled back and forth (``transport'', \cite[Lem.~27]{PS}) 
and augmented by itself a bounded number of times, until it is too large,
yielding a contradiction and thereby proving (\ref{eq:absinthe}).} 
\begin{proof}[Proof of Prop.~\ref{prop:bantu}]
Write $Y$ for the variety $\overline{\Cl(g)}$.
We start as in the proof of Prop.~\ref{prop:port}, defining a map
$\phi:Y\times Y\to G$ by
\begin{equation}\label{eq:slavo}\phi(y_0,y_1) = y_0 y_1.\end{equation}
(We do not bother to conjugate as in (\ref{eq:koweo})
because $Y$ is invariant under conjugation; it is also invariant
under inversion.)
The preimage of a generic point of $G$ is not, unfortunately, $0$-dimensional,
since $\dim(Y\times Y) = 2\cdot 2 > 3 = \dim(G)$. Let $g\in \phi(Y\times
Y)$. The preimage of $g$ is
\begin{equation}\label{eq:kolz}\{(y_0,y_0^{-1} g): y_0 \in Y, y_0^{-1} g\in Y\} = 
\{(y_0,y_0^{-1} g): y_0 \in Y \cap g Y\}.\end{equation}
It is clear from this that the dimension of the preimage of $g$ equals 
$\dim(Y\cap g Y)$, and so there are at most two
 points with $2$-dimensional preimage, namely, $g = \pm e\in G$.  
Assume $g\ne \pm e$.

By the usual
$|\text{domain}| \geq |\text{image}|/|\text{largest preimage}|$ argument, 
(\ref{eq:kolz}) implies
\[\phi(\{(y_0,y_1)\in (A\cap Y)\times (A\cap Y): y_1\ne \pm y_0^{-1}\}) \geq
\frac{|A\cap Y| (|A\cap Y|-2)}{\max_{g\ne \pm e} |A\cap Y\cap g Y|},\]
and, since $\phi(\{(y_0,y_1)\in (A\cap Y)\times (A\cap Y)\})\subset A^2$,
we see that
\begin{equation}\label{eq:korot}
|A\cap Y| \leq 2 + \sqrt{|A^2|\cdot \max_{g\ne \pm e} |A\cap Y\cap g Y|}.
\end{equation}
Hence, we should aim to bound $|A\cap Y\cap gY|$ from above.
The number of components of $Y\cap gY$ is $O(1)$ because the degree of $Y$
and $gY$ is bounded; let $Z$ be the irreducible component of $Y\cap gY$
containing the most elements of $A$. Since $g\ne \pm e$, $\dim(Z)\leq 1$;
we can assume that $\dim(Z)=1$, as otherwise what we wish to prove is trivial.
We want to bound $A\cap Z$.
(This is the {\em recursion} we referred to before; we are descending
to a lower-dimensional variety $Z$.)

We will now consider a map $\psi:Z\times Z\times Z\to G$ given by
\[\psi(z_0,z_1,z_2) = z_0 \cdot g_1 z_1 g_1^{-1} \cdot g_2 z_2 g_2^{-1}.\]
Much as in the proof of Prop.~\ref{prop:port}, we wish to show that
there are $g_1,g_2\in A^{k'}$ such that, at all points of $Z\times Z\times Z$ 
outside a proper subvariety $W_0$ of $Z\times Z\times Z$, the derivative of
$\psi$ is non-degenerate. Just as before, it will be enough
to find a single point of $Z\times Z\times Z$ at which the derivative
is non-degenerate. Choose any point $z_0$ on $Z$; we will look at the point
$(z_0,z_0,z_0)$.

We write $\mathfrak{z}$ for the tangent space of $Z z_0^{-1}$ at the
origin; it is a subspace of the Lie algebra of $G$. First, we compare
the vector space $\mathfrak{z}$ (obtained by deriving 
$\psi(z,z_0,z_0) \psi(z_0,z_0,z_0)^{-1}$ at $z=z_0$) and the vector
space $z_0 g_1 \mathfrak{z} g_1^{-1} z_0^{-1}$ (obtained by deriving 
$\psi(z_0,z,z_0) \psi(z_0,z_0,z_0)^{-1}$ at $z=z_0$).
We would like them to be linearly independent, i.e., intersect only at
the origin; this is the same as asking
whether there is a $g\in G(\overline{K})$ such that $\mathfrak{z}$ and
 $g\mathfrak{z} g^{-1}$ are linearly independent, since we can set 
$g_1 = z_0^{-1} g$. If there were no such $g$, then $g\mathfrak{z} g^{-1} =
\mathfrak{z}$ for all $g\in G(\overline{K})$, and so
$\mathfrak{z}$ would be an ideal of the Lie algebra $\mathfrak{g}$ of
$G$ (i.e., it would be invariant under the action $\ad$ of the Lie bracket).
However, this is impossible, since a simple (or almost simple\footnote{
Meaning that $G$ has no {\em connected} normal algebraic subgroups other than itself and
the identity; this is the case for $G=\SL_2$.})
group of Lie type $G$ has a simple
Lie algebra (prove as in \cite[III.9.8, Prop.~27]{MR0573068}), i.e.,
an algebra $\mathfrak{g}$ without ideals other than itself and $(0)$.
Hence there is a $g_1$ such that $\mathfrak{z}$ and
$z_0 g_1 \mathfrak{z} g_1^{-1} z_0^{-1}$ are linearly independent. This means
that a determinant does not vanish at $g_1$; thus we see that $\mathfrak{z}$ and
$z_0 g_1 \mathfrak{z} g_1^{-1} z_0^{-1}$ are linearly independent for $g_1$
outside a subvariety $W$ that is a union of (a bounded number of) components
of positive codimension (and bounded degree). If $|K|$ is larger than a
constant, a simple bound on the number of points on varieties
(weaker than either Lang-Weil or Schwartz-Zippel; \cite[Lem. 1]{MR0065218}
is enough) shows that $W(K)\ne G(K)$;
we can assume that $|K|$ is larger than a constant, as otherwise the
statement of the proposition we are trying to show is trivial.
By escape from subvarieties
(Lem.~\ref{lem:esc}), it follows there is a $g_1\in A^k$, $k=O(1)$,
 lying outside $W$; let us fix one such $g_1$.

We also want 
\[(x_0 g_1 x_0 g_1^{-1}) g_2 \mathfrak{z} g_2^{-1} (x_0 g_1 x_0 g_1^{-1})^{-1},\]
obtained by deriving $\psi(z_0,z_0,z) \psi(z_0,z_0,z_0)^{-1}$ at $z=z_0$, 
to be linearly independent from $\mathfrak{z} + 
x_0 g_1 \mathfrak{z} g_1^{-1} x_0^{-1}$. This is done by exactly the same
argument: $\mathfrak{z} + 
x_0 g_1 \mathfrak{z} g_1^{-1} x_0^{-1}$ cannot be an ideal of $\mathfrak{g}$
(because there are none, other than $\mathfrak{g}$ and $(0)$) and so there
is a subvariety $W'\subsetneq G$ (the union of a bounded number
of components of positive codimension and bounded degree) such that the linear
independence we want does hold for $g_2$ outside $W'$.
Again by a bound on the number of points, followed by escape 
(Lem.~\ref{lem:esc}),
there is a $g_2\in A^{k'}$, $k'=O(1)$, lying outside $W'$.

Thus, we have $g_1,g_2\in A^{k'}$ 
such that $\psi$ has a non-degenerate derivative
at at least one point (namely, $(z_0,z_0,z_0)$). This means that $\psi$
has non-degenerate derivative outside a proper subvariety 
$W_0\subset Z\times Z\times Z$ (consisting of a bounded number of components
of bounded degree).
We finish exactly as in the proof of Prop.~\ref{prop:port}, using 
the same counting argument to conclude that
\[\psi(\{x\in X:
x\notin W_0(K)\}|) \gg |A\cap Z(K)|^3 - O(|A\cap Z(K)|^2\]
and so $|A\cap Z(K)|^3 \ll |A^k|^{1/3}$ for some $k=O(1)$.
Substituting this into (\ref{eq:korot}), we obtain that
\[|A\cap Y| \ll
2 + \sqrt{|A^2|\cdot |A^k|^{1/3}} \ll |A^k|^{2/3}\]
for $A$ non-empty, as desired. 
\end{proof}
\begin{Rem}
The proof of Prop.~\ref{prop:bantu} lets itself be generalized fairly
easily. In particular, note that we used no properties of $Z$ other than
the fact that it was an irreducible one-dimensional variety of dimension
$1$. Thus, we have actually shown that
\[|A\cap Z(K)|\ll |A^k|^{\frac{1}{\dim(G)}}\]
for any one-dimensional irreducible subvariety of a simple (or almost
simple)
group of Lie type $G$; the implied constant depends only on the degree 
of $Z$, the dimension of $G$ and the degree of the group operation in $G$
as a morphism. The main way in which the proof was easier than a full
proof of 
\[|A\cap V|\ll |A^k|^\frac{\dim(V)}{\dim(G)}\]
 is that, since we were dealing with
low-dimensional varieties, the inductive process was fairly simple, as
were some of the counting arguments. However, the basic idea of the general
inductive process is the same as here -- go down in dimension, keeping
the degree under control. It is not necessary for the maps used in the
proof to be roughly injective (like the map $\phi$ in (\ref{eq:koweo})) ,
as 
long as the preimage of a generic point is
a variety whose dimension is smaller than the dimension of the domain
(as is the case for the map $\phi$ in (\ref{eq:slavo})).
This means, in particular, that we need not try to make the Lie algebras
$\mathfrak{z}$, $g \mathfrak{z} g^{-1}$ be linearly independent -- it is
enough to ask that $g \mathfrak{z} g^{-1}$ not be equal to $\mathfrak{z}$
(clearly a weaker condition when $\dim(\mathfrak{z})>1$); we
get $g \mathfrak{z} g^{-1} \ne \mathfrak{z}$ easily by the same argument 
as in the proof of Prop.~\ref{prop:bantu}, using the simplicity of $G$.
\end{Rem}

\begin{center}
* * *
\end{center}
We can now use the orbit-stabilizer theorem for sets to convert the upper
bound given by Prop.~\ref{prop:bantu} into a lower bound for $|A^2\cap T(K)|$.
The idea is that having many elements of $A^2$ (or $A^k$)
of a special form (e.g., elements of $T(K)$) is extremely useful. 
\begin{cor}\label{cor:nessumo}
Let $G=\SL_2$, $K$ a field. Let $A\subset G(K)$ be a finite set with $A=A^{-1}$,
$e\in A$, $A\subset G(K)$.
 Let $g$ be a regular semisimple element of $A^l$, $l\geq 1$. Assume
$|A^3|\leq |A|^{1+\delta}$, $\delta>0$. Then
\[|A^2\cap C(g)|\gg |A|^{\frac{1}{3} - O(\delta)} ,\]
where $k$ and the implied constants depend only on $l$.
\end{cor}
This corresponds to \cite[Thm. 6.2]{LP}. The centralizer $C(g)$ is a maximal
torus $T$. Clearly, $1/3 = \dim(T)/\dim(G)$.
\begin{proof}
By Prop.~\ref{prop:bantu} and Lemma~\ref{lem:lawve} (orbit-stabilizer),
\[|A^2\cap C(g)|\gg \frac{A}{|A^{k l}|^{2/3}} .\]
The inequality $|A^{kl}|\ll |A|^{1+O(\delta)}$ follows from
$|A^3|\leq |A|^{1+\delta}$ and (\ref{eq:jotor}).
\end{proof}

This has been a centerpiece of proofs of growth in $\SL_2$ -- and
linear algebraic groups in general --
since \cite{Hel08}. At the same time, there is a precision to be made.
Proposition 4.1 in that paper proved Cor.~\ref{cor:nessumo} for 
{\em some} $g\in A$ (i.e., there exists a torus $T=C(g)$ with a large
intersection with $A^2$). The same proof given there shows that
this is true for {\em most} $g\in A$, but it does not prove it for {\em
  all} $g\in A$. This was, in retrospect, an important technical weakness.
It made the rest of the argument more indirect and harder
to generalize. This remained the case when Cor.~\ref{cor:nessumo} was
itself stated in greater generality in \cite[Prop.~5.8 and Cor.~5.10]{HeSL3}.
 
\subsection{High multiplicity and spectral gaps, I}\label{subs:highmult}
In order to supplement our main argument, we will need to be able to show that,
if $A$ is very large ($|A|>|G|^{1-\delta}$, $\delta>0$ small), then
$(A\cup A^{-1} \cup \{e\})^k = G$. (See the statement of Thm.~\ref{thm:main08}.)
This task is not particularly hard; 
in \cite{Hel08}, it was done ``by hand'', using a descent to a Borel subgroup
and results on large subsets of $\mathbb{F}_p$. As Nikolov and Pyber later
pointed out, one can obtain a stronger result (with $k=3$) in a way that
generalizes very easily. This requires a key concept -- that of high eigenvalue
multiplicity -- which will appear again in \S \ref{subs:majaro}.

\begin{prop}[Frobenius]\label{prop:frob}
Let $G = \SL_2(\mathbb{F}_q)$, $q=p^\alpha$, $p$ odd. Then every non-trivial
complex representation of $G$ has dimension at least $(q-1)/2$.
\end{prop}
It is of course enough to show that every {\em irreducible} non-trivial
complex representation has dimension at least $(q-1)/2$.
\begin{proof}
By, e.g., the character tables in \cite[\S 5.2]{fultonharris}.
See also the standard reference 
\cite{MR0360852} or the exposition in \cite[\S 3.5]{MR1989434} (for
$G = \PSL_2(\mathbb{F}_p)$, $p$ prime).
\end{proof}
There are analogues of Prop.~\ref{prop:frob} for all finite
simple groups of Lie type.

Now consider a Cayley graph $\Gamma(G,A)$, where $A$ generates $G$ and
$A = A^{-1}$; we recall that
this is defined to be the graph having $G$ as
its set of vertices and $\{(g,ag): g\in G, a\in A\}$ as its set
of edges. We recall that 
the normalized {\em adjacency matrix} $\mathscr{A}$ is a linear operator
on the space $L$ of complex-valued functions on $G$: it is defined by
\begin{equation}\label{eq:hoko}
(\mathscr{A} f)(g) = \frac{1}{|A|} \sum_{a\in A} f(a g).\end{equation}
 Since $A$ is symmetric, $\mathscr{A}$
 has a full real spectrum
\begin{equation}\label{eq:trout}
\dotsc \leq \lambda_2 \leq \lambda_1\leq \lambda_0 = 1\end{equation}
with orthogonal eigenvectors $v_j$; the eigenvector $v_0$ corresponding to
the highest eigenvalue $\lambda_0$ is just a constant function. 

We can see from the definition (\ref{eq:hoko}) that every eigenspace of 
$\mathscr{A}$ is invariant under the action of $G$ on the right; in other
words, it is a representation -- and it can be trivial only for the eigenspace
consisting of constant functions, i.e., the eigenspace associated to
 $\lambda_0$. Hence, by Prop.~\ref{prop:frob}, every eigenvalue $\lambda_j$,
$j>0$, has multiplicity at least $(q-1)/2$.

The idea now is to use this high multiplicity to show a spectral gap,
i.e., a non-trivial upper bound for $\lambda_1$. Let us follow
\cite{MR2410393}, which shows that this is not hard for $A$ large.
The trace of $\mathscr{A}^2$ can be written in two ways: on one hand, it
is $1/|A|^2$ times the
 number of length-$2$ paths whose head equals their tail, and, on
the other, it equals a sum of squares of eigenvalues. In other words,
\[\frac{|G| |A|}{|A|^2} = \sum_j \lambda_j^2 \geq \frac{q-1}{2} \lambda_j^2\]
for every $j\geq 1$, and so
\[|\lambda_j| \leq \sqrt{\frac{|G|/|A|}{\frac{q-1}{2}}}.\]
If $A$ is large enough (close to $|G|$ in size), 
this is much smaller than $\lambda_0 = 1$.
This means that a few applications of $\mathscr{A}$ ``uniformize'' any
distribution very quickly, in that anything orthogonal to a constant function
gets multiplied by $\lambda_1<1$ (or less) repeatedly.
The proof of the following result is based on this idea.

\begin{prop}[\cite{MR2410393} and \cite{MR2800484}]\label{prop:diplo}
Let $G = \SL_2(\mathbb{F}_q)$, $q$ an odd prime power. Let $A\subset G$,
$A = A^{-1}$. Assume $|A|\geq 2 |G|^{8/9}$. Then
\[A^3 = G.\]
\end{prop}
Neither \cite{MR2410393} nor \cite{MR2800484} require $A=A^{-1}$; we
are assuming it for simplicity.
\begin{proof}
Suppose there is a $g\in G$ such that $g\notin A^3$. Then the inner
product $\langle \mathscr{A} 1_A, 1_{g A^{-1}}\rangle$ equals $0$.
We can assume that eigenvectors $v_j$ have $\ell_2$-norm $1$ (relative to
the counting measure on $G$, say). Then
\begin{equation}\label{eq:keki}\begin{aligned}
&\langle \mathscr{A} 1_A, 1_{g A^{-1}}\rangle =
\lambda_0 \langle 1_A,v_0\rangle \langle v_0, 1_{g A^{-1}}\rangle +
\sum_{j\geq 1} \lambda_j \langle 1_A, v_j\rangle \langle v_j, 1_{g A^{-1}}\rangle\\
&=  |A|\cdot \left(\frac{|A|}{\sqrt{|G|}}\right)^2 + 
O^*\left(\sqrt{\frac{2 |G| |A|}{q-1}}
\sqrt{\sum_{j\geq 1}  |\langle 1_A, v_j\rangle|^2}
\sqrt{\sum_{j\geq 1} |\langle v_j, 1_{g A^{-1}}\rangle|^2}\right)\\
&=  \frac{|A|^3}{|G|} + 
O^*\left(\sqrt{\frac{2 |G| |A|}{q-1}} |1_A|_2 |1_{g A^{-1}}|_2\right)
=  \frac{|A|^3}{|G|} + O^*\left(\sqrt{\frac{2|G| |A|}{q-1}} |A|\right) .
\end{aligned}\end{equation}
By $|G|= (q^2-q)q$, however, $|A|\geq 2 |G|^{8/9}$ implies
$|A|^3/|G| > \sqrt{2 |G| |A|/(q-1)} |A|$, and so (\ref{eq:keki})
means that 
$\langle \mathscr{A} 1_A, 1_{g A^{-1}}\rangle$ cannot be $0$.
\end{proof}



\subsection{Growth in $\SL_2(\mathbb{F}_q)$}\label{subs:pivoro}
We finally come to the proof of Thm.~\ref{thm:main08}. This is a ``modern'' 
proof, without any reliance on the sum-product theorem, and with a fairly
straightforward generalization to higher-rank groups. This part is a little
closer to \cite{PS} 
than to the first version
of \cite{BGT}.  Note the parallels with Prop.~\ref{prop:jutor} (pivoting).

We will actually be able to prove Thm.~\ref{thm:main08} over a general
finite field $\mathbb{F}_q$, and not just over $\mathbb{F}_p$. The statement
is as follows.
\begin{thm}
Let $G = \SL_2(\mathbb{F}_q)$. Let $A\subset G$ generate $G$. Then either
\[|A^3|\geq |A|^{1+\delta}\]
or $(A \cup A^{-1}\cup \{e\})^k = G$, where $\delta>0$ and 
$k\geq 1$ are absolute constants.
\end{thm}
\begin{proof}
By (\ref{eq:mony}), we can assume that $A=A^{-1}$ and
$e\in A$ without loss of generality. We can also assume $A$ to be larger 
than an absolute constant, as otherwise $|A^2| \geq |A|+1$ gives us
$|A^3|\geq |A|^{1+\delta}$ trivially. (If $|A^2|\leq |A|$, then
$A^2 \supset A\cdot e = A$ implies $A^2 = A$, and, since 
$A$ generates $G$, $A \cdot A = A$ implies that $A = G$.)

Assume $|A^3|<|A|^{1+\delta}$, where $\delta>0$ will be set later.
By Lemma \ref{lem:esc} (escape), there is an element $g_0\in A^2$ that is
regular semisimple (that is, $\tr(g_0)\ne \pm 2$). Its centralizer
$T = C_G(g_0)$ is a maximal torus.

Call $\xi \in G$ a {\em pivot} if the function 
$\phi_\xi:A\times T \to G$
defined by
\begin{equation}\label{eq:zazu}(a,t)\mapsto a \xi t \xi^{-1}\end{equation}
is injective when considered as a function from $A/\{\pm e\}
\times T/\{\pm e\}$ to
$G/\{\pm e\}$. (The analogy with the proof of Prop.~\ref{prop:jutor} is 
deliberate.)

{\em Case (a): There is a pivot $\xi$ in $A$.} Since $T = C(g_0)$,
 Cor.~\ref{cor:nessumo} (together with $|A^k|\leq |A|^{1+O(\delta)}$)
gives us that there are $\gg |A|^{1/3-O(\delta)}$ elements of $T$ in 
$A^2$. Hence, by the injectivity of $\phi_\xi$,
\[\left|\phi_\xi(A,A^2\cap T)\right| \geq \frac{1}{4} |A| |A^2\cap T|
 \gg |A| |A|^{\frac{1}{3} - O(\delta)}
= |A|^{\frac{4}{3} - O(\delta)}.\]
At the same time, $\phi_x(A,A^2\cap T) \subset A^5$, and so
\[|A^5|\gg |A|^{4/3-O(\delta)}.\]
For $\delta$ smaller than a positive constant,
this gives a contradiction
to $|A^3|<|A|^{1+\delta}$ by Ruzsa's inequality (\ref{eq:jotor}).
(Recall that we can assume that $|A|$ is larger than an absolute constant.)

{\em Case (b): There are no pivots $\xi$ in G.} Then, for every $\xi\in G$,
there are 
$a_1,a_2\in A$, $t_1,t_2\in T$, $(a_1,t_1)\ne (\pm a_2,\pm t_2)$ such that
$a_1 \xi t_1 \xi^{-1} = \pm e\cdot  a_2 \xi t_2 \xi^{-1}$, i.e.
\[a_2^{-1} a_1 =  \pm e \cdot \xi t_2 t_1^{-1} \xi^{-1}.\]
In other words,
\begin{equation}\label{eq:rojo}
A^{-1} A \cap \xi T \xi^{-1} \ne \{\pm e\}\end{equation}
for every $\xi\in G$.

Choose any $g\in \xi T \xi^{-1}$ with $g \ne \pm e$. Then $g$ is regular
semisimple and its centralizer $C(g)$ equals $\xi T \xi^{-1}$. (This is particular to $\SL_2$;
see the comments after the proof.) Thus, by Cor.~\ref{cor:nessumo},
we obtain that there are $\geq c |A|^{1/3 - O(\delta)}$ elements of 
$\xi T \xi^{-1}$ in $A^k$, where $k$ and $c>0$ are absolute. 
 This is valid for every conjugate
$\xi T \xi^{-1}$ of $T$ with $\xi \in G$. At least $(1/2) |G|/|T|$
 maximal tori of $G$ are of the form $\xi T \xi^{-1}$,
$\xi \in G$. Hence 
\begin{equation}\label{eq:abar}
|A^k| \geq \frac{1}{2} \frac{|G|}{|T|} (c |A|^{1/3-O(\delta)} - 2) \gg 
p^2 |A|^{1/3-O(\delta)}.\end{equation}
(Since any element of $G$ other than $\pm e$ can lie on at most one maximal
torus, there is no double counting.)

From (\ref{eq:abar}) it follows immediately that either 
$|A^3|\geq |A|^{1+\delta}$ (use (\ref{eq:jotor})) or $A\geq |G|^{1-O(\delta)}$.
In the latter case, Prop.~\ref{prop:diplo} implies that $A^3 = G$.

{\em Case (c): There are pivots and non-pivots in $G$.} Since
$\langle A\rangle = G$, this implies that there is
a $\xi\in G$, not a pivot, and an $a\in A$ such that $a \xi$ is a pivot. 
Since $\xi$ is not a pivot, we obtain (\ref{eq:rojo}), 
and so there are $\geq |A|^{1/3-O(\delta)}$ elements of $\xi T \xi^{-1}$
in $A^k$, just as before.

At the same time, $a\xi$ is a pivot, i.e., the map $\phi_{a \xi}$ (defined in 
(\ref{eq:zazu})) is injective. Hence
\[\left|\phi_{a \xi}(A, \xi^{-1} (A^k \cap \xi T \xi^{-1}) \xi)\right| \geq \frac{1}{4}
|A| |A^k \cap \xi T \xi^{-1}| \geq \frac{1}{4} |A|^{\frac{4}{3} - O(\delta)}.\]
Since $\phi_{a \xi}(A, \xi^{-1} (A^k \cap \xi T \xi^{-1}) \xi)y \subset A^{k+3}$,
it follows that \begin{equation}\label{eq:matameri}
|A^{k+3}|\geq \frac{1}{4} |A|^{4/3 - O(\delta)}.\end{equation}
We set $\delta$ small enough for Ruzsa's inequality (\ref{eq:jotor})
to imply that (\ref{eq:matameri}) contradicts $|A^3|\leq |A|^{1+\delta}$.
\end{proof}

One apparent obstacle to a generalization here is the fact that, in
higher-rank groups (e.g. $\SL_n$, $n\geq 3$), the centralizer $C(g)$
of an element $g\ne \pm e$ of a torus $T$ is not necessarily equal to $T$;
we have $C(g)=T$ only when $g$ is regular. This obstacle is not serious
here, as the number of non-regular elements of $A$ on a torus is
small by a dimensional bound; this is already in \cite[\S 5.5]{HeSL3}.
The difficulty in generalizing Thm.~\ref{thm:main08} to higher-rank
groups (\cite{HeSL3}, \cite{GH1}) resided, in retrospect, in the fact
that the version of Cor.~\ref{cor:nessumo} in \cite[\S 4]{Hel08} and
\cite[Cor.~5.10]{HeSL3} was slightly weaker,
as discussed before. This made the pivoting argument more complicated
and indirect, and thus harder to generalize; in particular, the sum-product
theorem was still used, in spite of the attempts to gain independence from
it in \cite[\S 3]{HeSL3}. These issues were resolved in both
\cite{PS} and \cite{BGT}. 

As pointed out in \cite{BGT}, Thm.~\ref{thm:main08} actually implies
the sum-product theorem; however, it is arguably more natural to deduce
the sum-product theorem, or Prop.~\ref{prop:musoff},   
from growth in the affine group (Prop.~\ref{prop:jutor}); multiplication
and addition correspond to two different group actions. See \S \ref{sec:sump}.

Let us now go briefly over the successive generalizations
of Thm.~\ref{thm:main08} up to now.
This will not just be a way to mention a series
of contributions by different people in the field -- it will also
be a good opportunity to make a few additional points on the internal
structure of the subject.

There are two main kinds of generalizations: changing the field and 
changing the Lie type.

{\bf Changing the field.} 
Dinai \cite{MR2788087} showed that Thm.~\ref{thm:main08} holds 
with $\mathbb{F}_p$ changed to a general finite field $\mathbb{F}_q$;
at the time,
this involved some careful work involving the subgroup structure of
$\SL_2(\mathbb{F}_q)$.
Varj\'u \cite{MR2862040} gave a second, simpler proof of the same 
generalization.

The proof in \cite{Hel08}
easily gives that Thm.~\ref{thm:main08} still holds if 
$\mathbb{F}_p$ is changed to $\mathbb{C}$ and $A$ is taken to be 
finite. In fact, in that case, there is a predecessor: Elekes 
and Kir{\'a}ly proved \cite{MR1869409} 
a result corresponding to Thm.~\ref{thm:main08}
with $\mathbb{R}$ instead of $\mathbb{F}_p$, with unspecified growth
bounds.
In general, in arithmetic combinatorics, results over $\mathbb{R}$ or 
$\mathbb{C}$ are more
accessible than results over $\mathbb{F}_p$: 
$\mathbb{R}$ has an ordering and a topology that a
general field, or $\mathbb{F}_p$ in particular, does not have.
See the discussion on the sum-product theorem at the beginning
of \S \ref{sec:sump} for a relevant instance of this. 

Equally important is the fact that the new growth bounds on non-commutative
groups are quantitatively strong ($|A^3|\geq |A|^{1+\delta}$).
Of course, one can take advantage of the structure of $\mathbb{R}$
or $\mathbb{C}$ and also give quantitatively strong estimates. A case in
point is \cite{MR2398145}, which gives (a) a simplified proof over
$\mathbb{C}$ (based on \cite{Hel08}) and (b) an early attack on $\SL_3$.




It turns out that, for applications, one needs a stronger generalization
to groups over $\mathbb{C}$: a proof of expansion requires results for
measures, and not just for sets $A$. If we are working over a finite field,
then results of type $|A^3|\geq |A|^{1+\delta}$ for finite sets $A$
and results of type $|\mu\ast \mu|_2 \leq |\mu|^{1+\delta}$ for measures
$\mu$ on $G$ (with $|\mu|_1=1$) are equivalent (as was already pointed out in
\cite{MR2415383}); however, over $\mathbb{R}$ or $\mathbb{C}$, results
for measures $\mu$ are stronger than results for sets $A$. It 
 is in the context of measures that the new ideas we are now discussing
become as crucial over infinite fields as they are over finite fields.

The first result for measures was proven
in \cite{BGSU2} for $\SU(2)$, which has the same Lie 
algebra type as $\SL_2$: the main idea is to redo the proof in \cite{Hel08},
still for finite sets, but keeping track of distances (e.g.,
where \cite{Hel08} uses that a map is injective, \cite{BGSU2} also checks that
the map does not shrink distances by more than a constant factor).

{\bf Changing the Lie type.} 
The generalization of Thm.\ \ref{thm:main08}
to $\SL_3(\mathbb{F}_p)$ \cite{HeSL3} was neither easy nor limited to
$\SL_3$ alone; it involved general work
on tori, conjugacy classes and slowly growing sets.
It was also then that
ideas from sum-product theorems ({\em pivoting}; see \S \ref{sec:corro}) 
were taken to the context of actions of groups on groups. 

Part of
the generalization to $\SL_n$ was carried out in \cite{GH1}, but, 
for instance, $\SO_n$ ($n$ large) resisted (and was an obstruction to a full
solution for $\SL_n$). Full and elegant generalizations to all finite
simple groups of Lie type (with bounds depending on the rank) were given by
Pyber and Szab\'o \cite{PS} and, independently, by Breuillard, Green
and Tao \cite{BGT}; this, of course, covered the classical groups
$\SL_n(\mathbb{F}_q)$, $\SO_n(\mathbb{F}_q)$, $\Sp_{2n}(\mathbb{F}_q)$,
with $\delta>0$ depending on $n$. (The issue of the dependence on $n$ is
important; we will discuss it in some detail later.)

The case in some sense opposite to that of simple groups is that of solvable
groups. We went over the case of a small but paradigmatic solvable
group in detail -- the affine group of $\mathbb{P}^1$ (\S \ref{subs:pivo}).
The general case of solvable subgroups of $\SL_n(\mathbb{F}_p)$ is treated
in \cite{GH2}. A clean generalization of \cite{GH2}
 to $\mathbb{F}_q$ still remains to
be done.

\subsection{High multiplicity and spectral gaps, II}\label{subs:majaro}
Now that we have proven  the main theorem (Thm.~\ref{thm:main08}), 
we may as well finish
our account of growth in linear groups by going briefly over the proof
of Thm.~\ref{thm:bg} (Bourgain-Gamburd), which gives us expanders.
We will keep an eye on how the proof (from \cite{MR2415383}) can be adapted
to general $G$.

In \S \ref{subs:hinto}, we said a pair $(G,A_G)$ gives an $\epsilon$-expander
if every $S\subset G$ with $|S|\leq |G|/2$ satisfies $|S\cup A_G S|\geq 
(1+\epsilon) |S|$. We also gave an alternative definition, based on eigenvalues;
phrased in terms of the normalized adjacency matrix 
$\mathscr{A}$ on the Cayley graph $\Gamma(G,A)$ (see (\ref{eq:hoko})), 
it states that
a graph is an $\epsilon$-expander if
the second largest eigenvalue $\lambda_1$ of $\mathscr{A}$
is at most $1-\epsilon$. Being an $\epsilon$-expander according to the
second definition implies being an $(\epsilon/2)$-expander according to the
first definition:
if $|S\cup A_G S|< (1+\epsilon/2) |S|$,
then $f:G\to \mathbb{C}$ defined by $f(x) = 1_S - |S|/|G|$ would obey
$\langle \mathscr{A} f,f\rangle> (1- \epsilon) \langle f,f\rangle$
and $\langle f, 1_G\rangle = 0$, a contradiction if 
$\lambda_1 \leq 1- \epsilon$. The two definitions are, in fact, equivalent for $|A|$ bounded; if the pair $(G,A_G)$ gives us an $\epsilon$-expander in the
first sense,
then the second largest eigenvalue $\lambda_1$ of $\mathscr{A}$ on
$\Gamma(G,A_G)$ satisfies $\lambda_1 \leq 1 - \epsilon^2/2 |A|^2$ (see, e.g.,
\cite[Thm.~1.2.3]{MR1989434} or \cite[\S 13.3.2]{MR2466937}).

We will now prove Thm.~\ref{thm:bg}, which states 
that, for $A\subset \SL_2(\mathbb{Z})$, $\langle A\rangle$ Zariski-dense,
$\{\Gamma(\SL_2(\mathbb{Z}/p\mathbb{Z}),A \mo p)\}_{\text{$p$ prime}}$ 
is a family of expander graphs.
We will be using the second definition, that is, we will show that 
$\lambda_1 \geq 1-\epsilon$, where $\epsilon>0$ depends only on $A$, not on $p$.


We recall at this point what was known before on expander families in
$\SL_2(\mathbb{Z}/p\mathbb{Z})$, even though we will not use those previous
results. As we said in \S \ref{subs:hinto}, it was known that the existence of 
the Selberg spectral gap \cite{MR0182610} on the surface $\Gamma(N)
\backslash \mathbb{H}$ implies that all pairs $(\SL_2(\mathbb{Z}/p\mathbb{Z}),
A \mo p)$ with
\begin{equation}
A = \left\{ \left(\begin{array}{cc} 1 &1\\0 &1\end{array}\right),
\left(\begin{array}{cc} 1 &0\\1 &1\end{array}\right)
\right\}\end{equation}
are $\epsilon$-expanders for some
fixed $\epsilon>0$; the same holds for all other
$A\subset \SL_2(\mathbb{Z})$ such that $\langle A\rangle$ is of finite index in
$\SL_2(\mathbb{Z})$.
The implication is not, from today's perspective,
particularly difficult, but tracing the first time it was remarked is 
non-trivial; what follows is a preliminary attempt at a one-paragraph history of the implication, since this
seems to be a matter of general interest in the field.

The implication was shown for compact quotients of $\mathbb{H}$ in Buser's work 
\cite{MR505920}.
See also Brooks
\cite{MR840402} (still for the compact case) and \cite{MR894565}
(non-compact case, in terms of the {\em Kazhdan constant}) and Burger
\cite{Burger} (compact case, in terms of eigenvalues). 
What is a little harder to pinpoint is the first proof for the non-compact
case (in particular, $\Gamma(N)\backslash \mathbb{H}$)
 in terms of eigenvalues. (At least some proofs for the compact case do
generalize to the non-compact case -- see, e.g., \cite[App.~A]{MR2922374},
based on Burger's approach -- but this seems not to have been obvious at first.)
A. Lubotzky and P. Sarnak
(private communication) state that the work leading to \cite{MR963118} 
was
originally centered on $\SL(\mathbb{Z}/p\mathbb{Z})$ and
$\Gamma(N) \backslash \mathbb{H}$ and, in particular, showed the 
correspondence in this (non-compact) case.
Thanks are due to them and to E. Kowalski for several references.

Let us go back to the proof of Thm.~\ref{thm:bg}, which is a much more
general statement, as it works for $A\subset \SL_2(\mathbb{Z})$ essentially
arbitrary. (We recall that the only condition is that $\langle A\rangle$
be Zariski-dense in $\SL_2(\mathbb{Z})$; this is roughly what is necessary
for $A \mo p$ to generate $\SL_2(\mathbb{Z}/p\mathbb{Z})$ in the first place.) 

Let us first sketch, as we promised, a proof of the weaker statement
that $G_p$ has logarithmic diameter $O_A(\log |G|)$ with respect to $A_p$.
Let us assume, for simplicity, that $A$ generates a free group, i.e.,
any two products
$a_1 \dotsb a_k$, $a_1' \dotsb a_{k'}'$ of elements of $A\cup A^{-1}$
are distinct, unless they are equal for the trivial reason of having
the same reduction (e.g., $x x^{-1} y z = y w^{-1} w z$, since both reduce to
$y z$ purely formally). (The condition that $A$ generates a free group
is fulfilled both by (\ref{eq:hosen}) and by (\ref{eq:pantalon}); in general,
the argument works because, by the Tits alternative, a subset of $\SL_n$
generating a Zariski-dense group must contain
a pair of elements generating a free group. For $n=2$, as we will soon see,
the Tits alternative is easy to prove by the use of the free group
$\Gamma(2)$.)

If $k$ and $k'$ are $< \log_c(p/2)$ (where
$c = 2\cdot 3 = 6$ for $A$ is in (\ref{eq:pantalon}), then 
$a_1 \dotsb a_k \neq a_1' \dotsb a_{k'}'$ implies
$a_1 \dotsb a_k \not\equiv a_1' \dotsb a_{k'}' \mo p$, simply because
the entries of the matrices involved are $<p/2$ in absolute value, and
thus cannot be congruent without being equal.
Hence, for 
$k = \lfloor \log_c p/2\rfloor$, $|(A_p \cup A_p^{-1} \cup e)^k|$ is
already quite large ($\geq (2|A|-1)^k \geq p^{\delta_c}$, $\delta_c\gg
0$).\footnote{This bound on growth in the free group is trivial:
given a word ending in, say, $x$, we can choose to prolong it by any
element
of $A\cup A^{-1}$ other than $x^{-1}$. Note, however, that obtaining
a result like Theorem~\ref{thm:main08} for the free group is far from
trivial; Theorem~\ref{thm:main08} (and \cite{MR2398145}) imply such a result,
but the first direct proof is due to Razborov \cite{MR3152939}, who proved
a strong bound $|A^3|\geq |A|^2/(\log |A|)^{O(1)}$ for any finite
subset $A$ of a free group on at least two elements.}
 We can then apply
Theorem \ref{thm:main08} a bounded number of times, and conclude that
the diameter of $G_p$ with respect to $A_p$ is in fact $\ll \log p
\ll \log |G|$, i.e., logarithmic.


As we will now see, the proof of Bourgain and Gamburd's expansion theorem
is considerably more difficult.
\begin{proof}[Proof of Thm.~\ref{thm:bg}]
Let $S = A \cup A^{-1} \cup \{e\}$. Let $G=\SL_2$. Let $\mu$ be the measure on
$G_p = \SL_2(\mathbb{Z}/p\mathbb{Z})$ given by
\[\mu(x) = \begin{cases} \frac{1}{|S|} &\text{if $x\in S$,}\\
0 &\text{otherwise.}\end{cases}\]

We consider the convolutions $\mu^{(k)} = \mu \ast \mu \ast \dotsc 
\ast \mu$. We will see how $|\mu^{(k)}|_2$ decreases as $k$ increases.
This happens very quickly at first ({\em stage 1}). It then goes
on happening quickly enough ({\em stage 2}), thanks to Thm.~\ref{thm:main08}
(applied via Prop.~\ref{prop:flatlem}, the Bourgain-Gamburd
``flattening lemma''). Once $|\mu^{(k)}|_2$ is quite small (not much larger
than $1/|G_p|$, which is the least it could be), the proof can be finished
off by an argument from \cite{SarnakXue}, based on the same
high-multiplicity phenomenon that was exploited in \S \ref{subs:highmult}.

{\em Stage 0: Reduction to $\langle A\rangle$ free.} 
For $G=\SL_2$, we can (as in \cite{MR2415383}) define
$H = \Gamma(2) = \{g\in G(\mathbb{Z}): g\equiv I \mo 2\}$; now,
 $H$ is both free and of finite index in $G(\mathbb{Z})$; hence
$\langle A\rangle \cap H$ is free (since, by the Nielsen-Schreier theorem, every subgroup
of a free group is free), Zariski-dense, and generated by a set $A'\subset \langle A\rangle$
of bounded size (Schreier generators). We can thus replace $A$ by $A'$
(at the cost of at most a constant factor -- depending on $A$ and $A'$ --
in the final bounds),
and assume from now on that $\langle A\rangle$ is free.

(For general $G$, the task is much more delicate,
since such a convenient $H$ does not in general exist, and also because the
``concentration in subgroups'' issue we will discuss below requires
stronger inputs to be addressed successfully -- Zariski density no longer
seems enough (given current methods). See \cite{SGV} for a general solution. 
An approach via products of random matrices is also possible
\cite{MR2443926}, \cite{MR2538500}.)

{\em Stage 1.} {\em Decrease in $|\mu^{(k)}|_2$ for $k \ll \log |G|$, due
to freedom.}
We can now assume that $\langle A\rangle$ is a free group on
$r\geq 2$ elements. By the argument we went over in the introduction
(shortly before the statement of Thm.~\ref{thm:bg}), there is
a constant $c$ depending only on $A$ such that two words on $A$ of length 
$k\leq c \log p$ reduce to the same element of $G(\mathbb{Z}/p\mathbb{Z})$
only if they give the same element of $G(\mathbb{Z})$; since $\langle A\rangle$
is free, this can happen only if they have the same reduction
(e.g., $x x^{-1} y z = x w^{-1} w z$). 
Thus, for instance, 
\[\begin{aligned}
\mu^{(k)}(e) 
= \frac{|\text{words of length $k$ reducing to the identity}|}{
r^{k}},\end{aligned}\]
where $\mu^{(k)} = \mu \ast \mu \ast \dotsb \ast \mu$ ($k$ times).
Hence, Kesten's bound on the number of words of given length reducing to
the identity \cite{MR0109367}
 gives us that, for any $\epsilon>0$,
\[\mu^{k}(e) \ll_\epsilon
\left(\sqrt{\frac{2r-1}{r}}+\epsilon\right)^{k},\]
and so, for $k = \lfloor c \log p\rfloor$,
\[\mu^{(k)}(e) \ll \frac{1}{p^{\eta}},\]
where $\eta>0$ depends only on $c$, and thus only on $A$.
(Note that, if $r\geq 3$, the simple bound
\cite[Lem.~2]{BroSha} can be used instead of Kesten's result.)

It turns out that, using the fact that $\langle A\rangle$ is free,
we can show not just that $\mu^{(k)}(e)$ is small, but that $\mu^{(k)}(G')$ is 
small for any proper subgroup $G'$ of $G$. For $G=\SL_2$, this is
relatively straight-forward:
every proper subgroup $G'$ of 
$G_p=\SL_2(\mathbb{Z}/p\mathbb{Z})$ is almost solvable, i.e., contains
a solvable subgroup $G''$ of bounded index. It is enough to show that
$\mu^{(2k)}(G'')$ is small (as this implies immediately that $\mu^{(k)}(G')$
is small, by pigeonhole). Because we are in $\SL_2$, $G''$ is not just solvable but $2$-step
solvable, i.e., any elements $g_1,g_2,g_3,g_4\in G''$ must satisfy 
\begin{equation}\label{eq:pikovaja}\lbrack 
\lbrack g_1,g_2\rbrack, \lbrack g_3,g_4\rbrack \rbrack = e.\end{equation}
By the same idea as before, for $k\leq c \log p$, $c$ small enough, this is possible only
if $g_1,g_2,g_3,g_4$ are projections $\mo p$ of elements of
$\langle A\rangle$ that also satisfy (\ref{eq:pikovaja}).
However, any set $S$ of words of length $\leq l$ in a free group such that
all $4$-tuples of elements of $S$ satisfy (\ref{eq:pikovaja}) must be of size $\leq l^{O(1)}$, by a simple argument \cite[Prop.~8 and Lem.~3]{MR2415383} based
on the fact that 
the centralizer of a non-trivial element in a free group is cyclic:
the centralizer is a free group (being a subgroup of a free group) but it cannot be of rank
$\geq 2$, as it satisfies a non-trivial relation.
 Hence $\mu^{(2 k)}(G'')$, and thus $\mu^{(k)}(G')$, is indeed small:
\[\mu^{(k)}(G') \ll \mu^{(2k)}(G'') \ll \frac{(2k)^{O(1)}}{r^{2k}} 
\ll \frac{1}{p^{\eta}},\]
where $\eta>0$ depends only on $c$, and thus only on $A$.

(For general $G$, showing that there is no concentration in a
proper subgroup $G'$ is a 
much more delicate matter. A fully general solution was given by \cite{SGV}
(``escape in mass from subvarieties'').)

{\em Stage 2.} {\em Decrease in $\mu^{(k)}$ for $k$ medium-sized, due to
$|A^3|\geq |A|^{1+\delta}$.} We are in the case in which one of the main results in this 
survey (Thm.~\ref{thm:main08}) will be applied (via Prop.~\ref{prop:flatlem}).
Consider $\mu^{(k)}$, $\mu^{(2 k)}$, $\mu^{(4 k)}$, etc.
At each step, we apply Prop.~\ref{prop:flatlem} (the flattening lemma)
with $K = p^{\delta'}$, $\delta'>0$ to be set later. If (\ref{eq:gorto})
fails every time, we obtain $|\mu^{(2^r k)}|_2^2 < 1/|G|$ after $r=O_{\delta',\eta}(1)$
steps; we then go to stage 3. 

Suppose (\ref{eq:gorto}) holds for some $k' = 2^j k$, $j\ll_{\delta',\eta} 1$.
Then Prop.~\ref{prop:flatlem} gives that there is
a $p^{O(\delta')}$-approximate subgroup $H\subset G$ of size 
$\ll p^{O(\delta')}/|\mu^{(k')}|_2^2$ and an element $g\in G$ such that $\mu^{(k')}(Hg)
\gg p^{-O(\delta')}$. In particular, $|H^3|< |H|^{1+O(\delta'/\eta)}$.
Choosing $\delta'>0$ small enough, we get a
contradiction to Thm.~\ref{thm:main08}, unless $|H|\geq |G|^{1-O(\delta)}$
(where we can make $\delta$ as small as we want) or $H$ is contained
in a proper subgroup $G'$ of $G$. 

If $|H|\geq |G|^{1-O(\delta)}$, then
$|H|\ll p^{O(\delta')}/|\mu^{(k')}|_2^2$ implies that $|\mu^{(k')}|_2^2 \ll 1/|G|^{1-\delta-\delta'}$, and we go to stage 3. Assume, then, that $H$ is contained in a proper
subgroup $G'$ of $G$. Then $\mu^{(k')}(G' g)\gg |G|^{-O(\delta')}$.
This implies that
$\mu^{(k)}(G' g)\gg |G|^{-O(\delta')}$ (simply because $\mu^{(k')} = \mu^{(k)} \ast
\mu^{(k'-k)}$), i.e., $\mu^{(k)}$ is concentrated in a subgroup, in
contradiction (once we set $\delta'$ small enough) with what we proved
in Stage 2.

{\em Stage 3.} {\em Amplifying the effect of $|\mu^{(k)}|_2$ small by
eigenvalue multiplicity.} We have got to an $\ell =k' \leq 2^{O_{\delta',\eta}(1)} k
\ll_{A,\delta} \log p \ll \log |G|$ such that $|\mu^{(\ell)}|_2^2 \ll
1/|G|^{1-\delta}$, where $\delta>0$ is as small as we want. This is still weaker
than a bound on $\ell^2$ mixing time (meaning
 an $\ell$ such that $|\mu^{\ell}-1/|G||_2^2
\ll \epsilon/|G|$), which is itself, in general, weaker than expansion. 
Let us see, however, how to get to expansion by using the high multiplicity
of eigenvalues (\S \ref{subs:highmult}). (This is as in Sarnak-Xue
\cite{SarnakXue}.)  
 The trace of $\mathscr{A}^{2\ell}$ is, on the one hand,
$|G| |\mu^{\ell}|_2^2$ (by definition of trace, since the probability of $x$
returning to $x$ after $2k$ steps of a random walk is $|\mu^\ell|_2^2$), 
and, on the other, $\sum \lambda_i^{2\ell}$ (sum of eigenvalues). Hence 
\[m_1 \lambda_1^{2\ell} \leq \sum \lambda_i^{2\ell} = |G| |\mu^\ell|_2^2 \ll \frac{|G|}{|G|^{1-\delta}} = |G|^{\delta},\]
where $m_1$ is the multiplicity of $\lambda_1$. As we saw in \S \ref{subs:highmult}, $m_1 = (p-1)/2 \gg |G|^{1/3}$. Thus, $\lambda_1^{2\ell} \gg |G|^{\delta-1/3}$.
 We set $\delta<1/3$ ($\delta=1/6$, say) and obtain that 
$\lambda_1 \leq 1 - \epsilon$, where $\epsilon>0$ depends only on $A$ (and $\delta$,
which is now fixed).
\end{proof}

Applications have called for generalizing Thm~\ref{thm:bg} in two
directions. One is that of changing the Lie type. Here the first
step was taken by Bourgain and Gamburd themselves \cite{MR2443926};
a fully general statement for all perfect $G$ 
is due to Varj\'u and Salehi-Golsefidy \cite{SGV}. (We have already
discussed one of the main issues involved in a generalization, namely, 
avoiding concentration in subgroups.) The other kind of generalization
consists in changing the ground ring. For many applications, the most important
change turns out not to be changing $\mathbb{F}_p$ for $\mathbb{F}_q$, but
changing $\mathbb{Z}/p \mathbb{Z}$ for $\mathbb{Z}/d\mathbb{Z}$. (This
is needed for the the {\em affine sieve} \cite{MR2587341}, one
of the main ways in which results in the area get applied nowadays.)
For $\SL_2$ and $d$ square-free, this was done in \cite{MR2587341};
\cite{MR2862040} and \cite{SGV}
solved the problem for $G$ general and $d$ square-free; \cite{BGSU2} 
addressed
$\SL_2$ and $d=p^k$, and \cite{MR2538500} did the same for $\SL_n$ and $d=p^k$.
So far the only result for general moduli $d$ is \cite{MR2897695},
which treats $\SL_n$; as of the time of writing, the case of $G$ general and $d$ general is not yet
finished.

\section{Growth in permutation groups}\label{sec:pergro}

\subsection{Introduction}
Our aim now will be to give some of the main ideas in the proof of 
quasipolynomial diameter for all Cayley graphs of the symmetric and
alternating groups (Thm.~\ref{thm:hs}):
\[\diam(\Gamma(G,A)) \ll \exp \left((\log n)^{O(1)}\right)\]
for any transitive group $G$ on $n$ elements and any
$A\subset G$ generating $G$. (Here $O(1)$ is just $4+\epsilon$.)

The proof uses much of the foundational material we gave in \S \ref{sec:sojor}.
The structure of the proof is very different from that for linear algebraic
groups, however. In particular, we do not have access to dimensional bounds
(since there is no clear meaning to {\em dimension} in a permutation group)
or to escape-from-subvarieties arguments (essentially for the same reason).

Additional tools come from two sources. There were existing results on diameters
of permutation groups: among those that were particularly useful, 
\cite[Thm. 1.1]{BS92} reduces the problem to that for $\langle A\rangle =
\Alt(n)$ (intuitively the hardest case) and \cite{BBS04} proves a polynomial
diameter bound provided that $A$ contains at least one element other than the
identity with small support. Note that \cite{BBS04} already uses the fact 
that even a very short random walk in $\Sym(n)$ takes an element of
$\Omega = \{1,2,\dots,n\}$ to any other element with almost uniform 
distribution.

Before \cite{MR3152942}, the strongest bound on the diameter of permutation
groups was that of Babai-Seress \cite{BS88}, who showed that, for any
permutation group $G$ on $n$ elements, and any $A\subset G$ generating $G$,
\[\diam(\Gamma(G,A)) \leq \exp((1+o(1)) \sqrt{n \log n}).\]
While this is much weaker than the bound in Thm.~\ref{thm:hs}, 
it does not assume transitivity (and indeed it can be tight for non-transitive
groups). Moreover, the proof (see also \cite{MR894827} and \cite{BLS87}) contains an idea
that was useful in \cite[\S 3.6]{MR3152942} -- namely, that prefixes can be played
with to give the existence of non-identity elements with small support
\cite[Lemma 1]{BS88}.

Another key source of ideas came from existing work on Classification-free
results on subgroups of $\Sym(n)$. The Classification of
Finite Simple Groups is a result whose first proof spanned many volumes;
its acceptance was gradual -- even the date of its completion, at some point
in the 80s, is unclear. Thus, there was an interest in what \cite{Bab82}
called ``intelligible proofs of the asymptotic consequences of the
Classification Theorem''. Work in this direction includes
\cite{Bab82}, \cite{Pyb93} and \cite{LP}.

 Precisely because this work
had combinatorial, relatively elementary bases, it turned out to be very robust:
that is, these are results on the size of subgroups that can be generalized
to any subsets that grow slowly under multiplication. (The basic idea here
is as in the orbit-stabilizer theorem and its consequences
(\S \ref{subs:joro}): these are
bounds on the size of a subgroup $H$ that are based on maps or processes that
multiply $H$ by itself a few times ($\leq k$ times, say); 
thus, if instead of having a subgroup $H$
we have a set $A$, we still have a result -- just one where $H$ gets replaced
sometimes by $A$ and sometimes by $A^k$, as we shall see.)

Just as a generalization of \cite{LP} played an important role in both
\cite{MR2436141} and \cite{BGT}, a generalization of \cite{Bab82} and 
\cite{Pyb93} plays an important role in \cite{MR3152942}. What \cite{MR3152942} uses
is not the final result
 in \cite{Bab82}, but rather an intermediate result,
the ``splitting lemma''. This is a result based on what is called the 
{\em probabilistic method} in combinatorics (generally, as in \cite{MR1885388},
traced back to Erd\H{o}s). This method is based on the observation that, if we 
show that something happens with positive probability, then it happens 
sometimes; thus, if we impose a convenient distribution (often the 
uniform one) on some initial objects, and we obtain that they then satisfy
a certain property with positive probability, we have shown that a configuration
of objects satisfying the property exists. The objects in \cite{Bab82} are
elements of a group $H$. Now, we, in \cite{MR3152942}, do not have the right to choose 
elements of $H$ at random; to do so would be to assume what we are trying to 
prove, namely, a small bound on the diameter. Instead (as in \cite{BBS04})
we mimic the effects of a uniform distribution by means of a random walk;
since the set $\{1,2,\dotsc,n\}$ being acted upon is small, a short random
walk is enough to give a distribution very close to uniform. 

The proof in \cite{MR3152942} contains many other elements; a full outline 
is given in \cite[\S 1.5]{MR3152942}. Here, let us focus on a crucial part: the 
generalization of Babai's
splitting lemma, and its application by means of the orbit-stabilizer theorem
to create elements of $A^k$ in small subgroups of $\Sym(n)$. We will then
look at a different part of the proof, giving a result of independent interest
(Prop.~\ref{prop:crat})
 on small generating sets. This will demonstrate how different random processes
-- not just random walks on graphs -- can be used to give explicit results
on growth and generation.

\subsubsection{Notation}
As we said earlier, we will follow here the sort of notation that is
current in the literature on permutation groups: actions $G\curvearrowright X$
are by default on the right, $x^g$ means $g(x)$ (so that $x^{gh} = (x^g)^h$),
$x^A$ is the orbit of $x$ under $A\subset G$, and 
$A_x$ means $\Stab(x) \cap A$.
There are two different kind of stabilizers of a set $S\subset X$: 
the {\em setwise stabilizer}
\[A_S = \{g\in A: S^g = S\}\]
and the {\em pointwise stabilizer}
\[A_{(S)} = \{g\in A: x^g = x\;\; \forall x\in S\}.\]
(The notation here is as in \cite{DM} and \cite{MR1970241}, not as in
\cite{MR0183775}.)

\subsection{Random walks and elements of small support}\label{subs:gotr}

We will start with some basic material on random walks. We will then be able
to go briefly over the proof of \cite{BBS04}.

Let us start by defining our terms. We will work with a
directed multigraph $\Gamma$, that is, a graph where the edges are
directed (i.e., they are arrows) and may have multiplicity $>1$. 
(The setting we will work in now
is more general than that of Cayley graphs.) We assume that $\Gamma$
is {\em strongly connected} (i.e., there is a path respecting the arrows
between any two points in the vertex set $V(\Gamma)$), {\em regular of
valency (or degree) $d$} (meaning that there are $d$ arrows (counted with multiplicity)
going out of every vertex in $V(\Gamma)$) and {\em symmetric} (i.e.,
the number of arrows from $x$ to $y$ is the same as the number of arrows
from $y$ to $x$, counting with multiplicity in both cases). 

We will study a lazy random walk: a particle moves from vertex to vertex,
and at each point in time, if it is at a vertex $x$, and the arrows
going out of $x$ end at the vertices $x_1$, $x_2$,\dots, $x_d$
(with repetitions possible), the particle decides to be lazy (i.e., 
stays at $x$) with probability $1/2$,
and moves to $x_i$ with probability $1/2d$. (Studying a lazy random
walk is a well-known trick used to avoid the possible effects of large
negative eigenvalues of the adjacency matrix.)

Let $x,y \in V(\Gamma)$. We write $p_k(x,y)$ for the probability
that a particle is at vertex $y$ after $k$ steps of a lazy random
walk starting at $x$. For $\epsilon>0$ given, the {\em $(\ell_\infty,\epsilon)$-mixing 
time} is the least $k$ such that \[\left|p_k(x,y)-\frac{1}{
|V(\Gamma)|}
\right|_\infty
\leq \frac{\epsilon}{|V(\Gamma)|}.\]
(A {\em mixing time} is a time starting at which the outcome of a random walk
is very close to a uniform distribution; the norm (e.g. $\ell_\infty$)
and the tolerance (namely, $\epsilon$) have to be specified (as they are here)
for ``very close'' to have a precise meaning.)

As before, the normalized adjacency matrix $\mathscr{A}$ is the
operator taking a function $f:V(\Gamma)\to \mathbb{C}$ to the function
$\mathscr{A}f:V(\Gamma)\to \mathbb{C}$ defined
by
\[\mathscr{A}f(x) = \frac{1}{d}
\sum_{\text{arrows $v$ with $\text{tail}(v) = x$}} f(\text{head}(v)).\]
(An arrow goes from its tail to its head.) 
Let $f_0, f_1, f_2,\dotsc$ be a full set of eigenvectors corresponding to the
eigenvalues $1=\lambda_0\geq \lambda_1\geq \lambda_2\geq \dotsc$ 
of $\mathscr{A}$. If $\Gamma$ is regular and symmetric, then $\mathscr{A}$
is a symmetric operator, and so all $\lambda_i$ are real,
and we can also assume all the $f_i$ to be real-valued and orthogonal to
each other.
The following fact is well-known. 
\begin{lem}\label{lem:chudo}
Let $\Gamma$ be a connected, regular and symmetric multigraph of
valency $d$ and with $N$ vertices. Then the $(\ell_\infty,\epsilon)$-mixing 
time is at most $N^2 d \log(N/\epsilon)$.
\end{lem}
The proof contains two steps: a trivial bottleneck bound gives a lower
bound on the eigenvalue gap $\lambda_0-\lambda_1 = 1-\lambda_1$, and
a lower bound on the eigenvalue gap gives an upper bound on the
mixing time. 
\begin{proof}
Let $f_1$ be an eigenvector corresponding to $\lambda_1$; since
$f_1$ is orthogonal to the constant function $f_0$, the maximum $r_+$ and
the minimum $r_-$ of $f_1(x)$ obey $r_+>0>r_-$. By pigeonhole, there is an 
$r\in (r_-,r_+\rbrack$ such that there are no $x\in V(\Gamma)$ with
$f_1(x) \in (r-\eta,r)$, where $\eta = (r_+-r_-)/N$. Let $S = \{x\in V(\Gamma):
f_1(x)\geq r\}$. Clearly, $S$ is neither empty nor equal to all of 
$V(\Gamma)$.

Since $\Gamma$ is connected, there is at least one $x\in S$ with at least one
arrow starting at $x$ and ending outside $S$. (This is the same as saying
that the {\em bottleneck} of a connected graph is $\geq 1/Nd$.) Hence
\begin{equation}\label{eq:givus}
\sum_{x\in S} \mathscr{A}f_1(x) \leq - \frac{\eta}{d} + \sum_{x\in S} f_1(x).
\end{equation}
Again by $\langle f_0,f_1\rangle=0$, the average of $f_1(x)$ over $S$ is $>0$;
trivially, it is also $<r_+$. Thus, (\ref{eq:givus}) gives us
\[\sum_{x\in S} \mathscr{A}f_1(x) \leq \left(1 - \frac{1}{N |S| d}\right)
\sum_{x\in S} f_1(x).\]
Therefore, $\lambda_1 \leq 1 - 1/N |S| d \leq 1 - 1/N^2 d$.

This implies the desired bound
on the mixing time (exercise\footnote{See the proof
of \cite[Lem.~4.1]{MR3152942} (or any of many other sources, e.g., 
\cite[Thm.~5.1]{MR1395866}) for a solution. It is easy to do this 
suboptimally
and obtain an extra factor of $|p_k(x,y)-1/N|\leq N \lambda_2^k$ instead of
$|p_k(x,y)-1/N|\leq \lambda_2^k$.}) by 
an idea already used in \S \ref{subs:highmult}:
every step of the random walk multiplies the vector describing the 
probability distribution of the particle by $(\mathscr{A}+I)/2$, and so
anything orthogonal to a constant function gets multiplied by $\lambda_1
= 1 - 1/N^2 d$ (or less) repeatedly.
\end{proof}

Lemma \ref{lem:chudo} may look weak, but it is actually quite useful for $N$
small, i.e., graphs with small vertex sets. 
When we work with a permutation group $G\leq \Sym(n)$, we may not
have all the geometry we had at hand when working with linear algebraic
groups, but we do have something else -- an action on the small
set $\{1,2,\dotsc,n\}$ (and tuples thereof); that action gives rise to
graphs with small vertex sets, allowing us to use Lemma \ref{lem:chudo}.

First, we prove that we can mimic the uniform distribution on $k$-tuples
by relatively short random walks. This is just as in \cite[\S 2]{BBS04}.

A {\em $k$-transitive} subgroup of $\Sym(n)$ is one whose action on the
set of $k$-tuples of distinct elements of $\{1,2,\dotsc,n\}$ is transitive.
\begin{lem}\label{lem:bulgo}
Let $G$ be a $k$-transitive subgroup of $\Sym(n)$. Let $A$
be a set of generators of $G$. Then there is a subset $A_0\subset A
\cup A^{-1}$ such that the following holds.

For $\epsilon>0$ arbitrary, 
for any $\ell \geq 4 n^{2k+1} \log(n^k/\epsilon)$
and for any $k$-tuples $\vec{x}$, $\vec{y}$ of distinct elements of
$\{1,2,\dotsc,n\}$, 
the probability that the outcome $g\in \langle A_0\rangle$
 of a lazy random walk of length $\ell$ (on the graph
$\Gamma(\langle A_0\rangle,A_0)$, starting at $e$) take $\vec{x}$ to $\vec{y}$
is at least $(1-\epsilon) (n-k)!/n!$ and at most $(1+\epsilon) (n-k)!/n!$.
\end{lem}
The number of $k$-tuples of distinct elements of $\{1,2,\dotsc,n\}$ is,
of course, $n!/(n-k)!$.
\begin{proof}
Since $A$ may be large, and it will be best to work with a generating set
that is not very large, we start by choosing a subset of $A$
that still generates $G$. This we do simply by choosing an element 
$g_1\in A$, and then an element $g_2\in A$ such that $\langle g_1\rangle
\lneq \langle g_1,g_2\rangle$, and then a $g_3\in A$ such that
$\langle g_1,g_2\rangle \lneq \langle g_1,g_2,g_3\rangle$, etc.,
until we get $\langle g_1,g_2,\dotsc,g_r\rangle = \langle A\rangle = G$
($r\geq 1$).
Since the longest subgroup chain in $\Sym(n)$ is of length $\leq 2n-3$ 
\cite{MR860123},\footnote{The trivial bound is 
$(\log n!)/\log 2$: in a subgroup chain
$H_1 \lneq H_2 \lneq \dotsc \lneq H_k$,
we have $|H_2|\geq 2 |H_1|$, $|H_3|\geq 2 |H_2|$, etc., simply because
the index of a proper subgroup of a group is always
$\leq 2$.} we see that $r\leq 2n-2< 2n$. Let
$A_0 = \{g_1,g_1^{-1},\dotsc,g_r,g_r^{-1}\}$.

Now define
the multigraph $\Gamma$ by letting the set of vertices consist of all
$k$-tuples of distinct elements of $\{1,2,\dotsc,n\}$; draw an arrow
between $\vec{z}=(z_1,z_2,\dotsc,z_k)$ and $(z_1^a,z_2^a,\dotsc,z_k^a)$
for every vertex (i.e., $k$-tuple) $\vec{z}$ and every $a\in A_0$.
(This is, of course, a Schreier graph.)
Finish by applying Lem.~\ref{lem:chudo}.
\end{proof}

We will now see how to adapt the {\em probabilistic method}
using Lem.~\ref{lem:bulgo} to approximate
a uniform distribution by a short random walk.

\begin{prop}[\cite{BBS04}]\label{prop:lalka}
Let $A\subset \Sym(n)$ generate a $3$-transitive subgroup of $\Sym(n)$.
Let $g\in A^{\ell_0}$, $\ell_0\geq 1$ arbitrary. Assume that
$0<|\supp(g)|<n$. Then, for any $\epsilon>0$,
there is an element $g' \in (A\cup A^{-1})^{\ell+4\ell_0}$, 
$\ell\ll n^7 \log (n/\epsilon)$,
$g\ne e$,
such that \[\supp(g')\leq 
3 + 3 (1+O(\epsilon)) |\supp(g)|^2/n,\]
where the implied constant is absolute.
\end{prop}
The conclusion is non-trivial only when $\supp(g)<n/3$.
\begin{proof}
Given $\sigma\in \Sym(n)$ and $x\in \{1,2,\dotsc,n\}$, 
let $h = \sigma^{-1} g \sigma$; thus, $\supp(h) = (\supp(g))^\sigma$.
Write $S$ for $\supp(g)$.
 When is 
$x$ in the support of
the commutator
 $\lbrack g,h \rbrack = g^{-1} h^{-1} g h$? 
(Defining the commutator
in this way, rather than as $g h g^{-1} h^{-1}$, 
is standard in the study of permutation groups.)
There are three possibilities:
\begin{enumerate}
\item $x\in S$ and $x^{g^{-1}}\in \supp(h)$, i.e.,
$x^{g^{-1}}\in S \cap S^{\sigma}$;
\item $x\in S$, $x^{g^{-1}}\not\in \supp(h)$ and $x\in \supp(h)$, and
so, in particular, $x\in S \cap S^{\sigma}$.
\item $x\notin S$, $x\in \supp(h)$ and $x^{h^{-1}}\in S$,
and so, in particular, $x^{h^{-1}} \in S \cap S^{\sigma}$.
\end{enumerate}
Thus, $\supp(\lbrack g,h\rbrack)$ is contained in
\[(S\cap S^{\sigma}) \cup
(S\cap S^{\sigma})^g \cup
(S\cap S^{\sigma})^h.\]

Now let $\sigma\in (A\cup A^{-1})^{\ell'}$ be the outcome of a lazy random walk
of length $\ell'=\lceil 4 n^{2k+1} \log(n^k/\epsilon)\rceil$
, $k=3$. Lemma \ref{lem:bulgo} tells us that, for any
$x,x'\in \{1,2,\dotsc,n\}$, $\sigma$ will take
$x$ to $x'$ with probability between $(1-\epsilon)/n$ and $(1+\epsilon)/n$.
Since expectation is additive, it follows that, for every $S\subset
\{1,2,\dotsc,n\}$,
\begin{equation}\label{eq:smet}\begin{aligned}
\mathbb{E}(|S\cap S^\sigma|) &= \sum_{x'\in S} \Prob(x'\in S^{\sigma})
= \sum_{x'\in S} \sum_{x\in S} \Prob(x^{\sigma} = x') \\
&=
 \sum_{x'\in S} \sum_{x\in S} \frac{1+O(\epsilon)}{n} = 
(1+O(\epsilon)) \frac{|S|^2}{n}.\end{aligned}\end{equation}

Writing $S = S$, we see that
\[\begin{aligned}\mathbb{E}(|\supp(\lbrack g,h\rbrack)|) &\leq
\mathbb{E}(|S\cap S^{\sigma}|) +
\mathbb{E}(|(S\cap S^{\sigma})^g|) +
\mathbb{E}(|(S\cap S^{\sigma})^h|)\\
&= 3 \cdot \mathbb{E}(|S\cap S^{\sigma}|)
\leq 3 (1+O(\epsilon)) \frac{|S|^2}{n}
= 3 (1+O(\epsilon)) \frac{|S|^2}{n}
.\end{aligned}\]

Now we could conclude that there exists a $\sigma\in (A\cup A^{-1})^{\ell'}$
such that $|\supp(\lbrack g,h\rbrack)|$ is at most
$3 (1+O(\epsilon)) |S|^2/n$. We forgot to take care of one detail,
however: $\lbrack g,\sigma^{-1} g \sigma\rbrack$ could be the identity.
Fortunately $\ell'$ is large enough that Lemma \ref{lem:bulgo} assures
us that, even if we specify that $y_i^{\sigma} = y_i'$ for some $y_i,y_i'\in
\{1,2,\dotsc, n\}$ ($i=1,2$; $y_1\ne y_2$, $y_1'\ne y_2'$), 
the probability that $x^{\sigma} = x'$ for
$x,x'\in \{1,2,\dotsc,n\}$ ($x\ne y_i$, $x'\ne y_i'$ for $i=1,2$) is 
$(1+O(\epsilon))/(n-2) = (1+O(\epsilon))/n$. (This is why we let $k=3$ and not $k=1$.)
We choose $y_1\in S$, $y_1'\in S$,
$y_2\notin S$, $y_2' = (y_1')^{g^{-1}}$. Then (by a brief computation)
$(y_1')^{\lbrack g,h\rbrack} \ne y_1'$, and so $\lbrack g,h\rbrack$ is not the
identity.

The analysis goes on as before, except we obtain
\[\mathbb{E}(|\{x\in S\cap S^{\sigma}|) \leq 1 + (1+O(\epsilon)) 
\frac{|S|^2}{n} \]
(or $2+(1+O(\epsilon)) |S|^2/n$ for general $S$; we are using the
fact that $y_2$ and $y_2'$ are never both in $S$ in our case) and so
\[\mathbb{E}(|\supp(\lbrack g,h\rbrack)|) \leq 
3 + 3 (1+O(\epsilon)) \frac{|S|^2}{n}.\]

Thus, there is a $\sigma\in (A\cup A^{-1})^{\ell'}$ such that
$g' = \lbrack g,h\rbrack = \lbrack g,\sigma^{-1} g \sigma\rbrack$ has
support $\leq 3 + 3 (1+O(\epsilon)) |S|^2/n$.
\end{proof}

\begin{cor}[\cite{BBS04}
]\label{cor:sto}
Let $A\subset \Sym(n)$ ($A=A^{-1}$) 
generate a $3$-transitive subgroup $G$ of $\Sym(n)$.
Assume there is a $g\in A$, $g\ne e$, with $|\supp(g)| \leq (1/3-\epsilon) n$,
$\epsilon>0$. Then 
\[\diam(\Gamma(G,A)) \ll_\epsilon n^{c_1} (\log n)^{c_2},\]
where $c_1=8$ and $c_2$ is absolute.
\end{cor}
\begin{proof}
Apply Prop.~\ref{prop:lalka}, then apply it again with $g'$ instead of $g$, and
again and again.
After $O(\log \log n)$ steps, we will have obtained an element
$g_1 \in A^{\ell_1}$, $\ell_1\ll_{\epsilon} n^7 (\log n)^{c}$, such that
$g_1\ne e$ and $|\supp(g_1)|\leq 3$. A brief argument suffices to show that
$A^{n^3}$ 
acts $3$-transitively (i.e., for any two $3$-tuples of distinct elements
of $\{1,2\dotsc,n\}$, there is an element of $A$ taking one to the other).
Hence either all $2$-cycles or 
all $3$-cycles are in $A^{\ell_1+2n^3}$ (in that they can be obtained
by conjugating $g_1$ by elements of $A^{n^3}$). If at least one element $h$
of $A$ is not in $\Alt(n)$, it is easy to construct a $2$-cycle (and hence
all $2$-cycles) by using $h$ and some well-chosen $3$-cycles. 
We then construct every element of $\langle A\rangle$ (which, without having
meant to,
we are now showing to be either $\Alt(n)$ or $\Sym(n)$) as a product of
length at most $n$ in our $2$-cycles (if $A\not\subset \Alt(n)$) or $3$-cycles
(if $A\subset \Alt(n)$).
\end{proof}
There is clearly some double-counting going on in the proof of 
Prop.~\ref{prop:lalka}. A more careful counting argument gives an improved
statement that results in a version of Cor.~\ref{cor:sto} with $1/2$
instead of $1/3$. Redoing Prop.~\ref{prop:lalka} with well-chosen words
other than $\lbrack g,h\rbrack$ results in still better bounds;
\cite{MR3226815} gives $\diam(\Gamma(G,A)) \ll n^{O(1)}$ provided
that there is a $g\in A$, $g\ne e$ with $|\supp(g)| \leq 0.63 n$.

The moral of this section is that short random walks can be enough for
``the probabilistic method'' in combinatorics (showing existence by showing
positive probability) to work, in that they serve to approximate the
uniform distribution on $k$-tuples ($k$ small) very well.

\subsection{Large orbits, pointwise stabilizers and stabilizer chains}\label{subs:magano}

The following result was a key part of the proof of Babai's elementary
bound on the size of $2$-transitive permutation groups other than
$\Alt(n)$ and $\Sym(n)$ \cite{Bab82}. It has a probabilistic proof.

\begin{lem}[Babai's splitting lemma \cite{Bab82}]\label{lem:babai}
Let $H\leq \Sym(n)$ be $2$-transitive.
 Let $\Sigma\subset 
\{1,2,\dotsc,n\}$. Let $\rho>0$. Assume that there are at least $\rho n (n-1)$
ordered pairs $(\alpha,\beta)$ of distinct elements of $\{1,2,\dotsc,
n\}$
such that there is no $g\in H_{(\Sigma)}$ with $\alpha^g = \beta$. Then
there is a subset $S$ of $H$ with
\[H_{(\Sigma^S)} = \{e\}\]
and $|S|\ll_\rho \log n$.
\end{lem}
\begin{proof}
Let $\alpha$, $\beta$ be distinct elements of $\{1,2,\dotsc,n\}$.
Let $h\in H$. Suppose there is a $g'\in H_{(\Sigma^{h})}$ such that
$\alpha^{g'} = \beta$. Then $g = h g' h^{-1}$ is an element of $H_{(\Sigma)}$
taking $\alpha^{h^{-1}}$ to $\beta^{h^{-1}}$.

The elements $h$ of $H$ such that $h^{-1}$ takes
 $(\alpha,\beta)$ to a given pair $(\alpha',\beta')$
of distinct elements of $\{1,2,\dotsc,n\}$ form a coset of $H_{(\alpha,\beta)}$.
Hence, if we choose an element $h\in H$ at random, $h^{-1}$ is equally likely
to take $(\alpha,\beta)$ to any given pair $(\alpha',\beta')$.
In particular, the probability that it will take $(\alpha,\beta)$ to a pair
$(\alpha',\beta')$ such that there is no $g\in H_{(\Sigma)}$ taking $\alpha$
to $\beta$ is at least $\rho$. By what we were saying, this would imply that
there is no $g'\in H_{(\Sigma^{h})}$ such that $\alpha^{g'} = \beta$.

Now take a set $S$ of $r$ elements of $H$ taken uniformly at random.
By the above, 
for a given $(\alpha,\beta)$, the probability that, for every $h\in H$,
there is a $g'\in H_{(\Sigma^{h})}$ such that $\alpha^{g'} = \beta$ is
at most $(1-\rho)^r$. There would be such a $g'$ for every $h\in H$ if
there were a $g' \in H_{\Sigma^S}$ such that $\alpha^{g'} = \beta$ (since
such a $g'$ would be good for every $h\in H$).
Hence, the probability that there is at least
one pair $(\alpha,\beta)$ of distinct elements such that there 
is a $g'\in H_{\Sigma^S}$ with $\alpha^{g'} = \beta$ is at most $n^2 (1-\rho)^r$.
For any $r > 2 (\log n)/\log(1/(1-\rho))$, we get $n^2 (1-\rho)^r < 1$,
and thus there exists a set $S$ of at most $r$ elements such that there is
no such pair $(\alpha,\beta)$. If there is no such pair, then the only
possible element of $H_{\Sigma^S}$ is the identity, i.e., $H_{\Sigma^S} =\{e\}$.
\end{proof}

We wish to adapt Lem.~\ref{lem:babai} to hold for subsets $A\subset G$
 instead of subgroups $H$. Here one of our leitmotifs
reappears, but undergoes a change. Adapting a result on subgroups to
hold for subsets is a recurrent idea that we have seen throughout this
survey. However, so far, we have usually done this by relaxing the condition 
that $A$ be a subgroup into the condition that $A^3$ not be much 
larger than $A$. This is a tactic that often works when the underlying idea is 
basically quantitative, as is the case, e.g., for the orbit stabilizer theorem.

Another tactic consists in redoing an essentially constructive proof
keeping track of how many products are taken; this is, for example, how
a lemma of Bochert's gets adapted in \cite[Lem.~3.12]{MR3152942}. To generalize
Babai's splitting lemma, however, we will follow a third tactic -- namely,
making a probabilistic proof into what we can call a stochastic one, viz.,
one based on random walks, or, more generally, on a random process.

We already saw how to use random walks in this way in 
\S \ref{subs:gotr} (Babai-Beals-Seress); the idea is to approximate the uniform
distribution by a short random walk, using Lem.~\ref{lem:bulgo}.

\begin{lem}[Splitting lemma for sets (\cite{MR3152942}, Prop.~5.2)]\label{lem:juto}
Let $A\subseteq \Sym(n)$ with $A=A^{-1}$, $e\in A$ and $\langle A\rangle$ 
 $2$-transitive. 
Let $\Sigma\subseteq \{1,2,\dotsc, n\}$. Let $\rho>0$.Assume that
there are at least $\rho n(n-1)$ ordered pairs $(\alpha,\beta)$ of distinct elements of 
$\{1,2,\dotsc, n\}$ such that there is no 
$g\in (A^{\lfloor 9n^6\log n\rfloor})_{(\Sigma)}
$ with $\alpha^g=\beta$. Then there is a subset $S$
 of $A^{\lfloor 5n^6\log n\rfloor}$ with $$(A A^{-1})_{(\Sigma^S)}=\{e\}$$  and
$|S|\ll_\rho \log n$.
\end{lem}
Passing from the statement of Lem.~\ref{lem:babai} to the statement of
Lem.~\ref{lem:juto}, some instances of $H$ get replaced by $A$ and some
get replaced by $A^k$, $k$ moderate. This is what makes it possible to
give a statement true for general sets $A$ without assumptions on the size
of $A^3$; of course, a moment's thought shows that
the statement is particularly strong when $A$ grows slowly. This is much as 
in Prop.~\ref{prop:port} or Prop.~\ref{prop:bantu}.
\begin{proof}[Sketch of proof] 
Exercise. Adapt the proof of 
 Lem.~\ref{lem:babai}, using Prop.~\ref{prop:lalka}. 
It is useful to note that
$g A_{(\Sigma^g)} g^{-1} = (g A g^{-1})_{(\Sigma)}$.
\end{proof}

Let us see how to use Lem.~\ref{lem:juto}. (This part of the argument
is similar to that in \cite{Pyb93}, which sharpened the bounds in 
\cite{Bab82}.) By pigeonhole, $(A A^{-1})_{\Sigma^S} = \{e\}$ can be the case 
only if $|A|\leq n^{|\Sigma^S|}$, and that can happen only if 
\[|\Sigma| \gg \frac{\log |A|}{(\log n)^2}.\]
This means that there is a constant $c$ such that, for every
$\Sigma\subset \{1,2,\dotsc,n\}$ with $|\Sigma| \leq c (\log |A|)/(\log n)^2$,
the assumption (``there are at least $\rho n (n-1)$ ordered pairs\dots'')
in Lem.~\ref{lem:juto} does not hold (since the conclusion cannot
hold.)

Thus, for any $\sigma<1$, we are guaranteed to be able to find $\Sigma =
\{\alpha_1,\alpha_2,\dotsc, \alpha_m\}$, $m\gg_\sigma (\log |A|)/(\log n)^2$,
such that, for $A' = A^{\lfloor 9n^6\log n\rfloor}$,
\begin{equation}\label{eq:janak}
|\alpha_i^{A'_{(\alpha_1,\dotsc,\alpha_{i-1})}}|\geq \sigma n
\end{equation}
for every $1\leq i\leq m$; we are setting $\rho = 1-\sigma$. (The use of stabilizer chains 
$A>A_{(\alpha_1)}>A_{(\alpha_1,\alpha_2)}>\dotsc$ goes back to the algorithmic
work of Sims \cite{MR0257203}, \cite{Sim71}, as does
the use of the size of the orbits in (\ref{eq:janak}); see
\cite[\S 4.1]{MR1970241}.) 

By (\ref{eq:janak}), $(A')^m$ occupies at least $(\sigma n)^m$ cosets of
the pointwise stabilizer $\Sym(n)_{(\Sigma)}$ (exercise; \cite[Lem. 3.17]{MR3152942}),
out of $n!/(n-m)!<n^m$ possible cosets of
$\Sym(n)_{(\Sigma)}$ in $\Sym(n)$. The number of cosets of
$\Sym(n)_{(\Sigma)}$ in the setwise stabilizer $\Sym(n)_{\Sigma}$ is 
$m!$, which is much larger than $n^m/(\sigma n)^m = (1/\sigma)^m$. (We can
work with $\sigma=9/10$, say.) A hybrid of Lem.~\ref{lem:dora} and
Lem.~\ref{lem:arana} (\cite[Lem.~3.7]{MR3152942}) then shows immediately that
$(A')^{2m} \cap \Sym(n)_{\Sigma}$ intersects many ($\geq \sigma^m m!$) 
cosets of $\Sym(n)_{(\Sigma)}$ (and, in particular, $|(A')^{2m} \cap
\Sym(n)_{\Sigma}|\geq \sigma^m m!$).

Let us see what we have got. We have constructed many elements of 
$(A')^{2m} \subset A^{n^{O(1)}}$ lying in a ``special subgroup'' 
$(\Sym(n))_{\Sigma}$. This is in loose analogy
to the situation over linear algebraic
groups, where we constructed many elements of $A^2$ lying in a special
subgroup $T = C(g)$ (Cor.~\ref{cor:nessumo}). Moreover, the elements
of $(A')^{2m} \cap (\Sym(n))_{\Sigma}$ will act (by conjugation) on
an even more special\footnote{Note the shift, or non-shift, in the meaning
of ``special'' (dictated by the requirements of the problems at hand). Before,
a special subgroup was exactly that, namely, an algebraic subgroup (``special''
meaning ``lying on a set of positive codimension, algebraically speaking'').
Here the role of ``special subgroups'' is played by stabilizers (in relation
to the natural
action of $\Sym(n)$ on $\{1,2,\dotsc,n\}$, or powers
thereof). The difference
is, however, less than it seems at first: the algebraic subgroup $T$ is also
given as a stabilizer $C(g)$ (in relation to
 the action by conjugation of $G$ on itself, which is a natural action of $G$ on
an affine space.)}
 subgroup, namely, $(\Sym(n))_{(\Sigma)}$.

This is a turning point in the proof of Thm.~\ref{thm:hs}, just as 
Cor.~\ref{cor:nessumo} (or its weaker version, \cite[Prop.~4.1]{Hel08})
was a turning point in the proof of Thm.~\ref{thm:main08}.

\subsection{Constructing small generating sets}\label{subs:genset}

Let $A$ be a set of generators of $\Sym(n)$ or 
$\Alt(n)$. The set $A$ may be large -- inconveniently so
for some purposes. Can we find a set $S\subset A^{\ell}$ of bounded size
($\ell$ moderate) so that $S$ generates $\langle A\rangle$?

This is a question that arises in the course of the proof of Thm.~\ref{thm:hs}.
Addressing it will give us the opportunity to show how to use stochastic 
processes other than a simple random walk in order to put a generalized
probabilistic method into practice.

Let us start with a simple lemma.

\begin{lem}{\rm \cite[Lem.~4.3]{MR3152942}}\label{lem:macanus}
Let $A\subset \Sym(n)$, $e\in A$. Assume $\langle A\rangle$ is transitive. Then there is a $g\in A^n$ such that
$|\supp(g)|\geq n/2$.
\end{lem} 
\begin{proof}[Sketch of proof]
For every $i\in \{1,2,\dotsc,n\}$, there is a $g_i\in A$
such that $i\in \supp(g)$ (why?). Consider $g = g_1^{r_1} g_2^{r_2}
\dotsc g_n^{r_n}$, where $r_1,r_2,\dotsc r_n\in \{0,1\}$ are independent
random variables taking the values $0$ and $1$ with equal probability. Show
that the expected value of $|\supp(g)|$ is at least $n/2$ (exercise).
\end{proof}

We will be using $g$ to move an element of $\{1,2,\dotsc,n\}$
 around and another element $h$ (produced
by a random walk) in order to scramble $\{1,2,\dotsc,n\}$.

\begin{prop}{\rm \cite[Lem.~4.5 and Prop.~4.6]{MR3152942}}\label{prop:crat} 
Let $A\subseteq \Sym(n)$ with $A = A^{-1}$, $e\in A$ and $\langle A\rangle
=\Sym(n)$ or $\Alt(n)$. Then there are $g\in A^n$
and $j,h\in A^{n^{O(\log n)}}$ such that $\langle g,j,h\rangle$ is transitive. 
\end{prop}
\begin{proof}[Extended sketch of proof]
By Lem.~\ref{lem:macanus}, there is a $g\in A^n$ with $|\supp(g)|\geq n/2$.
Let $h \in A^\ell$, $\ell = n^{7 k}$, $k = \lceil 8 \log n \rceil$ (say) be the 
outcome of a lazy random walk on $\Gamma(G,A)$ of length $\ell$
(starting at $e$). We can assume $n$ is larger than an absolute constant.

We will consider words of the form 
\[
f(\vec{a}) = h g^{a_1} h g^{a_2} \dotsc h g^{a_k},
\]
where $a_1,\dotsc, a_k\in \{0,1\}$. We wish to show that, for
$\beta\in \{1,2,\dotsc,n\}$ taken at random (with the uniform
distribution on $\{1,2,\dotsc,n\}$), the orbit
$\beta^{f(\vec{a})}$, $\vec{a} \in \{0,1\}^k$,
 is likely to be very large ($\gg n/(\log n)^2$).

A simple sphere-packing bound shows that there is a set
$V\subset \{0,1\}^k$, $|V|\geq n$, such that the Hamming distance\footnote{The {\em Hamming distance} $d(\vec{x},\vec{y})$ between two elements
$\vec{x}, \vec{y}\in \{0,1\}^k$ is the number of indices 
$1\leq i\leq k$ such that $x_i \ne y_i$.} between
any two distinct elements of $V$ is at least $k/5 > \log_2 n$. (Exercise.) 
We wish to show that, for $\beta$ random and $\vec{a}, \vec{a}'
\in V$ distinct, it is very unlikely that $\beta^{f(\vec{a})}$ equal
$\beta^{f(\vec{a}')}$.

Write $\vec{a} = (a_1,a_2,\dotsc, a_k)$, $\vec{a}' = 
(a_1',a_2',\dotsc, a_k')$.
Consider the sequences
\begin{equation}\label{eq:futbol}\begin{aligned}
\beta_0 = \beta, \beta_1 = \beta_0^{h g^{a_1}}, \beta_2 = \beta_1^{h g^{a_2}},\dotsc,
\beta_k = \beta_{k-1}^{h g^{a_k}},\\
\beta_0' = \beta,
\beta_1' = \beta^{h g^{a_1'}}, \beta_2' = (\beta_1')^{h g^{a_2'}},\dotsc,
\beta_k' = (\beta_{k-1}')^{h g^{a_k'}}.\end{aligned}\end{equation}

It is very unlikely that $\beta^h = \beta$ (probability $\sim 1/n$)
or that $\beta^{h g} = \beta$, i.e., $\beta^h = \beta^{g^{-1}}$
(probability $\sim 1/n$). If neither of these unlikely occurrences takes
place, it is also very unlikely (total probability $\leq (1+o(1)) 4/n$) 
that $\beta_1^h$ or $\beta_1^{h g}$ equal
$\beta$ or $\beta^h$. The reason is that, since $\beta_1$ has not been seen
"before" (i.e., $(\beta,\beta_1)$ is a pair of distinct elements),
the distribution of $h^{\beta_1}$ is almost uniform, even conditionally
on $\beta=x$, for any $x$. (This can be easily made rigorous; it is much
as (for instance) in the proof of Prop.~\ref{prop:lalka}, right after
(\ref{eq:smet}).)
Thus, the probability that $\beta_1^h = \beta$ (for example) is $\sim 1/n$
(the same as the probability of $\beta_1^h = x$ for any $x$ other than
$\beta^h$).
Proceeding in this way, 
we obtain that it is almost
certain (probability $1-O(k/n)$) that $\beta_1,\beta_2,\dotsc,\beta_k$
are all distinct. (Recall that, by Lemma \ref{lem:bulgo}, 
a random walk of length $\ell$
mixes $k$-tuples (and even $2k$-tuples)
of distinct elements. It is also relevant that $k$ is
very small compared to $n$, as this means that hitting one of $k$
(or rather $2k$) visited elements by picking an element of $\{1,2,\dotsc,n\}$
at random is highly unlikely. We can keep our independence from the past
as long as we do not go back to it.)

Let us see what happens to $\beta_0',\beta_1',\dotsc$ in the meantime.
Start at $i=0$ and increase $i$ by $1$ repeatedly.
As long as $a_i = a_i'$, we have $\beta_i = \beta_i'$. As soon as 
$a_i \ne a_i'$ (denote by $i_1$ the first index $i$ for which this happens),
we {\em may} have $\beta_i \ne \beta_i'$; this happens when $\beta_{i-1}^h
\in \supp(g)$, i.e., it happens
 with probability
$\sim |\supp(g)|/n$. If this happens, then, by the same argument as above,
it is highly likely that the two paths in (\ref{eq:futbol}) diverge, i.e.,
$\beta_j \ne \beta_j'$ for all $j>i$, and, for that matter, that they also
avoid
each other's past ($\beta_j\ne \beta_l'$ for all $j>i$ and all $l<j$).
(It is useful to keep track of the latter condition for the same reason
as above, namely, to keep our independence from events that have already
been determined.)

Since $\vec{a}$, $\vec{a}'$ are at Hamming distance at least
$n$ from each other, it is very unlikely that $\beta_{i-1}^h\in
\supp(g)$ for all $i$ such that $a_i \ne a_i'$ (probability
$\leq (1-|\supp(g)|/n)^{n}\leq 2^{-n}$, since there are $n$ such indices $i$). 
Hence the two paths almost certainly diverge -- never to meet again, as
we just showed; in particular, $\beta^{f(\vec{a})}$ and 
$\beta^{f(\vec{a}')}$ are almost certainly distinct. They
are distinct with probability $\geq 1 - O((\log n)^2/n)$ for any distinct
$\vec{a}, \vec{a'}\in V$ and 
$\beta\in \{1,2,\dotsc,n\} $ random, to be precise.

By Cauchy-Schwarz, this implies that the expected value of 
$1/|\beta^{\langle g,h\rangle}|$ for $\beta$ random is $O((\log n)^2/n)$.
This implies, in turn, that the expected value of the
 number of orbits of $\langle g,h\rangle$
is $O((\log n)^2)$. (Exercises.)

A third element $j\in \langle A\rangle$ obtained by a random walk of length
$\ell$ almost certainly merges these orbits, i.e.,
$\langle g,h,j\rangle$ is transitive. (Longer but easy exercise.) 
Hence there exist (many) $g,h,j\in A^\ell$ such that $\langle g,h,j\rangle$ is 
transitive.
\end{proof}

\begin{cor}{\rm \cite[Cor.~4.7]{MR3152942}}\label{cor:tabor}
Let $A\subseteq \Sym(n)$ with $A = A^{-1}$, $e\in A$ and $\langle A\rangle
=\Sym(n)$ or $\Alt(n)$. Then, for every $k\geq 1$,
 there is a set $S\in A^{(3 n)^k n^{O(\log n)}}$ of size at most $3 k$ such
 that
$\langle S\rangle$ is transitive.
\end{cor}
In particular, if we want $\langle S\rangle$ to be $\Sym(n)$ or $\Alt(n)$
(something we do not need in the application in \cite{MR3152942}) then it is
enough to set $k=6$, as the Classification of Finite Simple groups implies
that a $6$-transitive group must be either $\Sym(n)$ or $\Alt(n)$.
\begin{proof}[Sketch of proof]
Apply Prop.~\ref{prop:crat} repeatedly, using Schreier generators to pass
to pointwise stabilizers of $\{1\}$, $\{1,2\}$, etc.
\end{proof}

How far can arguments such as those in the proof of Prop.~\ref{prop:crat}
be pushed? Here there is again a ``classical'' argument to be examined in
the light of random processes and random walks, namely, the work
of Broder and Shamir on the spectral gap of random graphs \cite{BroSha}.
The ideas there and those in Prop.~\ref{prop:crat} are some of the elements
leading to \cite{HSZ}, which gives a bound of 
$O\left(n^2 (\log n)^{O(1)}\right)$ on the diameter of $\Gamma(\langle
A\rangle,
A)$ for $A = \{g,h\}$, $g, h\in \Sym(n)$ random.

\subsection{The action of the setwise stabilizer on the pointwise
stabilizer}

How is this all put together to give Thm.~\ref{thm:hs}? The entire argument
is outlined in detail in \cite[\S 1.5]{MR3152942}. Here, let us go over a crucial
step and look quickly at what then follows, skipping some of the complications. 

We are working with a set $A\subset \Sym(n)$ generating $\Alt(n)$ or
$\Sym(n)$. By the end of \S \ref{subs:magano}, we had constructed a large
subset $\Sigma\subset \{1,2,\dotsc, n\}$ ($m=|\Sigma|\gg (\log |A|)/(\log
n)^2$) such that $(A'')_\Sigma$ (where $A'' = (A')^{2 m} = A^{n^{O(1)}}$) 
intersects $\geq \sigma^m m!$ cosets of $\Sym(n)_{(\Sigma)}$;
in other words, the projection of $(A'')_\Sigma$
to $\Sym(\Sigma)$ (by restriction to $\Sigma$) has $\geq \sigma^m m!$
elements. 

By the trick of demanding (\ref{eq:janak}) for $m+1$ rather than $m$,
we can ensure that $\langle (A'')_{(\Sigma)}\rangle$ has at least one
large orbit $\Gamma$ ($|\Gamma|\geq \sigma n$). 
We can actually assume that $\langle (A'')_{(\Sigma)}\rangle$ acts as 
$\Sym$ or $\Alt$ on $\Gamma$, since otherwise we are done by a different
argument (called {\em descent} in \cite[\S 6]{MR3152942}, as in ``infinite descent'',
because it is inductive; it is also the one step that
involves the Classification of Finite Simple Groups). Then, by 
Cor.~\ref{cor:tabor}, there is a set $S\subset 
\left((A'')^{n^{O(\log n)}}\right)_{(\Sigma)}$,
$|S|\leq 6$, such that $\langle S\rangle$ acts as a $2$-transitive group
on $\Gamma$.

We now consider the action of the elements of $(A'')_\Sigma$ on the
elements of $S$ by conjugation. By the orbit-stabilizer principle
(Lem.~\ref{lem:orbsta}), either (a) there is an element $g\ne e$ of 
$(A'')_\Sigma$ commuting with every element of $S$, or (b) the orbit 
$\{g s g^{-1} : g\in (A'')_\Sigma\}$ of some $s\in S$ is of size
$\geq |(A'')_\Sigma|^{1/6} \geq (\sigma^m m!)^{1/6}$.
This orbit is entirely contained in the pointwise stabilizer 
$\left(A^{n^{O(\log n)}}\right)_{(\Sigma)}$.

In case (a), $g$ must act trivially on $\Gamma$ (since it commutes
with a $2$-transitive group on $\Gamma$) and so we are done by
Cor.~\ref{cor:sto} (Babai-Beals-Seress). In case (b), we have succeeded
in constructing many elements of $A^{n^{O(\log n)}}$ in
 the pointwise stabilizer of the set $\Sigma$.

This does not mean we are done yet; perhaps there were already many
elements of $A^2$ in the pointwise stabilizer of $\Sigma$. (Otherwise
Lem.~\ref{lem:dora} does mean that $A$ is growing rapidly, and so
we are done.) However, having many elements in the pointwise stabilizer
of $\Sigma$ does mean that we can start now an iteration, constructing
a second set $\Sigma_2$ and a longer stabilizer chain
satisfying (\ref{eq:janak}) with $A$ replaced
by $\left(A^{n^{O(\log n)}}\right)_{(\Sigma)}$, and then a third set
$\Sigma_3$,
and so on. Instead of focusing on making $A$ grow,
we focus on making the length of the stabilizer chain grow, until it
reaches size about $n$, at which point we are done.

\section{Some open problems}
The following questions are hard and far from new.
\begin{enumerate}
\item Consider all Cayley graphs $\Gamma(G,A)$ with $G=\SL_2(\mathbb{F}_p)$,
$A\subset G$, $|A|=2$, $\langle A\rangle = G$, $p$ arbitrary. Are they
all $\epsilon$-expanders for some fixed $\epsilon>0$? 

Here \cite[p. 96]{MR1235570} states that an affirmative answer was
made plausible by the experiments in \cite{MR1203870}. The proof in 
\cite{MR2415383} is valid for $A = A_0 \mo p$ ($A_0$ fixed) and 
also for $A$ random
(with probability $1$), among other cases; see also \cite{MR2746951}.
Expansion for $A=A_0 \mo p$, $A_0$ fixed, 
is known for all non-abelian simple groups
of Lie type $G$ (of bounded rank) thanks to \cite{SGV}; a proof of
expansion has also been announced for such $G$ and $A$ random
\cite{BGGT2}. Expansion has been conjectured
for general $G$ of bounded rank and 
arbitrary $A$; see, e.g., \cite[Conj.~2.29]{MR2869010}.

\item\label{it:usto} Does every Cayley graph $\Gamma(G,A)$ with $G=\SL_n(\mathbb{F}_p)$,
$A$ a set of generators of $G$ ($n$ and $p$ arbitrary) have diameter
$(\log |G|)^{C}$, where $C$ is an {\em absolute} constant?

(This is Babai's conjecture in the case of linear algebraic groups. 
The cases $n=2$, $n=3$ were proven in \cite{Hel08}, \cite{HeSL3}.
Both \cite{BGT} and \cite{PS} give this result with $C$ depending on $n$.)

\item\label{it:horol} Does every Cayley graph $\Gamma(\Sym(n),A)$ ($A$ a set of generators
of $\Sym(n)$) have diameter $n^{O(1)}$?

(This predates Babai's more general conjecture. Here Thm.~\ref{thm:hs}
(as in \cite{MR3152942}) is the best result to date.)

\item Let $A$ consist of two random elements of $\Alt(n)$. Is 
$\Gamma(\Alt(n),A)$ an $\epsilon$-expander, $\epsilon>0$ fixed? Is
the diameter of $\Gamma(\Alt(n),A)$ at most $n (\log n)^{O(1)}$?

(A ``yes'' to the former question implies a ``yes'' to the latter, but there
is no consensus on what the answer to either question should be.)

\item (``Navigation'') Given a set $A$ of generators of $G=\SL_2(\mathbb{F}_p)$ 
($|A|=2$ if you wish) and an element
$g$ of $G$, can you find in time $O((\log p)^{c_1})$ a product of length
$O((\log p)^{c_2})$ of elements of $A\cup A^{-1}$
equal to $g$?

(A probabilistic algorithm for a specific $A$ is given in \cite{MR1976231}.)
\end{enumerate}

One of the difficulties in answering question (\ref{it:horol})
resides in the fact that a statement such as Thm.~\ref{thm:main08}
cannot be true for all subsets $A$ of a symmetric group $G$. Both Pyber and Spiga have given
counterexamples. The following counterexample is due to Pyber:
let $G = \Sym(2n+1)$ and $A = H\cup \{\sigma,\sigma^{-1}\}$, where
$\sigma$ is the shift $m\to m+2 \mo 2n+1$ and $H$ is the subgroup generated
by all transpositions $(i,i+1)$ with $1\leq i\leq n$. Then $|A^3|\ll
|A|$. See also \cite[\S 3]{MR2898694} and \cite{MR2876252}.

 What happens if the conditions
on $A\subset \Sym(n)$ are strengthened? Can a statement such as Thm.~\ref{thm:main08} then be true? In a first draft of the present text, the author
asked whether $r$-transitivity is enough. That is: let $A\subset G$
($G=\Sym(n)$ or $G=\Alt(n)$) be a set of generators of $G$; assume 
that $A$ is $r$-transitive, meaning 
that, for any two $r$-tuples
$v_1$, $v_2$ of distinct elements of $\{1,2,\dotsc,n\}$, there 
is a $g\in A$ such that $g$ takes $v_1$ to $v_2$. If $r$ is greater
than a constant (say $6$), does it follow
that \begin{equation}\label{eq:gormo}\text{either}\;\;\;\;
|A^3|>|A|^{1+\delta}\;\;\;\;
\text{or}\;\;\;\;A^k = G,\end{equation}
where $k$ is an absolute constant? L. Pyber promptly showed that the answer
is ``no'': let $A$ be the union of any large subgroup $H<G$ and the union
of all $2r$-cycles; then $|A^3|\leq |C|^2 |A|\leq n^{4r} |A|$,
and this is much smaller than $|A|^{1+\delta}$ for $H$ large.

What if $A\subset G$ is of the form $A=B^k$, where $|B|=O(1)$? Is this
a sufficient condition for (\ref{eq:gormo}) to hold? It is easy to see that
a``yes'' answer here,
together with a stronger version of Prop.~\ref{prop:crat} (with $j,h\in
A^{n^{O(1)}}$ instead of $j,h\in
A^{n^{O(\log n)}}$),
would imply a ``yes'' answer to (\ref{it:horol}) above.


A separate, more open-ended question is that of the relevance of 
work on permutation groups 
to the study of linear algebraic groups. As we discussed before,
the problem of proving growth in $\Alt_n$ is closely related to that of
proving growth in $\SL_n$ uniformly as $n\to \infty$.
\begin{Cha}
Apply the ideas in \cite{MR3152942} to
question (\ref{it:usto}) above.
\end{Cha}

Finally, let us end with a question for which the time is arguably ripe, but 
for which there is still no full answer. The idea is to give a full
description of subsets of $A$ that fail to grow.

\begin{conj}
Let $K$ be a field. Let $A$ be a finite subset of $\GL_n(K)$ with
$A=A^{-1}$, $e\in A$. Then,
for every $R\geq 1$, either
\begin{enumerate}
\item $|A^3|\geq R |A|$, or else
\item there are two subgroups $H_1\leq H_2$ in $\GL_n(K)$ and an integer
$k = O_n(1)$ such that
\begin{itemize}
\item $H_1$ and $H_2$ are both normal in $\langle A\rangle$, and $H_2/H_1$
is nilpotent,
\item $A^k$ contains $H_1$, and
\item $|A^k\cap H_2|\geq R^{-O_n(1)} |A|$.
\end{itemize}
\end{enumerate}
\end{conj}
This conjecture was made fairly explicitly in \cite{HeSL3} (see comments
after \cite[Thm~1.1]{HeSL3}), where it was also proven for $n=3$
and $K = \mathbb{F}_p$ (in a slightly weaker form). 
The same conjecture was proven for $n$ general
and $K = \mathbb{F}_p$ as \cite[Thm.~2]{GH2} (joint with Pyber and Szab\'o).
Breuillard, Green and Tao have given to this conjecture the name of
{\em Helfgott-Lindenstrauss conjecture}; in \cite{BGTstru}, they proved
a qualitative version with non-explicit bounds (valid even for non-algebraic
groups). The case of $n$ general
and $K$ general, as stated here,
 remains open. A slightly modified version of the conjecture (for $n$ and $K$ general, but
with $H_2/H_1$ soluble rather than
nilpotent) has been proven by Pyber and Szab\'o \cite[Thm.~8]{PSSup}.
\bibliographystyle{alpha}
\bibliography{survgr}

\newcommand{\etalchar}[1]{$^{#1}$}
\def\cprime{$'$}
\begin{thebibliography}{BGH{\etalchar{+}}14}

\bibitem[AS00]{MR1885388}
N.~Alon and J.~H. Spencer.
\newblock {\em The probabilistic method}.
\newblock Wiley-Interscience Series in Discrete Mathematics and Optimization.
  Wiley-Interscience [John Wiley \& Sons], New York, second edition, 2000.
\newblock With an appendix on the life and work of Paul Erd{\H{o}}s.

\bibitem[Bab86]{MR860123}
L.~Babai.
\newblock On the length of subgroup chains in the symmetric group.
\newblock {\em Comm. Algebra}, 14(9):1729--1736, 1986.

\bibitem[Bab06]{Babeul}
L.~Babai.
\newblock On the diameter of {E}ulerian orientations of graphs.
\newblock In {\em Proc. 17th ACM - SIAM Symp. on Discrete Algorithms
  (SODA'06)}, pages 822--831. ACM-SIAM, 2006.

\bibitem[Bab82]{Bab82}
L.~Babai.
\newblock On the order of doubly transitive permutation groups.
\newblock {\em Invent. Math.}, 65(3):473--484, 1981/82.

\bibitem[Bas72]{MR0379672}
H.~Bass.
\newblock The degree of polynomial growth of finitely generated nilpotent
  groups.
\newblock {\em Proc. London Math. Soc. (3)}, 25:603--614, 1972.

\bibitem[BBS04]{BBS04}
L.~Babai, R.~Beals, and {\'A}.~Seress.
\newblock On the diameter of the symmetric group: polynomial bounds.
\newblock In {\em Proceedings of the {F}ifteenth {A}nnual {ACM}-{SIAM}
  {S}ymposium on {D}iscrete {A}lgorithms}, pages 1108--1112 (electronic), New
  York, 2004. ACM.

\bibitem[BD92]{MR1161056}
D.~Bayer and P.~Diaconis.
\newblock Trailing the dovetail shuffle to its lair.
\newblock {\em Ann. Appl. Probab.}, 2(2):294--313, 1992.

\bibitem[BG08a]{MR2443926}
J.~Bourgain and A.~Gamburd.
\newblock Expansion and random walks in {${\rm SL}_d(\Bbb Z/p^n\Bbb Z)$}. {I}.
\newblock {\em J. Eur. Math. Soc. (JEMS)}, 10(4):987--1011, 2008.

\bibitem[BG08b]{BGSU2}
J.~Bourgain and A.~Gamburd.
\newblock On the spectral gap for finitely-generated subgroups of {$\rm
  SU(2)$}.
\newblock {\em Invent. Math.}, 171(1):83--121, 2008.

\bibitem[BG08c]{MR2415383}
J.~Bourgain and A.~Gamburd.
\newblock Uniform expansion bounds for {C}ayley graphs of {${\rm
  SL}_2(\mathbb{F}_p)$}.
\newblock {\em Ann. of Math. (2)}, 167(2):625--642, 2008.

\bibitem[BG09]{MR2538500}
J.~Bourgain and A.~Gamburd.
\newblock Expansion and random walks in {${\rm SL}_d(\Bbb Z/p^n\Bbb Z)$}. {II}.
\newblock {\em J. Eur. Math. Soc. (JEMS)}, 11(5):1057--1103, 2009.
\newblock With an appendix by J. Bourgain.

\bibitem[BG10]{MR2746951}
E.~Breuillard and A.~Gamburd.
\newblock Strong uniform expansion in {${\rm SL}(2,p)$}.
\newblock {\em Geom. Funct. Anal.}, 20(5):1201--1209, 2010.

\bibitem[BG11a]{MR2749571}
E.~Breuillard and B.~Green.
\newblock Approximate groups. {I}: {T}he torsion-free nilpotent case.
\newblock {\em J. Inst. Math. Jussieu}, 10(1):37--57, 2011.

\bibitem[BG11b]{MR2825469}
E.~Breuillard and B.~Green.
\newblock Approximate groups, {II}: {T}he solvable linear case.
\newblock {\em Q. J. Math.}, 62(3):513--521, 2011.

\bibitem[BG12]{MR2912036}
E.~Breuillard and B.~Green.
\newblock Approximate groups, {III}: the unitary case.
\newblock {\em Turkish J. Math.}, 36(2):199--215, 2012.

\bibitem[BGGT]{BGGT2}
E.~Breuillard, B.~Green, R.~Guralnick, and T.~Tao.
\newblock Expansion in finite simple groups of {L}ie type.
\newblock Available as \texttt{arxiv.org:1309.1975}.

\bibitem[BGH{\etalchar{+}}14]{MR3226815}
J.~Bamberg, N.~Gill, Th.~P. Hayes, H.~A. Helfgott, {\'A}.~Seress, and P.~Spiga.
\newblock Bounds on the diameter of {C}ayley graphs of the symmetric group.
\newblock {\em J. Algebraic Combin.}, 40(1):1--22, 2014.

\bibitem[BGK06]{MR2225493}
J.~Bourgain, A.~A. Glibichuk, and S.~V. Konyagin.
\newblock Estimates for the number of sums and products and for exponential
  sums in fields of prime order.
\newblock {\em J. London Math. Soc. (2)}, 73(2):380--398, 2006.

\bibitem[BGS10]{MR2587341}
J.~Bourgain, A.~Gamburd, and P.~Sarnak.
\newblock Affine linear sieve, expanders, and sum-product.
\newblock {\em Invent. Math.}, 179(3):559--644, 2010.

\bibitem[BGS11]{MR2892611}
J.~Bourgain, A.~Gamburd, and P.~Sarnak.
\newblock Generalization of {S}elberg's {$\frac{3}{16}$} theorem and affine
  sieve.
\newblock {\em Acta Math.}, 207(2):255--290, 2011.

\bibitem[BGT11]{BGT}
E.~Breuillard, B.~Green, and T.~Tao.
\newblock Approximate subgroups of linear groups.
\newblock {\em Geom. Funct. Anal.}, 21(4):774--819, 2011.

\bibitem[BGT12]{BGTstru}
E.~Breuillard, B.~Green, and T.~Tao.
\newblock The structure of approximate groups.
\newblock {\em Publications math\'ematiques de l'IH\'ES}, 116:115--221, 2012.

\bibitem[BH92]{MR1208801}
L.~Babai and G.~L. Hetyei.
\newblock On the diameter of random {C}ayley graphs of the symmetric group.
\newblock {\em Combin. Probab. Comput.}, 1(3):201--208, 1992.

\bibitem[BH05]{BH}
L.~Babai and Th.~P. Hayes.
\newblock Near-independence of permutations and an almost sure polynomial bound
  on the diameter of the symmetric group.
\newblock In {\em Proceedings of the {S}ixteenth {A}nnual {ACM}-{SIAM}
  {S}ymposium on {D}iscrete {A}lgorithms}, pages 1057--1066 (electronic), New
  York, 2005. ACM.

\bibitem[BIW06]{MR2272272}
B.~Barak, R.~Impagliazzo, and A.~Wigderson.
\newblock Extracting randomness using few independent sources.
\newblock {\em SIAM J. Comput.}, 36(4):1095--1118 (electronic), 2006.

\bibitem[BKL89]{MR1022771}
L.~Babai, W.~M. Kantor, and A.~Lubotsky.
\newblock Small-diameter {C}ayley graphs for finite simple groups.
\newblock {\em European J. Combin.}, 10(6):507--522, 1989.

\bibitem[BKT04]{MR2053599}
J.~Bourgain, N.~Katz, and T.~Tao.
\newblock A sum-product estimate in finite fields, and applications.
\newblock {\em Geom. Funct. Anal.}, 14(1):27--57, 2004.

\bibitem[BLS87]{BLS87}
L.~Babai, E.~M. Luks, and {\'A}.~Seress.
\newblock Permutation groups in {N}{C}.
\newblock In {\em Proceedings of the {N}ineteenth {A}nnual {ACM} {S}ymposium on
  the {T}heory of {C}omputing}, pages 409--420, New York, 1987. ACM.

\bibitem[BNP08]{BNP}
L.~Babai, N.~Nikolov, and L.~Pyber.
\newblock Product growth and mixing in finite groups.
\newblock In {\em Proc. 19th ACM - SIAM Symp. on Discrete Algorithms
  (SODA'08)}, pages 248--257. ACM-SIAM, 2008.

\bibitem[Bou72]{MR0573068}
N.~Bourbaki.
\newblock {\em \'{E}l\'ements de math\'ematique. {F}asc. {XXXVII}. {G}roupes et
  alg\`ebres de {L}ie. {C}hapitre {II}: {A}lg\`ebres de {L}ie libres.
  {C}hapitre {III}: {G}roupes de {L}ie}.
\newblock Hermann, Paris, 1972.
\newblock Actualit{\'e}s Scientifiques et Industrielles, No. 1349.

\bibitem[BRD15]{BRD}
J.~Button and C.~Roney-Dougal.
\newblock An explicit upper bound for the {H}elfgott delta in ${\rm SL}(2,p)$.
\newblock {\em J. Algebra.}, 421:493--511, 2015.

\bibitem[Bro86]{MR840402}
R.~Brooks.
\newblock The spectral geometry of a tower of coverings.
\newblock {\em J. Differential Geom.}, 23(1):97--107, 1986.

\bibitem[Bro87]{MR894565}
R.~Brooks.
\newblock On the angles between certain arithmetically defined subspaces of
  {${\bf C}^n$}.
\newblock {\em Ann. Inst. Fourier (Grenoble)}, 37(1):175--185, 1987.

\bibitem[BS87a]{MR894827}
L.~Babai and {\'A}.~Seress.
\newblock On the degree of transitivity of permutation groups: a short proof.
\newblock {\em J. Combin. Theory Ser. A}, 45(2):310--315, 1987.

\bibitem[BS87b]{BroSha}
A.~Broder and E.~Shamir.
\newblock On the second eigenvalue of random regular graphs.
\newblock In {\em 28th {A}nnual {S}ymposium on the {F}oundations of {C}omputer
  {S}cience ({FOCS} 1987)}, 1987.

\bibitem[BS88]{BS88}
L.~Babai and {\'A}.~Seress.
\newblock On the diameter of {C}ayley graphs of the symmetric group.
\newblock {\em J. Combin. Theory Ser. A}, 49(1):175--179, 1988.

\bibitem[BS92]{BS92}
L.~Babai and {\'A}.~Seress.
\newblock On the diameter of permutation groups.
\newblock {\em European J. Combin.}, 13(4):231--243, 1992.

\bibitem[BS94]{MR1305895}
A.~Balog and E.~Szemer{\'e}di.
\newblock A statistical theorem of set addition.
\newblock {\em Combinatorica}, 14(3):263--268, 1994.

\bibitem[Bur86]{Burger}
M.~Burger.
\newblock {\em Petites valeurs propres du Laplacien et topologie de Fell}.
\newblock PhD thesis, Universit\'{e} de Lausanne, 1986.

\bibitem[Bus78]{MR505920}
P.~Buser.
\newblock Cubic graphs and the first eigenvalue of a {R}iemann surface.
\newblock {\em Math. Z.}, 162(1):87--99, 1978.

\bibitem[BV12]{MR2897695}
J.~Bourgain and P.~P. Varj{\'u}.
\newblock Expansion in {$SL_d({\bf Z}/q{\bf Z}),\,q$} arbitrary.
\newblock {\em Invent. Math.}, 188(1):151--173, 2012.

\bibitem[Cha02]{MR1909605}
M.-Ch. Chang.
\newblock A polynomial bound in {F}reiman's theorem.
\newblock {\em Duke Math. J.}, 113(3):399--419, 2002.

\bibitem[Cha08]{MR2398145}
M.-Ch. Chang.
\newblock Product theorems in {${\rm SL}_2$} and {${\rm SL}_3$}.
\newblock {\em J. Inst. Math. Jussieu}, 7(1):1--25, 2008.

\bibitem[CS10]{MR2738997}
E.~Croot and O.~Sisask.
\newblock A probabilistic technique for finding almost-periods of convolutions.
\newblock {\em Geom. Funct. Anal.}, 20(6):1367--1396, 2010.

\bibitem[Din06]{MR2231895}
O.~Dinai.
\newblock Poly-log diameter bounds for some families of finite groups.
\newblock {\em Proc. Amer. Math. Soc.}, 134(11):3137--3142 (electronic), 2006.

\bibitem[Din11]{MR2788087}
O.~Dinai.
\newblock Growth in {${\rm SL}_2$} over finite fields.
\newblock {\em J. Group Theory}, 14(2):273--297, 2011.

\bibitem[Dix69]{dixon}
J.~D. Dixon.
\newblock The probability of generating the symmetric group.
\newblock {\em Math. Z.}, 110:199--205, 1969.

\bibitem[DM96]{DM}
J.~D. Dixon and B.~Mortimer.
\newblock {\em Permutation groups}, volume 163 of {\em Graduate Texts in
  Mathematics}.
\newblock Springer-Verlag, New York, 1996.

\bibitem[DS81]{MR626813}
P.~Diaconis and M.~Shahshahani.
\newblock Generating a random permutation with random transpositions.
\newblock {\em Z. Wahrsch. Verw. Gebiete}, 57(2):159--179, 1981.

\bibitem[DS98]{MR1658464}
V.~I. Danilov and V.~V. Shokurov.
\newblock {\em Algebraic curves, algebraic manifolds and schemes}.
\newblock Springer-Verlag, Berlin, 1998.
\newblock Reprint of the original English edition from the series Encyclopaedia
  of Mathematical Sciences [{\em Algebraic geometry. I}, Encyclopaedia Math.
  Sci., 23, Springer, Berlin, 1994; MR1287418 (95b:14001)].

\bibitem[DSC93]{MR1245303}
P.~Diaconis and L.~Saloff-Coste.
\newblock Comparison techniques for random walk on finite groups.
\newblock {\em Ann. Probab.}, 21(4):2131--2156, 1993.

\bibitem[DSV03]{MR1989434}
G.~Davidoff, P.~Sarnak, and A.~Valette.
\newblock {\em Elementary number theory, group theory, and {R}amanujan graphs},
  volume~55 of {\em London Mathematical Society Student Texts}.
\newblock Cambridge University Press, Cambridge, 2003.

\bibitem[EHK12]{MR2922374}
J.~S. Ellenberg, Ch. Hall, and E.~Kowalski.
\newblock Expander graphs, gonality, and variation of {G}alois representations.
\newblock {\em Duke Math. J.}, 161(7):1233--1275, 2012.

\bibitem[EK01]{MR1869409}
Gy. Elekes and Z.~Kir{\'a}ly.
\newblock On the combinatorics of projective mappings.
\newblock {\em J. Algebraic Combin.}, 14(3):183--197, 2001.

\bibitem[Ele97]{MR1472816}
Gy. Elekes.
\newblock On the number of sums and products.
\newblock {\em Acta Arith.}, 81(4):365--367, 1997.

\bibitem[EM03]{MR1948103}
G.~A. Edgar and Ch. Miller.
\newblock Borel subrings of the reals.
\newblock {\em Proc. Amer. Math. Soc.}, 131(4):1121--1129 (electronic), 2003.

\bibitem[EMO05]{MR2129706}
A.~Eskin, Sh. Mozes, and H.~Oh.
\newblock On uniform exponential growth for linear groups.
\newblock {\em Invent. Math.}, 160(1):1--30, 2005.

\bibitem[ES83]{MR820223}
P.~Erd{\H{o}}s and E.~Szemer{\'e}di.
\newblock On sums and products of integers.
\newblock In {\em Studies in pure mathematics}, pages 213--218. Birkh\"auser,
  Basel, 1983.

\bibitem[FH91]{fultonharris}
W.~Fulton and J.~Harris.
\newblock {\em Representation theory: a first course}, volume 129 of {\em
  Graduate Texts in Mathematics}.
\newblock Springer-Verlag, New York, 1991.

\bibitem[{Fiz}]{Fizpont}
G.~{Fiz Pontiveros}.
\newblock Sums of dilates in $\mathbb{Z}_p$.
\newblock Preprint. Available at \texttt{arxiv.org:1203.2659}.

\bibitem[FKP10]{MR2587441}
D.~Fisher, N.~H. Katz, and I.~Peng.
\newblock Approximate multiplicative groups in nilpotent {L}ie groups.
\newblock {\em Proc. Amer. Math. Soc.}, 138(5):1575--1580, 2010.

\bibitem[Fre73]{MR0360496}
G.~A. Fre{\u\i}man.
\newblock {\em Foundations of a structural theory of set addition}.
\newblock American Mathematical Society, Providence, R. I., 1973.
\newblock Translated from the Russian, Translations of Mathematical Monographs,
  Vol 37.

\bibitem[Fur77]{MR0498471}
H.~Furstenberg.
\newblock Ergodic behavior of diagonal measures and a theorem of {S}zemer\'edi
  on arithmetic progressions.
\newblock {\em J. Analyse Math.}, 31:204--256, 1977.

\bibitem[Gam02]{MR1900698}
A.~Gamburd.
\newblock On the spectral gap for infinite index ``congruence'' subgroups of
  {${\rm SL}_2(\bold Z)$}.
\newblock {\em Israel J. Math.}, 127:157--200, 2002.

\bibitem[GH14]{GH2}
N.~Gill and H.~Helfgott.
\newblock Growth in solvable subgroups of {${\rm
  GL}_r(\mathbb{Z}/p\mathbb{Z})$}.
\newblock {\em Math. Annalen}, 360(1--2):157--208, 2014.

\bibitem[GH11]{GH1}
N.~Gill and H.~A. Helfgott.
\newblock Growth of small generating sets in {${\rm SL}_n(\Bbb Z/p\Bbb Z)$}.
\newblock {\em Int. Math. Res. Not. IMRN}, (18):4226--4251, 2011.

\bibitem[GHR]{GHR}
N.~Gill, H.~Helfgott, and M.~Rudnev.
\newblock On growth in an abstract plane.
\newblock To appear in {\em Proc. Amer. Math. Soc}. Available as
  \texttt{arxiv:1212.5056}.

\bibitem[GK07]{MR2359478}
A.~A. Glibichuk and S.~V. Konyagin.
\newblock Additive properties of product sets in fields of prime order.
\newblock In {\em Additive combinatorics}, volume~43 of {\em CRM Proc. Lecture
  Notes}, pages 279--286. Amer. Math. Soc., Providence, RI, 2007.

\bibitem[Gow98]{MR1631259}
W.~T. Gowers.
\newblock A new proof of {S}zemer\'edi's theorem for arithmetic progressions of
  length four.
\newblock {\em Geom. Funct. Anal.}, 8(3):529--551, 1998.

\bibitem[Gow01]{MR1844079}
W.~T. Gowers.
\newblock A new proof of {S}zemer\'edi's theorem.
\newblock {\em Geom. Funct. Anal.}, 11(3):465--588, 2001.

\bibitem[Gow08]{MR2410393}
W.~T. Gowers.
\newblock Quasirandom groups.
\newblock {\em Combin. Probab. Comput.}, 17(3):363--387, 2008.

\bibitem[GR05]{MR2166359}
B.~Green and I.~Z. Ruzsa.
\newblock Sum-free sets in abelian groups.
\newblock {\em Israel J. Math.}, 147:157--188, 2005.

\bibitem[Gra]{Gra}
A.~Granville.
\newblock Additive combinatorics.
\newblock {U}npublished notes.

\bibitem[Gro81]{MR623534}
M.~Gromov.
\newblock Groups of polynomial growth and expanding maps.
\newblock {\em Inst. Hautes {\'E}tudes Sci. Publ. Math.}, (53):53--73, 1981.

\bibitem[GS04]{MR2104475}
A.~Gamburd and M.~Shahshahani.
\newblock Uniform diameter bounds for some families of {C}ayley graphs.
\newblock {\em Int. Math. Res. Not.}, (71):3813--3824, 2004.

\bibitem[GT08]{MR2415379}
B.~Green and T.~Tao.
\newblock The primes contain arbitrarily long arithmetic progressions.
\newblock {\em Ann. of Math. (2)}, 167(2):481--547, 2008.

\bibitem[Gui73]{MR0369608}
Y.~Guivarc'h.
\newblock Croissance polynomiale et p\'eriodes des fonctions harmoniques.
\newblock {\em Bull. Soc. Math. France}, 101:333--379, 1973.

\bibitem[GV12]{SGV}
A.~S. Golsefidy and P.~P. Varj{\'u}.
\newblock Expansion in perfect groups.
\newblock {\em Geom. Funct. Anal.}, 22(6):1832--1891, 2012.

\bibitem[Hel08]{Hel08}
H.~A. Helfgott.
\newblock Growth and generation in {${\rm SL}_2(\mathbb{Z}/p\mathbb{Z})$}.
\newblock {\em Ann. of Math. (2)}, 167(2):601--623, 2008.

\bibitem[Hel11]{HeSL3}
H.~A. Helfgott.
\newblock Growth in {${\rm SL}_3(\mathbb{Z}/p\mathbb{Z})$}.
\newblock {\em J. Eur. Math. Soc. (JEMS)}, 13(3):761--851, 2011.

\bibitem[HK05]{MR2150389}
B.~Host and B.~Kra.
\newblock Nonconventional ergodic averages and nilmanifolds.
\newblock {\em Ann. of Math. (2)}, 161(1):397--488, 2005.

\bibitem[HP95]{MR1329903}
E.~Hrushovski and A.~Pillay.
\newblock Definable subgroups of algebraic groups over finite fields.
\newblock {\em J. Reine Angew. Math.}, 462:69--91, 1995.

\bibitem[Hru12]{MR2833482}
E.~Hrushovski.
\newblock Stable group theory and approximate subgroups.
\newblock {\em J. Amer. Math. Soc.}, 25(1):189--243, 2012.

\bibitem[HS14]{MR3152942}
Harald~A. Helfgott and {\'A}kos Seress.
\newblock On the diameter of permutation groups.
\newblock {\em Ann. of Math. (2)}, 179(2):611--658, 2014.

\bibitem[HSZ15]{HSZ}
H.~A. Helfgott, {\'A}.~Seress, and A.~Zuk.
\newblock Random generators of the symmetric group: diameter, mixing time and
  spectral gap.
\newblock  {\em J. Algebra}, 421:349--368, 2015.

\bibitem[HW08]{MR2436141}
E.~Hrushovski and F.~Wagner.
\newblock Counting and dimensions.
\newblock In {\em Model theory with applications to algebra and analysis.
  {V}ol. 2}, volume 350 of {\em London Math. Soc. Lecture Note Ser.}, pages
  161--176. Cambridge Univ. Press, Cambridge, 2008.

\bibitem[Kas07]{MR2342639}
M.~Kassabov.
\newblock Symmetric groups and expander graphs.
\newblock {\em Invent. Math.}, 170(2):327--354, 2007.

\bibitem[Ka{\v{z}}67]{MR0209390}
D.~A. Ka{\v{z}}dan.
\newblock On the connection of the dual space of a group with the structure of
  its closed subgroups.
\newblock {\em Funkcional. Anal. i Prilo\v zen.}, 1:71--74, 1967.

\bibitem[Kes59]{MR0109367}
H.~Kesten.
\newblock Symmetric random walks on groups.
\newblock {\em Trans. Amer. Math. Soc.}, 92:336--354, 1959.

\bibitem[KMS84]{KMS84}
D.~Kornhauser, G.~Miller, and P.~Spirakis.
\newblock Coordinating pebble motion on graphs, the diameter of permutation
  groups, and applications.
\newblock In {\em Proceedings of the 25th IEEE Symposium on Foundations of
  Computer Science}, pages 241--250, Singer Island, FL, 1984. IEEE Computer
  Society Press.

\bibitem[Kon92]{MR1289921}
S.~V. Konyagin.
\newblock Estimates for {G}aussian sums and {W}aring's problem modulo a prime.
\newblock {\em Trudy Mat. Inst. Steklov.}, 198:111--124, 1992.

\bibitem[Kowa]{Konline}
E.~Kowalski.
\newblock Expander graphs.
\newblock Course notes available at
  \url{http://www.math.ethz.ch/~kowalski/expander-graphs.pdf}.

\bibitem[Kowb]{KowBourbaki}
E.~Kowalski.
\newblock Sieve in expansion.
\newblock S{\'e}minaire {B}ourbaki, 63{\`e}me ann{\'e}e, 2010-2011, no. 1028.

\bibitem[Kow13]{MR3144176}
E.~Kowalski.
\newblock Explicit growth and expansion for {${\rm SL}_2$}.
\newblock {\em Int. Math. Res. Not. IMRN}, (24):5645--5708, 2013.

\bibitem[Lar03]{MR1976231}
M.~Larsen.
\newblock Navigating the {C}ayley graph of {${\rm SL}_2(\Bbb F_p)$}.
\newblock {\em Int. Math. Res. Not.}, (27):1465--1471, 2003.

\bibitem[Lor12]{MR2917136}
O.~Lorscheid.
\newblock Algebraic groups over the field with one element.
\newblock {\em Math. Z.}, 271(1-2):117--138, 2012.

\bibitem[Lov96]{MR1395866}
L.~Lov{\'a}sz.
\newblock Random walks on graphs: a survey.
\newblock In {\em Combinatorics, {P}aul {E}rd\H{o}s is eighty, {V}ol.\ 2
  ({K}eszthely, 1993)}, volume~2 of {\em Bolyai Soc. Math. Stud.}, pages
  353--397. J{\'a}nos Bolyai Math. Soc., Budapest, 1996.

\bibitem[LP11]{LP}
M.~J. Larsen and R.~Pink.
\newblock Finite subgroups of algebraic groups.
\newblock {\em J. Amer. Math. Soc.}, 24(4):1105--1158, 2011.

\bibitem[LPS88]{MR963118}
A.~Lubotzky, R.~Phillips, and P.~Sarnak.
\newblock Ramanujan graphs.
\newblock {\em Combinatorica}, 8(3):261--277, 1988.

\bibitem[LPW09]{MR2466937}
D.~A. Levin, Y.~Peres, and E.~L. Wilmer.
\newblock {\em Markov chains and mixing times}.
\newblock American Mathematical Society, Providence, RI, 2009.
\newblock With a chapter by James G. Propp and David B. Wilson.

\bibitem[LR92]{MR1203870}
J.~D. Lafferty and D.~Rockmore.
\newblock Fast {F}ourier analysis for {${\rm SL}_2$} over a finite field and
  related numerical experiments.
\newblock {\em Experiment. Math.}, 1(2):115--139, 1992.

\bibitem[LS74]{MR0360852}
V.~Landazuri and G.~M. Seitz.
\newblock On the minimal degrees of projective representations of the finite
  {C}hevalley groups.
\newblock {\em J. Algebra}, 32:418--443, 1974.

\bibitem[Lub12]{MR2869010}
A.~Lubotzky.
\newblock Expander graphs in pure and applied mathematics.
\newblock {\em Bull. Amer. Math. Soc. (N.S.)}, 49(1):113--162, 2012.

\bibitem[LW54]{MR0065218}
S.~Lang and A.~Weil.
\newblock Number of points of varieties in finite fields.
\newblock {\em Amer. J. Math.}, 76:819--827, 1954.

\bibitem[LW93]{MR1235570}
A.~Lubotzky and B.~Weiss.
\newblock Groups and expanders.
\newblock In {\em Expanding graphs ({P}rinceton, {NJ}, 1992)}, volume~10 of
  {\em DIMACS Ser. Discrete Math. Theoret. Comput. Sci.}, pages 95--109. Amer.
  Math. Soc., Providence, RI, 1993.

\bibitem[McK84]{McK}
P.~McKenzie.
\newblock Permutations of bounded degree generate groups of polynomial
  diameter.
\newblock {\em Inform. Process. Lett.}, 19(5):253--254, 1984.

\bibitem[Mil68]{MR0244899}
J.~Milnor.
\newblock Growth of finitely generated solvable groups.
\newblock {\em J. Differential Geometry}, 2:447--449, 1968.

\bibitem[MV00]{MR1780213}
F.~Martin and A.~Valette.
\newblock Markov operators on the solvable {B}aumslag-{S}olitar groups.
\newblock {\em Experiment. Math.}, 9(2):291--300, 2000.

\bibitem[MVW84]{MR735226}
C.~R. Matthews, L.~N. Vaserstein, and B.~Weisfeiler.
\newblock Congruence properties of {Z}ariski-dense subgroups. {I}.
\newblock {\em Proc. London Math. Soc. (3)}, 48(3):514--532, 1984.

\bibitem[NC00]{MR1796805}
M.~A. Nielsen and I.~L. Chuang.
\newblock {\em Quantum computation and quantum information}.
\newblock Cambridge University Press, Cambridge, 2000.

\bibitem[Nor87]{MR880952}
M.~V. Nori.
\newblock On subgroups of {${\rm GL}_n({\bf F}_p)$}.
\newblock {\em Invent. Math.}, 88(2):257--275, 1987.

\bibitem[NP11]{MR2800484}
N.~Nikolov and L.~Pyber.
\newblock Product decompositions of quasirandom groups and a {J}ordan type
  theorem.
\newblock {\em J. Eur. Math. Soc. (JEMS)}, 13(4):1063--1077, 2011.

\bibitem[Pet12]{MR3063158}
Giorgis Petridis.
\newblock New proofs of {P}l\"unnecke-type estimates for product sets in
  groups.
\newblock {\em Combinatorica}, 32(6):721--733, 2012.

\bibitem[Pl{\"u}70]{MR0266892}
H.~Pl{\"u}nnecke.
\newblock Eine zahlentheoretische {A}nwendung der {G}raphentheorie.
\newblock {\em J. Reine Angew. Math.}, 243:171--183, 1970.

\bibitem[PPSS12]{MR2898694}
Ch.~E. Praeger, L.~Pyber, P.~Spiga, and E.~Szab{\'o}.
\newblock Graphs with automorphism groups admitting composition factors of
  bounded rank.
\newblock {\em Proc. Amer. Math. Soc.}, 140(7):2307--2318, 2012.

\bibitem[PS]{PS}
L.~Pyber and E.~Szab{\'o}.
\newblock Growth in finite simple groups of {L}ie type of bounded rank.
\newblock Submitted. Available at {\texttt{arxiv.org:1005.1881}}.

\bibitem[PS14]{PSSup}
L.~Pyber and E.~Szab\'o.
\newblock Growth in linear groups.
\newblock In {\em Thin groups and superstrong approximation}, volume~61 of
          {\em Math. Sci. Res. Inst. Publ.}, pages 253--268. Cambridge Univ. Press,
          Cambridge, 2014.

\bibitem[Pyb93]{Pyb93}
L.~Pyber.
\newblock On the orders of doubly transitive permutation groups, elementary
  estimates.
\newblock {\em J. Combin. Theory Ser. A}, 62(2):361--366, 1993.

\bibitem[Raz14]{MR3152939}
A.~A. Razborov.
\newblock A product theorem in free groups.
\newblock {\em Ann. of Math. (2)}, 179(2):405--429, 2014.

\bibitem[Rok63]{MR0193206}
V.~A. Rokhlin.
\newblock Generators in ergodic theory.
\newblock {\em Vestnik Leningrad. Univ.}, 18(1):26--32, 1963.

\bibitem[Rot53]{MR0051853}
K.~F. Roth.
\newblock On certain sets of integers.
\newblock {\em J. London Math. Soc.}, 28:104--109, 1953.

\bibitem[RT85]{MR810596}
I.~Z. Ruzsa and S.~Turj{\'a}nyi.
\newblock A note on additive bases of integers.
\newblock {\em Publ. Math. Debrecen}, 32(1-2):101--104, 1985.

\bibitem[Ruz89]{MR2314377}
I.~Z. Ruzsa.
\newblock An application of graph theory to additive number theory.
\newblock {\em Sci. Ser. A Math. Sci. (N.S.)}, 3:97--109, 1989.

\bibitem[Ruz91]{MR1139055}
I.~Z. Ruzsa.
\newblock Arithmetic progressions in sumsets.
\newblock {\em Acta Arith.}, 60(2):191--202, 1991.

\bibitem[Ruz94]{MR1281447}
I.~Z. Ruzsa.
\newblock Generalized arithmetical progressions and sumsets.
\newblock {\em Acta Math. Hungar.}, 65(4):379--388, 1994.

\bibitem[Ruz99]{MR1701207}
I.~Z. Ruzsa.
\newblock An analog of {F}reiman's theorem in groups.
\newblock {\em Ast\'erisque}, (258):xv, 323--326, 1999.
\newblock Structure theory of set addition.

\bibitem[San12]{MR2994508}
T.~Sanders.
\newblock On the {B}ogolyubov-{R}uzsa lemma.
\newblock {\em Anal. PDE}, 5(3):627--655, 2012.

\bibitem[Sel65]{MR0182610}
A.~Selberg.
\newblock On the estimation of {F}ourier coefficients of modular forms.
\newblock In {\em Proc. {S}ympos. {P}ure {M}ath., {V}ol. {VIII}}, pages 1--15.
  Amer. Math. Soc., Providence, R.I., 1965.

\bibitem[Ser03]{MR1970241}
{\'A}.~Seress.
\newblock {\em {P}ermutation {G}roup {A}lgorithms}, volume 152 of {\em
  Cambridge Tracts in Mathematics}.
\newblock Cambridge University Press, Cambridge, 2003.

\bibitem[Sha97]{MR1645694}
Y.~Shalom.
\newblock Expanding graphs and invariant means.
\newblock {\em Combinatorica}, 17(4):555--575, 1997.

\bibitem[Sha99]{MR1691549}
Y.~Shalom.
\newblock Expander graphs and amenable quotients.
\newblock In {\em Emerging applications of number theory ({M}inneapolis, {MN},
  1996)}, volume 109 of {\em IMA Vol. Math. Appl.}, pages 571--581. Springer,
  New York, 1999.

\bibitem[Sim70]{MR0257203}
Ch.~C. Sims.
\newblock Computational methods in the study of permutation groups.
\newblock In {\em Computational {P}roblems in {A}bstract {A}lgebra ({P}roc.
  {C}onf., {O}xford, 1967)}, pages 169--183. Pergamon, Oxford, 1970.

\bibitem[Sim71]{Sim71}
C.~C. Sims.
\newblock Computation with permutation groups.
\newblock In {\em Proc. {S}econd {S}ymposium on {S}ymbolic and {A}lgebraic
  {M}aniupulation}, pages 23--28. {A}{C}{M} {P}ress, New York, NY, 1971.

\bibitem[Sol09]{MR2538014}
J.~Solymosi.
\newblock Bounding multiplicative energy by the sumset.
\newblock {\em Adv. Math.}, 222(2):402--408, 2009.

\bibitem[SP12]{MR2965280}
J.-Ch. Schlage-Puchta.
\newblock Applications of character estimates to statistical problems for the
  symmetric group.
\newblock {\em Combinatorica}, 32(3):309--323, 2012.

\bibitem[Spi12]{MR2876252}
P.~Spiga.
\newblock Two local conditions on the vertex stabiliser of arc-transitive
  graphs and their effect on the {S}ylow subgroups.
\newblock {\em J. Group Theory}, 15(1):23--35, 2012.

\bibitem[Spr86]{MR842444}
T.~A. Springer.
\newblock Conjugacy classes in algebraic groups.
\newblock In {\em Group theory, {B}eijing 1984}, volume 1185 of {\em Lecture
  Notes in Math.}, pages 175--209. Springer, Berlin, 1986.

\bibitem[SSV05]{MR2155059}
B.~Sudakov, E.~Szemer{\'e}di, and V.~H. Vu.
\newblock On a question of {E}rd{\H{o}}s and {M}oser.
\newblock {\em Duke Math. J.}, 129(1):129--155, 2005.

\bibitem[ST83]{MR729791}
E.~Szemer{\'e}di and W.~T. Trotter, Jr.
\newblock Extremal problems in discrete geometry.
\newblock {\em Combinatorica}, 3(3-4):381--392, 1983.

\bibitem[SX91]{SarnakXue}
P.~Sarnak and X.~X. Xue.
\newblock Bounds for multiplicities of automorphic representations.
\newblock {\em Duke Math. J.}, 64(1):207--227, 1991.

\bibitem[Sze]{Szege}
B.~Szegedy.
\newblock On higher order fourier analysis.
\newblock Available as \texttt{arxiv.org:1203.2260}.

\bibitem[Sze69]{MR0245555}
E.~Szemer{\'e}di.
\newblock On sets of integers containing no four elements in arithmetic
  progression.
\newblock {\em Acta Math. Acad. Sci. Hungar.}, 20:89--104, 1969.

\bibitem[Tao]{Tonline}
T.~Tao.
\newblock Course notes.
\newblock Available at \url{http://www.math.ucla.edu/~tao/254b.1.12w/}.

\bibitem[Tao08]{MR2501249}
T.~Tao.
\newblock Product set estimates for non-commutative groups.
\newblock {\em Combinatorica}, 28(5):547--594, 2008.

\bibitem[Tao10]{MR2791295}
T.~Tao.
\newblock Freiman's theorem for solvable groups.
\newblock {\em Contrib. Discrete Math.}, 5(2):137--184, 2010.

\bibitem[Tit57]{MR0108765}
J.~Tits.
\newblock Sur les analogues alg\'ebriques des groupes semi-simples complexes.
\newblock In {\em Colloque d'alg\`ebre sup\'erieure, tenu \`a {B}ruxelles du 19
  au 22 d\'ecembre 1956}, Centre Belge de Recherches Math\'ematiques, pages
  261--289. \'Etablissements Ceuterick, Louvain, 1957.

\bibitem[Tit72]{MR0286898}
J.~Tits.
\newblock Free subgroups in linear groups.
\newblock {\em J. Algebra}, 20:250--270, 1972.

\bibitem[Toi]{Tointon}
M.~Tointon.
\newblock Freiman's theorem in an arbitrary nilpotent group.
\newblock Submitted. Available as \texttt{arxiv:1211.3989}.

\bibitem[TV06]{MR2289012}
T.~Tao and V.~Vu.
\newblock {\em Additive combinatorics}, volume 105 of {\em Cambridge Studies in
  Advanced Mathematics}.
\newblock Cambridge University Press, Cambridge, 2006.

\bibitem[Var]{Varlat}
P.~P. Varj\'u.
\newblock Random walks in compact groups.
\newblock Available as \texttt{arxiv.org:1209.1745}.

\bibitem[Var12]{MR2862040}
P.~P. Varj{\'u}.
\newblock Expansion in {${\rm SL}_d(\mathcal{O}_K/I)$}, {$I$} square-free.
\newblock {\em J. Eur. Math. Soc. (JEMS)}, 14(1):273--305, 2012.

\bibitem[Wei84]{MR763908}
B.~Weisfeiler.
\newblock Strong approximation for {Z}ariski-dense subgroups of semisimple
  algebraic groups.
\newblock {\em Ann. of Math. (2)}, 120(2):271--315, 1984.

\bibitem[Wie64]{MR0183775}
H.~Wielandt.
\newblock {\em Finite permutation groups}.
\newblock Translated from the German by R. Bercov. Academic Press, New York,
  1964.

\bibitem[Wol68]{MR0248688}
J.~A. Wolf.
\newblock Growth of finitely generated solvable groups and curvature of
  {R}iemanniann manifolds.
\newblock {\em J. Differential Geometry}, 2:421--446, 1968.

\bibitem[Zie07]{MR2257397}
T.~Ziegler.
\newblock Universal characteristic factors and {F}urstenberg averages.
\newblock {\em J. Amer. Math. Soc.}, 20(1):53--97 (electronic), 2007.

\end{thebibliography}
\end{document}